\documentclass[10pt]{article}

\usepackage[english]{babel}
\usepackage[T1]{fontenc}
\usepackage[utf8]{inputenc}
\usepackage{amsmath,amsfonts,graphicx,color}
\usepackage{qsymbols,psfrag,color,xcolor,amsthm}
\usepackage{enumerate}

\usepackage{lineno}

\usepackage{ulem}

\def\eps{\bvarepsilon}

\def \bpar#1{\left\{\begin{array}{#1} }
\def \epar { \end{array}\right.}

\newcommand{\EE}{\mathop{\hbox{\rm I\kern-0.17em E}}\nolimits}
\newcommand{\PP}{\mathop{\hbox{\rm I\kern-0.17em P}}\nolimits}

\def\eps{\varepsilon}

\newcounter{c}
\def \bir{\begin{itemize}\compact \setcounter{c}{0}}
\def \eir{\end{itemize}\vspace{-2em}~}

\newcommand{\Co}{\mathcal{C}}

\newcommand{\E}{{\mathbb E}}
\newcommand{\N}{{\mathbb N}}
\newcommand{\Z}{{\mathbb Z}}
\renewcommand{\P}{{\mathbb P}}

\newcommand{\R}{{\mathbb R}}

\newcommand{\ind}{{\bf 1}}

\newcommand{\sgn}{{\rm sgn}}
\newcommand{\Card}{{\rm Card}}
\def\Var{\mbox{Var}}

\def\Gr{\mbox{Gr}}

\newcommand{\adh}{{\rm adh}}

\def \da{\downarrow}
\def \ua{\uparrow}

\def \CBW{{\sf CBW}}

\def \RDWtoc{{\sf RLW}^{(2n), \to\bullet}}
\def \RDWfromc{{\sf RLW}^{(2n),\bullet \to}}
\def \RBWfromc{{\sf RBW}^{\bullet \to }} %centrifugal (fromc = from center)
\def \RBWtoc{{\sf RBW}^{\to \bullet }}  % centripetal (toc= to center)

\def \W{{\bf W}}
\def \Wu{{\bf W}^{\ua}}
\def \Wd{{\bf W}^{\da}}
\def \Wnu{\W^{2n,\ua}}
\def \Wnd{\W^{2n,\da}}

\def \Wnur{{\bf W}^{(2n),\ua}} % rescaled
\def \Wndr{{\bf W}^{(2n),\da}} % rescaled

\def \Cun{{\bf C}^{2n,\ua}}
\def \Cdn{{\bf C}^{2n,\da}}
\def \Cn{{\bf C}^{2n}}

\def \Cu{{\bf C}^{\ua}}
\def \Cd{{\bf C}^{\da}}
\def \C{{\bf C}}

\def \Cyl{{\sf Cyl}}

\def \Cylu{\Cyl^{\ua}}
\def \Cyldn{\Cyl^{\da}_{2n}}
\def \Cylun{\Cyl^{\ua}_{2n}}

\def \Sliced{\Slice^{\da}}
\def \Sliceu{\Slice^{\ua}}

\def \Circle{{\sf Circle}}

\def \cro#1{\llbracket#1\rrbracket}
\def \eref#1{(\ref{#1})}
\def \1{{\bf 1}}%{\mathds{1}}}
\def \l{\left}
\def \r{\right}

\def \ben{\begin{eqnarray}}
\def \een{\end{eqnarray}}
\def \be{\begin{eqnarray*}}
\def \ee{\end{eqnarray*}}

\def \beq{\begin{equation}}
\def \eq{\end{equation}}

\def \ZdnZ{{\mathbb{Z}/2n\mathbb{Z}}}
\def \Nb{{\sf Nb}}

\def \Slice{{\sf Slice}}

\def \Var{{\sf Var}}

\def \dis{\displaystyle}
\def \app#1#2#3#4#5{\begin{array}{rccl}
    #1:&#2&\longrightarrow&#3\\ &#4&\longmapsto&#5
  \end{array}}

\def \rpiz{\mathbb{R}/2\pi\mathbb{Z}}

\def \rpuz{\mathbb{R}/\mathbb{Z}}

\def \eqd{\sur{=}{(d)}}
\def \dd{\xrightarrow[n]{(d)}}

\def\pmi{\overline{\R}}

\newcommand{\compact}{
  \topsep0pt
  \itemsep=0pt
  \partopsep=0pt
  \parsep=0pt
}

\def \sous#1#2{\mathrel{\mathop{\kern 0pt#1}\limits_{#2}}}
\def \sur#1#2{\mathrel{\mathop{\kern 0pt#1}\limits^{#2}}}

\def \Exp{{\sf Exp}}

\def \modu{\,{\sf mod}\,}
\newqsymbol{`E}{\mathbb{E}}
\newqsymbol{`P}{\mathbb{P}}
\newqsymbol{`R}{\mathbb{R}}
\newqsymbol{`e}{\varepsilon}
\newqsymbol{`o}{\omega}

\newtheorem{lem}{Lemma}[section]
\newtheorem{pro}[lem]{Proposition}%[section]
\newtheorem{theo}[lem]{Theorem}%[section]
\newtheorem{cor}[lem]{Corollary}%[section]
\newtheorem{exm}[lem]{Example}
\newtheorem{rem}[lem]{Remark}%[section]

\title{\bf Directed, cylindric and radial Brownian webs}

\author{David Coupier\thanks{Univ. Lille, CNRS, UMR 8524 - Laboratoire P. Painlev\'e, D\'epartement GIS - Polytech'Lille, F-59000 Lille, France; E-mail~: \texttt{david.coupier@math.univ-lille1.fr}}\, \thanks{Univ. de Valenciennes, CNRS, EA 4015 - LAMAV, F-59313 Valenciennes Cedex 9, France},\qquad  Jean-Fran\c{c}ois Marckert\thanks{CNRS UMR 5800 - LaBRI, Univ. Bordeaux, 351 cours de la Libération, 33405 Talence Cedex-France; E-mail~: \texttt{marckert at labri.fr}}, \qquad Viet Chi Tran\thanks{Univ. Lille, CNRS, UMR 8524 - Laboratoire P. Painlev\'e, UFR de Math\'ematiques, F-59000 Lille, France; E-mail~: \texttt{chi.tran@math.univ-lille1.fr}}}

\begin{document}

\maketitle

\vspace{-0.5cm}

\begin{abstract}The Brownian web (BW) is a collection of coalescing Brownian paths $(W_{(x,t)},(x,t) \in `R^2)$ indexed by the plane. It appears in particular
as continuous limit of various discrete models of directed forests of coalescing random walks and navigation schemes. Radial counterparts have been considered but global invariance principles are hard to establish. In this paper, we consider cylindrical forests which in some sense interpolate between the directed and radial forests: we keep the topology of the plane while still taking into account the angular component. We define in this way the cylindric Brownian web (CBW), which is locally similar to the planar BW but has several important macroscopic differences. For example, in the CBW, the coalescence time between two paths admits exponential moments and the CBW as its dual contain each a.s. a unique bi-infinite path. This pair of bi-infinite paths is distributed as a pair of reflected Brownian motions on the cylinder. Projecting the CBW on the radial plane, we obtain a radial Brownian web (RBW), i.e. a family of coalescing paths where under a natural parametrization, the angular coordinate of a trajectory is a Brownian motion. Recasting some of the discrete radial forests of the literature on the cylinder, we propose rescalings of these forests that converge to the CBW, and deduce the global convergence of the corresponding rescaled radial forests to the RBW. In particular, a modification of the radial model proposed in Coletti and Valencia is shown to converge to the CBW.
\\
{\textbf Keywords~:}  Brownian web, navigation algorithm, random spanning forests, weak convergence of stochastic processes.   \\
{\textbf AMS classification~:} Primary  60J05, 60G52, 60J65, 60D05  Secondary 60G57; 60E99.
\end{abstract}
\small
\noindent \textbf{Acknowledgements~:} This work has benefitted from the GdR GeoSto 3477. J.F.M. has been partially funded by ANR GRAAL (ANR-14-CE25-0014) and D.C. by ANR PPPP (ANR-16-CE40-0016). D.C. and V.C.T. acknowledge support from Labex CEMPI (ANR-11-LABX-0007-01). \\
\normalsize

\section{Introduction}

The Brownian web, BW in the sequel, is a much fascinating object introduced in \cite{Arratia,tothwerner,fontes2004}. It is formed by a family of coalescing Brownian trajectories $(W_{x,t},(x,t)\in \mathbb{R}^2)$, roughly speaking starting at each point $(x,t)$ of the plane $\mathbb{R}^2$ (we consider only 2D objects in this paper). For $(x,t)\in \R^2$,
\ben\label{eq:bs}
(W_{x,t}(s),s\geq t)\eqd x+ (B^{(x,t)}_{s-t}, s \geq t)
\een
where $B^{(x,t)}$ is a standard Brownian motion (BM) starting at 0 and indexed by $(x,t)$. The trajectories started from two different points of the time-space $\R^2$ are independent Brownian motions until they intersect and coalesce.
The BW appears as the continuous limit of various discrete models of coalescing random walks and navigation schemes (e.g. \cite{berestyckigarbansen,colettidiasfontes,colettivalle,coupiersahasarkartran,ferrarifonteswu,ferrarilandimthorisson,gangopadhyayroysarkar,newmanravishankarsun,roysahasarkar,vallezuaznabar}).

\par
Recently, radial (2D) counterparts of these discrete directed forests have been considered and naturally, attempts have been carried to obtain invariance principles for these objects and define a ``radial Brownian web'' (RBW; \cite{baccellibordenave,baccellicoupiertran, colettivalencia,FVV,johanssonviklundsolaturner,johanssonviklundsolaturner2,norristurner,norristurner1,norristurner2}). Nevertheless, the rescaling needed in the BW case is somehow incompatible with a ``nice Brownian limit'' in the radial case. For directed forests in the plane, time is accelerated by $n^2$ say, while space is renormalized by $n$, for a scaling parameter $n\to+\infty$. In the radial case, the ``space and time'' parameterizations are related by the fact that the circle perimeter is proportional to its radius. This hence prevents a renormalization with different powers of $n$ (2 and 1 for $n^2$ and $1/n$) unless we consider only local limits.

\par The main idea of this paper is the creation of the cylindric Brownian web (CBW) that allows to involve the angular characteristic of the radial problems, while keeping a geometry close to the plane. The usual BW is indexed by $\R\times \R$, where the first component is the space component. The CBW is an object indexed by the cylinder
\begin{equation}
\Cyl=(\rpuz)\times \R\label{def:Cyl}
 \end{equation}
where the first component $\rpuz$ is the circle. Topologically, $\Cyl$ somehow interpolates between the plane $`R\times `R$ and the plane equipped with the polar coordinate system $(\rpuz)\times `R_+$ suitable to encode a RBW, as we will see.

\par Similarly to \eqref{eq:bs}, we can define the CBW $\W^\ua=(\W^\ua_{(x,t)},\ (x,t)\in \Cyl)$ as the family of coalescing trajectories
\ben\label{eq:bs2}
(\W^\ua_{x,t}(s),s\geq t)\eqd \l(x+ B^{(x,t)}_{s-t},s\geq t\r) \modu1
\een
that is, independent Brownian motions taken modulo $1$ which coalesce upon intersecting. Note that the time will be flowing upwards in the graphical representations of this paper, and hence the notation with the upward arrow. Later, dual objects will be defined with their inner time running downward. Also, to distinguish between planar and cylindrical objects, cylindrical objects will be denoted with bold letters.

In Section \ref{sec:topo}, we recall the topological framework in which the (planar) BW as introduced by Fontes et al. \cite{fontes2004} is defined. Many convergence results on the plane leading to the BW can be turned into convergence results on the cylinder with the CBW as limit since the map
\[
\app{{\sf proj}}{`R}{`R\ /\ \mathbb{Z}}{x}{x\modu 1}%\label{eq:projection}
\]
is quite easy to handle and to understand. We recall some criteria established in the literature that allow to obtain the BW as limit of discrete directed forests. Then, we extend these results to the cylinder for the CBW. We show that the CBW can arise as the limit of a cylindrical lattice web that is the analogous of the coalescing random walks introduced by Arratia \cite{Arratia}. We end the section by showing different ways to project the CBW on the radial plane to obtain radial `Brownian' webs. \par
In Section \ref{sect:infinitebranches}, the properties of the CBW are investigated. We show in particular that there is almost surely (a.s.) a unique bi-infinite branch in the CBW as well as in its dual, which is a main difference with the planar BW. Starting with a discrete lattice and taking the limit, we can characterize the joint distribution of these two infinite branches as the one of a pair of reflected Brownian motions modulo 1, in the spirit of Soucaliuc et al. \cite{STW}. We also prove that the coalescence time between two (or more) branches admits exponential moment, when its expectation in the plane is infinite. All these behaviors are closely related to the topology of the cylinder.
\par In Sections \ref{sect:ConvCBW} {and \ref{sec:RBW}}, we play with the convergences to the BW in the directed plane, to the CBW in the cylinder and to the RBW in the ``radial'' plane. In the plane, several examples of directed forests in addition to the coalescing random walks of Arratia are known to converge to the Brownian webs, for example \cite{ferrarifonteswu,roysahasarkar}. Other radial trees such as the one introduced by Coletti and Valencia \cite{colettivalencia} are known to converge locally to Brownian webs. We consider the corresponding cylindrical forests and show that they converge to the CBW with a proper rescaling. For example, in Section \ref{sec:RBW}, we propose a radial forest similar to the radial forest of \cite{colettivalencia}, built on a sequence of circles on which a Poisson processes are thrown. When carried to the cylinder, this amounts to throwing Poisson processes with different rates on circles of various heights. We show how the rates and heights can be chosen to have the convergence of the cylindrical forest to the CBW, which is carefully established by adapting well-known criteria (e.g. \cite{fontes2004,SSS}) to the cylinder. The convergence for the latter model has its own interest: as the intensity of points increases with the height in the cylinder, the convergence is obtained for the shifted forests. It is classical in these proofs that the key ingredient for checking the convergence criteria amounts in proving estimates for the tail distribution of the coalescence time between two paths. In our case, this is achieved by using the links between planar and cylindrical models, and thanks to the Skorokhod embedding theorem which connects our problem to available estimates for Brownian motions. However we have to use clever stochastic dominations as well to obtain these estimates. Projecting the cylinder on the (radial) plane then provides a radial discrete forests which converges after normalisation to the radial Brownian web. This convergence is a global convergence, whereas only local convergences are considered in \cite{colettivalencia}.

\section{Cylindric and Radial Brownian Web}
\label{sec:topo}

In this Section we introduce the cylindric Brownian web, several natural models of radial Brownian webs together with some related elements of topology, in particular, some convergence criteria. But we start with the definition of the standard BW given in \cite{fontes2004}.

\subsection{The standard Brownian Web}
\label{section:BWTopo}

Following Fontes \& al. \cite{fontes2004} (see also Sun \cite{Sun} and Schertzer et al. \cite{SSS}), we consider the BW as a compact {random } subset of the set of continuous trajectories started from every space-time point of $\pmi^2=[-\infty,\infty]^2$ equipped with the following distance $\rho$
\ben\label{eq:rho}\rho((x_1,t_1),(x_2,t_2))=
\|A(x_1,t_1)-A(x_2,t_2)\|_\infty,\een
where the map $A$ is given by
\ben\label{eq:AAA}
\app{A}{\pmi^2}{[-1,1]^2}{(x,t)}
{(\Phi(x,t),\Psi(t))=\l(\frac{\tanh(x)}{1+|t|},\tanh(t)\r)}.\een

For $t_0\in\overline{\R}$, $C[t_0]$ denotes the set of functions $f$ from $[t_0,+\infty]$ to $\pmi$ such that $\Phi(f(t),t)$ is continuous.
Further, the set of continuous paths started from every space-time points is
\[\Pi=\bigcup_{t_0\in\pmi} C[t_0]\times \{t_0\}.\]
$(f,t_0)\in \Pi$ represents a path starting at $(f(t_0), t_0)$. For $(f,t_0)\in \Pi$, we denote by $\tilde{f}$ the function that coincides with $f$ on $[t_0,+\infty]$  and which is constant equals $f(t_0)$  on $[-\infty,t_0)$. The space $\Pi$ is equipped with the distance $d$ defined by
\[d((f_1,t_1),(f_2,t_2))=\l(\sup_t\l|\Phi(\tilde{f_1}(t),t)-\Phi(\tilde{f_2}(t),t)\r|\r)\vee |\Psi(t_1)-\Psi(t_2)|.\]
The distance depends on the starting points of the two elements of $\Pi$, as well as their global graphs.
Further, the set ${\cal H}$ of compact subsets of $(\Pi,d)$ is equipped with the $d_{\cal H}$ Hausdorff metric (induced by $d$), and ${\cal F}_{\cal H}$, the associated  Borel $\sigma$-field.\\
The BW $W=(W_{x,t},\ (x,t)\in \R^2)$ is a random variable (r.v.) taking its values in $({\cal H},{\cal F}_{\cal H})$. It can be seen as a collection of coalescing Brownian trajectories indexed by $\R^2$. Its distribution is characterized by the following theorem due to Fontes \& al. \cite[Theo. 2.1]{fontes2004}:

\begin{theo}
There exists an $({\cal H},{\cal F}_{\cal H})$-valued r.v. $W$ whose distribution is uniquely determined by the following three properties.
\begin{itemize}
\compact
\item[$(o)$] From any point $(x,t)\in`R^2$, there is a.s. a unique path $W_{x,t}$ from $(x,t)$,
\item[$(i)$] For any $n\geq 1$, any $(x_1,t_1),\dots,(x_n,t_n)$, the $W_{x_i,t_i}$'s are distributed as coalescing standard Brownian motions,
\item[$(ii)$] For any (deterministic) dense countable subset ${\cal D}$ of $`R^2$, a.s., $W$ is the closure in $({\cal H},d_{\cal H})$ of $(W_{x,t},(x,t)\in {\cal D})$.
\end{itemize}
\end{theo}

In the literature, the BW arises as the natural limit for sequences of discrete forests constructed in the plane. Let $\chi$ be a family of trajectories in ${\cal H}$. For $t>0$ and $t_0,a,b\in \mathbb{R}$ with $a<b$, let
\begin{equation}
\label{def:eta}
{\eta}_\chi(t_0,t; a,b) := \Card \big\{ f(t_0+t) \ | \ (f,s) \in \chi, f(t_0)\in[a,b] \big\}
\end{equation}
be the number of distinct points in $\mathbb{R}\times\{t_0+t\}$ that are touched by paths in $\chi$ which also touch some points in $[a,b]\times\{t_0\}$. We also consider the number of distinct points in $[a,b]\times\{t_0+t\}$ which are touched by paths of $\chi$ born before $t_0$:
\begin{equation}
\label{eq:widehat}
\widehat{\eta}_\chi(t_0,t; a,b) := \Card \big\{ f(t_0+t) \in [a,b] \ | \ (f,s) \in \chi,\ s\leq t_0 \big\} ~.
\end{equation}
Th. 6.5. in \cite{SSS} gives a criterion for the convergence in distribution of sequences of r.v. of $({\cal H},{\cal F}_{{\cal H}})$ to the BW, which are variations of the criteria initially proposed by \cite{fontes2004}:

\begin{theo}
\label{theo:2.2}
Let $(\chi^n)_{n\geq 1}$ be a sequence of $({\cal H},{\cal F}_{{\cal H}})$-valued r.v. which a.s. consists of non-crossing paths. If $(\chi^n)_{n\geq 1}$ satisfies conditions (I), and either (B2) or (E) below, then $\chi^n$ converges in distribution to the standard BW.
\begin{itemize}
\compact
\item[(I)] For any dense countable subset $\mathcal{D}$ of $\R^2$ and for any deterministic $y_1,\cdots,y_m\in {\cal D}$, there exists paths $\chi^n_{y_1},\dots \chi^n_{y_m}$ of $\chi^n$ which converge in distribution as $n\to +\infty$ to coalescing Brownian motions started at $y_1,\dots,y_m$.
\item[(B2)] For any $t>0$, as $`e\to 0^+$,
$$
`e^{-1} \limsup_{n\to+\infty} \sup_{(a,t_0)\in `R^2} \P\big( \eta_{\chi^n}(t_0,t;a,a+`e) \geq 3 \big) \to 0 ~.
$$
\item[(E)] For any limiting value $\chi$ of the sequence $(\chi^n)_{n\geq 1}$, and for any $t >0$, $t_0\in\R$, $a<b\in\R$,
$$
\E \big( \widehat{\eta}_\chi(t_0,t ; a,b)\big) \leq \E\big(\widehat{\eta}_W(t_0,t ; a,b)\big) ~,
$$
where $W$ denotes the BW.
\end{itemize}
\end{theo}

In this paper we focus on forests with non-crossing paths. But there also exist in the literature convergence results without this assumption: see Th. 6.2. or 6.3. in \cite{SSS}. For forests with non-crossing paths, the condition $(I)$ implies the tightness of $(\chi^n)_{n\geq 1}$. The conditions $(B2)$ or $(E)$ somehow ensure that the limit does not contain `more paths' than the BW. In the literature, proofs of $(B2)$ and $(E)$ are both based on an estimate of the coalescence time of two given paths. However, condition $(B2)$ is sometimes more difficult to check. It is often verified by applying FKG positive correlation inequality \cite{FKG70}, which turns out to be difficult to verify in some models. When the forest exhibits some Markov properties, it could be easier to check $(E)$ as it is explained in \cite{newmanravishankarsun} or \cite{SSS}, Section 6.1. Let us give some details. The condition $(E)$ mainly follows from
\begin{equation}
\label{comingdowninf}
\limsup_{n\rightarrow +\infty} \E \big( \widehat{\eta}_{\chi^n}(t_0,\varepsilon ; a,b) \big) < +\infty ~,
\end{equation}
for any $\varepsilon>0$, $t_0\in\R$ and $a<b\in\R$, which can be understood as a coming-down from infinity property. Statement (\ref{comingdowninf}) shows that for any limiting value $\chi$,
the set of points $\chi(t_0 ; t_0+\varepsilon)$
 of $\R\times \{t_0+\varepsilon\}$ that are hit by the paths of $\chi(t_0)$ --
 paths of $\chi$ born before time $t_0$ --
 constitutes a locally finite set.  Thus, condition (I) combined with the Markov property, implies that the paths of $\chi$ starting from $\chi(t_0 ; t_0+\varepsilon)$ are distributed as coalescing Brownian motions. Hence,
\begin{eqnarray}
\E \big(\widehat{\eta}_{\chi(t_0)}(t_0,t ; a,b)\big) \leq \E \big(\widehat{\eta}_{W}(t_0+\varepsilon,t-\varepsilon ; a,b)\big) & = & \frac{b-a}{\sqrt{\pi (t-\varepsilon)}} \nonumber\\
& \to & \frac{b-a}{\sqrt{\pi t}} = \E \big(\widehat{\eta}_{W}(t_0,t ; a,b)\big)
\label{limitEetaBW}
\end{eqnarray}
as $\varepsilon\to 0$ and (E) follows. For details about the identity (\ref{limitEetaBW}) see \cite{SSS}.

\subsection{The Cylindric Brownian Web}
\label{section:CylBrownWeb}

We propose to define the CBW $\Wu=(\Wu_{x,t}\, ,\ (x,t)\in \Cyl)$ on a functional space similar to ${\cal H}$ so that the characterizations of the distributions and convergences in the cylinder are direct adaptations of their counterparts in the plane (when these counterparts exist! See discussion in Section \ref{sec:FCP}). In particular, this will ensure that the convergences in the cylinder and in the plane can be deduced from each other provided some conditions on the  corresponding discrete forests are satisfied.

\par The closed cylinder is the compact metric space $\overline{\Cyl}=(\rpuz)\times \overline{\R}$,
for the metric
\begin{equation}
\label{eq:rhoO}
\rho_O((x_1,t_1),(x_2,t_2))= d_{\rpuz}(x_1,x_2)\vee |\Psi(t_1)-\Psi(t_2)|
\end{equation}
where $d_{\rpuz}(x,y)= \min\{ |x-y|, 1-|x-y|\}$ is the usual distance in $\rpuz$. In the sequel, we use as often as possible the same notation for the CBW as for the planar BW, with an additional index $O$ (as for example $\rho$ and $\rho_O$).

For $t_0\in\pmi$, the set $C_O[t_0]$ denotes the set of continuous functions $f$ from $[t_0,+\infty]$ to $\rpuz$, and $\Pi_O$  the set $\bigcup_{t_0\in\pmi} C_O[t_0]\times \{t_0\}$, where $(f,t_0)\in \Pi_O$ represents a path starting at $(f(t_0), t_0)$. For $(f,t_0)\in \Pi_O$, we denote by $\tilde{f}$ the function that coincides with $f$ on $[t_0,+\infty]$  and which equals to $f(t_0)$  on $[-\infty,t_0)$. On $\Pi_O$, define a distance $d_O$ by
\[d_O((f_1,t_1),(f_2,t_2))=\l(\sup_t d_{\rpuz}(\tilde{f_1}(t),\tilde{f_2}(t))\r)\vee |\Psi(t_1)-\Psi(t_2)|.\]
Further, ${\cal H}_O$, the set of compact subsets of $(\Pi_O,d_O)$ is equipped with the $d_{{\cal H}_O}$ Hausdorff metric (induced by $d_O$), and ${\cal F}_{{\cal H}_O}$, the associated  Borel $\sigma$-field.
The CBW is a r.v. taking its values in $({\cal H}_O,{\cal F}_{{\cal H}_O})$, and is characterized by the following theorem (similar to the Theo. 2.1. in Fontes \& al. \cite{fontes2004} for planar BW).

\begin{theo}
\label{th:defCBW}
There is an $({\cal H}_O,{\cal F}_{{\cal H}_O})$-valued r.v. $\Wu$ whose distribution is uniquely determined by the following three properties.
\begin{itemize}
\compact
\item[$(o)$] From any point $(x,t)\in\Cyl$, there is a.s. a unique path $\Wu_{x,t}$ from $(x,t)$,
\item[$(i)$] for any $n\geq 1$, any $(x_1,t_1),\dots,(x_n,t_n)$ the joint distribution of  the $\Wu_{x_i,t_i}$'s is that of coalescing standard Brownian motions modulo $1$,
\item[$(ii)$] for any (deterministic) dense countable subset ${\cal D}$ of $\Cyl$, a.s., $\Wu$ is the closure in $({\cal H}_O,d_{{\cal H}_O})$ of $(W^\ua_{x,t},(x,t)\in {\cal D})$.
\end{itemize}
\end{theo}

As in the planar case, the CBW $\Wu$ admits a dual counterpart, denoted by $\Wd$ and called the \textit{dual Cylindric Brownian Web}. For details (in the planar case) the reader may refer to Section 2.3 in \cite{SSS}. For any $t_0\in\pmi$, identifying each continuous functions $f\in C_O[t_0]$ with its graph as a subset of $\overline{\Cyl}$, $\widehat{f}:=-f=\{(-x,-t) : (x,t)\in f\}$ defines a continuous path running backward in time and starting at time $-t_0$. Following the notations used in the forward context, let us define the set $\widehat{\Pi}_O$ of such backward continuous paths (with all possible starting time), equipped with the metric $\widehat{d}_O$ (the same as $d_O$ but on $\widehat{\Pi}_O$). Further, $\widehat{{\cal H}}_O$ denotes the set of compact subsets of $(\widehat{\Pi}_O,\widehat{d}_O)$ equipped with the Hausdorff metric induced by $\widehat{d}_O$. Theorem 2.4 of \cite{SSS} admits the following cylindric version.

\begin{theo}
\label{theo:caract_CBW}
There exists an ${\cal H}_O\times\widehat{{\cal H}}_O$ valued r.v. $(\Wu,\Wd)$ called the \textit{double Cylindric Brownian Web}, whose distribution is uniquely determined by the two following properties:
\begin{itemize}
\item[(a)] $\Wu$ and $-\Wd$ are both distributed as the CBW.
\item[(b)] A.s. no path of $\Wu$ crosses any path of $\Wd$.
\end{itemize}
\end{theo}
Moreover, the dual CBW $\Wd$ is a.s. determined by $\Wu$ (and vice versa) since for any point $(x,t)\in\Cyl$, $\Wd$ a.s. contains a single (backward) path starting from $(x,t)$ which is the unique path in $\widehat{\Pi}_O$ that does not cross any path in $\Wu$.\par

For all $-\infty\leq t\leq t'<+\infty$, let us denote by ${\cal F}^{\ua}_{t,t'}$ the $\sigma-$algebra generated by the CBW $\Wu$ between time $t$ and $t'$:
\begin{equation}
\label{tribu:t-t'}
\mathcal{F}^\ua_{t,t'} = \sigma \left( \left\{ \left\{ \W^\ua_{(x,s)}(s') , \, t<s'\leq t' \right\} , \, x\in\rpuz , \, t<s\leq t' \right\} \right) ~.
\end{equation}We write $\mathcal{F}^\ua_{t'}$ instead of $\mathcal{F}^\ua_{-\infty,t'}$. The CBW is Markov with respect to the filtration $(\mathcal{F}^\ua_{t})_{t\in`R}$ and satisfies the strong Markov property, meaning that for any stopping time $T$ a.s. finite, the process
$$
\left\{ \left\{ \W^\ua_{(x,T+t)}(T+s) , \, s\geq t \right\} , \, x\in\rpuz , \, t\geq 0 \right\}
$$
is still a CBW restricted to the semi-cylinder $\Cyl^{+}:=(\rpuz)\times `R^{+}$ which is independent of $\cap_{t>T}\mathcal{F}^\ua_{t}$. In the same way, we can also define the $\sigma-$algebra ${\cal F}^{\da}_{t,t'}$, where $t\geq t'$, with respect to the dual CBW $\Wd$.\par

The convergence criteria \cite[Th. 6.5]{SSS} or Theorem \ref{theo:2.2} above has hence a natural counterpart on the cylinder. For $a,b \in \rpuz$  denote by $[a\to b]$ the interval from $a$ to $b$ when turning around the circle counterclockwise, and by  $|a\to b|$ its Lebesgue measure (formally: for $a<b$, $[a\to b] = [a,b]$ and if $a>b$, $[a\to b]=[a,1]\cup[0,b]$). For $X$ a r.v. in ${\cal H}_O$, denote by
$$
\eta_X^O(t_0,t;[a\to b]) := \Card \big\{ f(t_0+t)\ | \ (f,s)\in X, f(t_0) \in [a\to b] \}
$$
be the number of distinct points in $\rpuz\times\{t_0+t\}$ that are touched by paths in $X$ which also touch some points in $[a\to b]\times\{t_0\}$. We also set
$$
\widehat{\eta}_X^O(t_0,t ; [a\to b]) := \Card \big\{ f(t_0+t) \in [a\to b] \ |\ (f,s)\in X,\ s\leq t_0 \big\} ~.
$$
Here is the counterpart of Theorem \ref{theo:2.2}  in the cylinder:
{\begin{theo}
\label{theo:conv_cyl}
Let $(\chi^n)_{n\geq 1}$ be a sequence of $({\cal H}_O,{\cal F}_{{\cal H}_O})$-valued r.v. which a.s. consist of non-crossing paths. If $(\chi^n)_{n\geq 1}$ satisfies conditions (IO), and either (B2O) or (EO), then $\chi^n$ converges in distribution to the CBW $\Wu$.
\begin{itemize}
\compact
\item[(IO)] For any dense countable subset $\mathcal{D}$ any deterministic $y_1,\cdots,y_m\in {\cal D}$, there exists for every $n\geq 1$, paths $\chi^n_{y_1}\dots \chi^n_{y_m}$ in $\chi^n$ such that $\chi^n_{y_1}\dots \chi^n_{y_m}$ converge in distribution as $n\to +\infty$ to coalescing Brownian motions modulo $1$ started at $y_1,\dots,y_m$.
\item[(B2O)] For any $t>0$, as $`e\to 0^+$,
$$
`e^{-1} \limsup_{n\to+\infty} \sup_{(a,t_0)\in \Cyl} \P \big( \eta^O_{\chi^n}(t_0,t;[a\to a+`e \modu1]) \geq 3 \big) \to 0 ~.
$$
\item[(EO)] For any limiting value $\chi$ of the sequence $(\chi^n)_{n\geq 1}$, and for any $t >0$, $t_0\in\R$ and $a,b\in\rpuz$,
$$
\E \big( \widehat{\eta}^{O}_\chi(t_0,t ; [a\to b]) \big) \leq \E \big( \widehat{\eta}^{O}_{\Wu}(t_0,t ; [a \to b]) \big) ~.
$$
\end{itemize}
\end{theo}}

{This section ends with a summary of the relationships between $\eta_{W}$, $\widehat{\eta}_{W}$, $\eta^O_{\Wu}$ and $\widehat{\eta}^O_{\Wu}$ where $W$ denotes the planar BW. First, in the plane, as noticed in \cite{fontes2004} Section 2, $\eta_{W}(t_0,t ; a,b)$ and $\widehat{\eta}_{W}(t_0,t ; a,b)+1$ are identically distributed. This can be shown using duality arguments. In the cylinder the situation is a little bit different: it is not difficult to show that, for $t,t_0>0$ and $a,b\in\rpuz$,
$$
\eta^O_{\Wu}(t_0,t ; [a \to b]) \eqd \widehat{\eta}^O_{\Wu}(t_0,t ; [a \to b]) + \ind_{\mbox{NoBackCoal}} ~,
$$
where the event $\mbox{NoBackCoal}$ means that the cylindric BMs starting from $(a,t_0)$ and $(b,t_0)$ are allowed to coalesce before time $t_0+t$ but not from the side $[b \to a]$ ({more precisely, $|\Wu_{(a,t_0)}(s)\to \Wu_{(b,t_0)}(s)|$ stays in $[0,1)$ for $s\in[t_0,t_0+t]$}).

{Moreover, for any $t,t_0>0$ and $a,b\in\rpuz$ with $|a \to b|<1$, we will prove at the end of the current section that
\begin{equation}
\label{eta^O<eta}
\eta^O_{\Wu}(t_0,t ; [a \to b]) \leq_{S} \eta_{W}(t_0,t ; a,b) ~,
\end{equation}
where $\leq_{S}$ stands for the stochastic domination. Statement (\ref{eta^O<eta}) traduces the following natural principle: trajectories merge easier in the cylinder than in the plane. However there is no stochastic comparison between $\widehat{\eta}^O_{\Wu}$ and $\widehat{\eta}_{W}$. Indeed, the expectation of $\widehat{\eta}_{W}(t_0,t ; a,b)$ tends to $0$ as $t\to\infty$ thanks to identity (\ref{limitEetaBW}) whereas this does not hold in the cylinder. Theorem \ref{theo:BiInfB} (below) states the a.s. existence in $\Wu$ of a bi-infinite path. So, for any $t,t_0$, $\widehat{\eta}^O_{\Wu}(t_0,t ; [0 \to 1])$ is larger than $1$ and, by rotational invariance,
$$
\E \big(\widehat{\eta}^O_{\Wu}(t_0,t ; [a \to b])\big) = |a\to b| \; \E \big(\widehat{\eta}^O_{\Wu}(t_0,t ; [0 \to 1])\big) \geq |a\to b| ~.
$$}
It then remains to prove \eref{eta^O<eta}. Let us focus on the planar BW $W$ restricted to the strip $\R\times[t_0,t_0+t]$. First, by continuity of trajectories, with probability at least $1-\varepsilon$, there exists $\delta>0$ such that $\sup_{0\leq d\leq \delta}|W_{a,t_0}(t_0+d)- W_{b,t_0}(t_0+d)|<1$ (since $|a-b|<1$) where $W_{x,t}$ denotes the BM starting at $(x,t)$. The coming-down from infinity property satisfied by the BW ensures that the number of remaining BMs at level $\R\times\{t_0+\delta\}$ and starting from $[a,b]\times\{t_0\}$ is a.s. finite. Let $\kappa$ be this (random) number. When defining a realization of the BW, we need to decide, in case of coalescence of two trajectories, which one survives. In order to compute $\eta_{W}(t_0,t ; a,b)$ we label these remaining BMs by $1,\ldots,\kappa$ from left to right and when two of them merge, the BM having the lower label is conserved while the other one is stopped. This stopping rule allows us to determine the set of labels of remaining BMs at level $\R\times\{t_0+t\}$, say $\mathcal{L}$, whose cardinality is $\eta_{W}(t_0,t ; a,b)$. Now, let us complete the previous stopping rule as follows: if the BM with label $2\leq j\leq\kappa$ meets the path $1+W_{a,t_0}$ between times $t_0+\delta$ and $t_0+t$ then it stops. Although $1+W_{a,t_0}$ does not correspond to any trajectory in the planar BW $W$-- and then appears as artificial --, it coincides with $W_{a,t_0}$ in the cylinder and then has label $1$. According to this completed rule, we obtain a new set of labels of remaining BMs at level $\R\times\{t_0+t\}$. It is included in $\mathcal{L}$ and its cardinality has the same distribution than $\eta^O_{\Wu}(t_0,t ; [a \to b])$. In conclusion the previous construction allows us to bound from above $\eta^O_{\Wu}(t_0,t ; [a \to b])$ by $\eta_{W}(t_0,t ; a,b)$ on an event of probability at least $1-\varepsilon$, for any $\varepsilon>0$.

\subsection{The Cylindric Lattice Web}
\label{section:CylLatticeWeb}

As for the BW, the CBW can be constructed as the limit of a sequence of discrete directed forests on the cylinder.
For any integer $n\geq 1$, define the ``cylindric lattice'' as~:
\[\Cylun= \{(x,t),~ x \in \mathbb{Z}/2n\mathbb{Z},~ t \in \mathbb{Z},~ x-t \modu2 = 0 \},\]
and consider $(\xi(w), w \in \Cylun)$ a collection of i.i.d. Rademacher r.v. associated with the vertices of $\Cylun$. The cylindric lattice web (CLW) is the collection of random walks
\[\Wnu=\l(\Wnu_{w}, \ w \in \Cylun\r)\]
indexed by the vertices of $\Cylun$, where for $w=(x,t)$,
\ben
\label{eq:S}
\bpar{lcl}
\Wnu_{(x,t)}(t)&=&x\\
\Wnu_{(x,t)}(s)&=&\Wnu_{(x,t)}(s-1) + \xi(\Wnu_{(x,t)}(s-1),s-1) \modu 2n,~~ s >t.
\epar\een
The sequence of paths $(\Wnu_w,w\in \Cylun)$ is equivalent to that introduced by Arratia \cite{Arratia} in the planar case. The union of the random paths $((\Wnu_{(x,t)}(s),s),s\geq t)$ for $(x,t)\in \Cylun$, coincides with the set of edges $\{(w,w+(\xi(w),1)), w \in \Cylun\}$ (see Figure \ref{fig:Cyln}).
\begin{figure}[!ht]
\begin{center}
\begin{tabular}{cc}
\includegraphics[height=3.2cm]{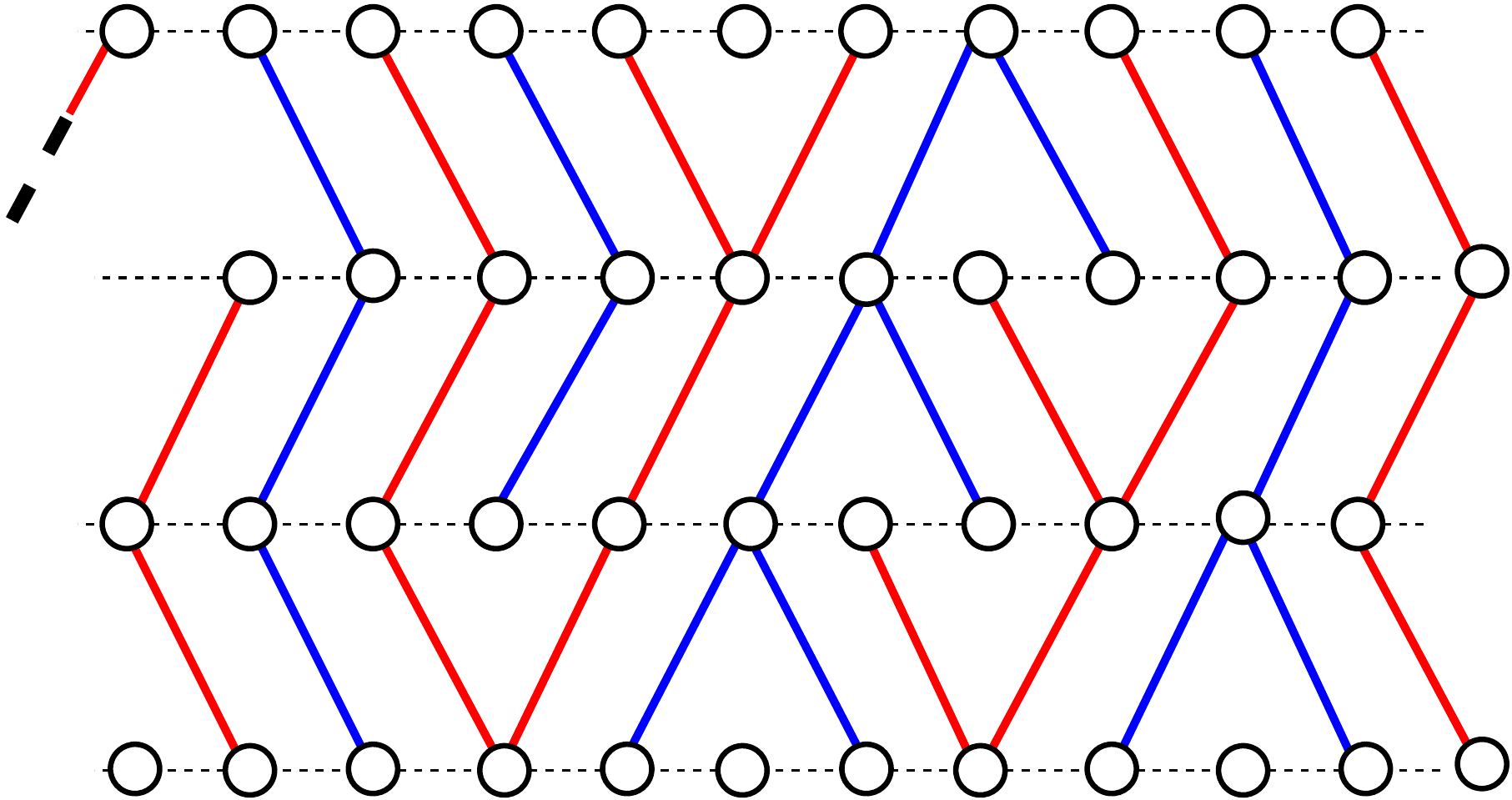} & \includegraphics[height=5.4cm,trim=0cm 4cm 0cm 0cm]{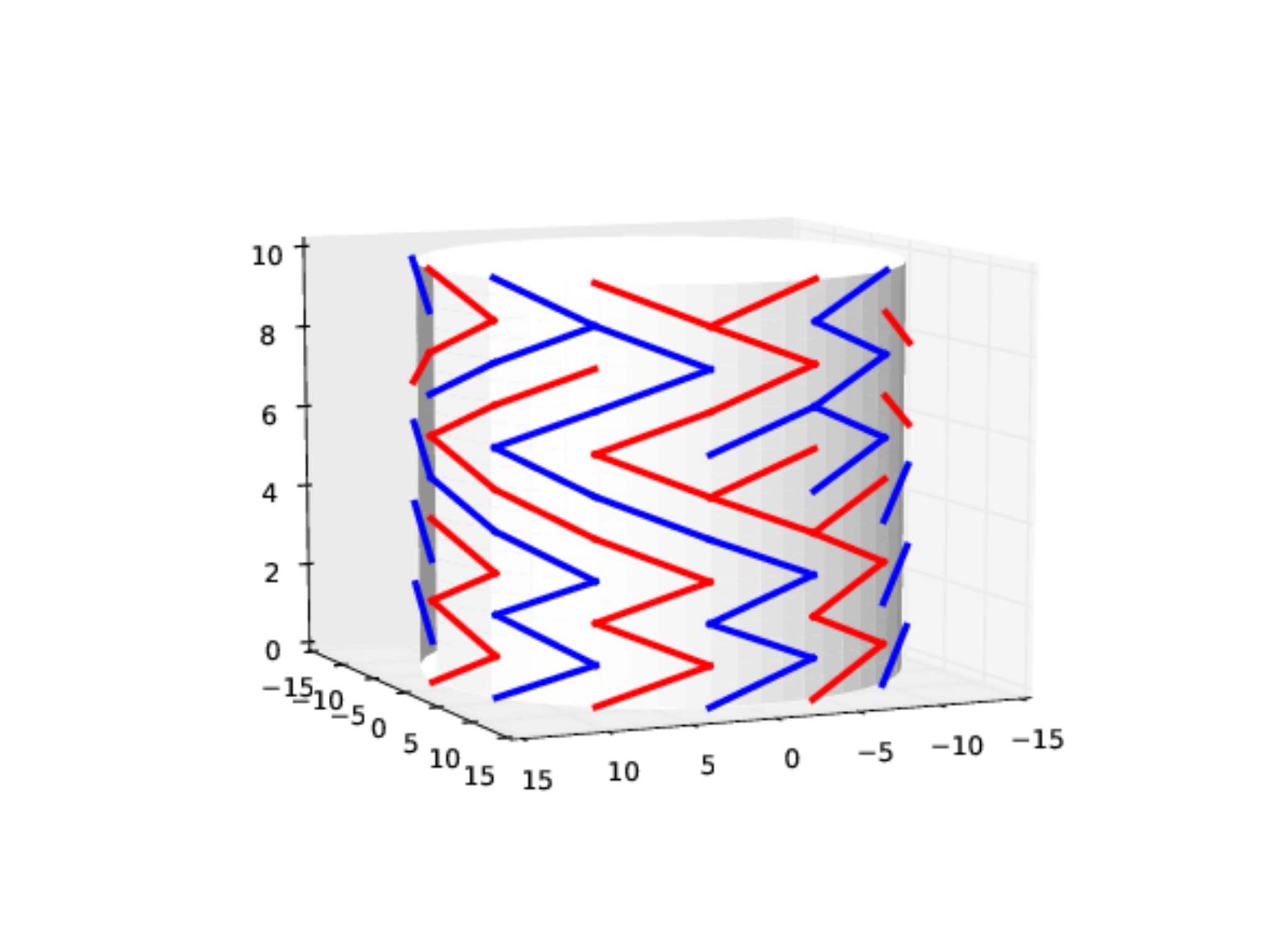}
\end{tabular}
\end{center}
\caption{\label{fig:Cyln} {\small \textit{Standard and cylindric lattice webs: the primal and dual ones are respectively in blue and red.}}}
\end{figure}
The dual $\Wnd$ of $\Wnu$ is a reversed time CLW (and shifted by 1) defined on the ``dual'' $\Cyldn$ of $\Cylun$~:
\[\Cyldn= \{(x,t), x \in \ZdnZ, t \in \mathbb{Z}, x-t \modu 2 = 1 \}.\]
$\Wnd$ is the collection of random walks $\Wnd=\l(\Wnd_{w},\ w \in \Cyldn\r)$ indexed by the vertices of $\Cyldn$ such that for $w=(x,t)\in \Cyldn$, and using the same family $(\xi(w),\ w \in \Cylun)$ as before:
\ben\label{eq:Sp}
\bpar{lcl}
 \Wnd_{(x,t)}(t)&=&x,\\
 \Wnd_{(x,t)}(s)&=&\Wnd_{(x,t)}(s+1)-\xi\big(\Wnd_{(x,t)}(s+1)-(0,1)),s\big),       \textrm{ for } s \leq t.
\epar \een
We define, for any $h\in \mathbb{Z}$, for any direction $D\in \{\ua,\da\}$, the horizontal slice by
\[\Slice_{2n}^D(h)= \Cyl^D_{2n} \cap \l( \ZdnZ\times\{h\}\r), \]
so that the random walks $(\W_w^{2n,D},w \in \Slice_{2n}^D(h))$ start from the points of $\Slice_{2n}^{D}(h)$.\\

The normalized CLW and its dual are defined as follows. For $D\in \{\ua , \da\}$ and for any $(x,t)$ in $\Cyl_{2n}^D$, set
\ben
\W^{(2n),D}_{(\frac{x}{2n},\frac{t}{n^2})}(s):= \frac{1}{2n} \W_{(x,t)}^{2n, D}\l( {4n^2 s }\r) ~~ \textrm{ for }s\geq  \frac{t}{n^2} \mbox{ if }D=\ua,\mbox{ and }s\leq \frac{t}{n^2}\mbox{ if }D=\da.\label{def:CLW}
\een
Since $\W^{(2n),D}_{(x,t)}(4 n^2 s )$ takes its values in $\ZdnZ$, $2n$ is the right space normalization, which implies the time normalization as usual.

\begin{pro}
\label{pro:main-lattice}
The pair of renormalized CLW $(\W^{(2n),\ua},\W^{(2n),\da})$ converges in distribution to the pair of CBW $(\Wu,\Wd)$.
\end{pro}

\begin{proof}
Let us first prove the convergence of the marginals. Since $\Wnu$ and $\Wnd$ have the same distribution (up to a reversal of time and a shift by $1$).

To do it, we mainly refer to the proof of the convergence towards the (planar) BW of the sequence of lattice webs $(W^{(2n)})_{n\geq 1}$, obtained from normalizing the random walks on the grid $\Gr=\{(x,t)\in \Z^2,\, x-t \mod 2 =0 \}$ similarly to \eqref{def:CLW}: see \cite[Section 6]{fontes2004} for further details. As for $(I)$ the proof of $(IO)$ is a basic consequence of the Donsker invariance principle and is omitted here. {The same coupling argument used to prove \eref{eta^O<eta} leads to the following stochastic domination: for $n\geq 1$, $t_0\in \R$, $t>0$, $a\in [0,2\pi]$ and $\varepsilon>0$,
\begin{equation}
\label{eq:domi-eta}
\eta^O_{\W^{(2n),\ua}}(t_0,t ; [a\to  a+\varepsilon]) \leq_S \eta_{W^{(2n)}}(t_0,t ; a, a+\varepsilon) ~.
\end{equation}
Hence condition $(B2)$ satisfied by the rescaled (planar) lattice web $W^{(2n)}$ (see Section 6 in \cite{fontes2004}) implies condition $(B2O)$ for $\W^{(2n),\ua}$. Then Theorem \ref{theo:conv_cyl} applies and gives the convergence of $(\W^{(2n),\ua})_{n\geq 1}$ to $\Wu$.}

{The convergence of the marginals implies that the distributions of $\{(\Wnur,\Wndr)\}_{n\geq 1}$ form a tight sequence in the set of measures on $\mathcal{H}_{O}\times\widehat{\mathcal{H}}_{O}$. It then suffices to prove that any limiting value of this sequence, say $(\mathcal{X}^\ua,\mathcal{X}^\da)$, is distributed as the double CBW $(\Wu,\Wd)$. To do it, we check the criteria of Theorem \ref{theo:caract_CBW}. Item $(a)$ has already been proved. To check $(b)$, let us assume by contradiction that with positive probability there exists a path $\pi_{z}\in\mathcal{X}^\ua$ which crosses a path $\hat{\pi}_{\hat{z}}\in\mathcal{X}^\da$.

By definition of $(\mathcal{X}^\ua,\mathcal{X}^\da)$, this would lead to the existence, for $n$ large enough and with positive probability, of a path of $\Wnur$ crossing a path of $\Wndr$.
This is forbidden since the lattice webs have non crossing paths.}
\end{proof}

\subsection{Radial Brownian Webs}

\subsubsection{The standard Radial Brownian Web and its dual}
\label{section:foretsradiales}

Our goal is now to define a family of coalescing paths, indexed by the distances of their starting points to the origin in $\mathbb{R}^2$, that we will call {\it radial Brownian web}. Let us start with some topological considerations.
Our strategy consists in sending the semi-cylinder $\Cyl^{+}:=(\rpuz)\times `R^{+}$ onto the plane equipped with the polar coordinate system $(\rpiz)\times `R^+$ by using the
map
\begin{equation}
\label{def:phi}
\app{\varphi_\star}{\rpuz \times `R^{+}}{(\rpiz) \times `R^{+}}{(x, f_\star(t))}{(2\pi x, t)} ~,
\end{equation}
where $f_\star(t):=t/(4\pi^{2})$. The presence of factor $1/(4\pi^{2})$ will be discussed below. Let
\[\Slice(h)=\{ (x,h), x \in \rpuz\},\]
be the horizontal slice at height $h$ of $\Cyl$. For any $t>0$, $\varphi_\star$ projects $\Slice(f_\star(t))$ on $\Circle(0,t):=\rpiz \times \{t\}$. It also identifies $\Slice(0)$ with the origin.

The map $\varphi_\star$ induces the metric $\rho_{\bullet}$ on the radial plane ${\rpiz \times `R^{+}}$  by
$$
\rho_{\bullet}((x_1,t_1),(x_2,t_2)) := \rho_{O}(\varphi_\star^{-1}(x_1,t_1),\varphi_\star^{-1}(x_2,t_2)),
$$
for any elements $(x_1,t_1),(x_2,t_2)\in (\rpiz)\times `R^+$. Following the beginning of Section \ref{section:CylBrownWeb}, we can construct a measurable space $({\cal H}_{\bullet},{\cal F}_{{\cal H}_{\bullet}})$ equipped with the distance $\rho_{\bullet}$.
Of course, the map $\varphi_\star$ is continuous for the induced topology, so that the image of a (weak) converging sequence by $\varphi_\star$ is a (weak) converging sequence.
We call  \textit{standard in-radial Brownian web}, and denote by $\RBWtoc$, the image under $\varphi_\star$ of the dual CBW $\Wd$ restricted to $\Cyl^{+}$. In particular, $\varphi_\star$ sends the trajectory $\Wd_{x,f_\star(t)}(s)$ for $s$ going from $f_\star(t)$ to 0 on the path $\RBWtoc_{2\pi x,t}(s)$ for $s$ going from $t$ to 0 where
\begin{equation}
\label{def:RBWtoc}
\RBWtoc_{2\pi x,t}(s) := s \exp \l( 2i\pi \Wd_{x,f_\star(t)}(f_\star(s)) \r) ~.
\end{equation}
Notice that the natural time of the trajectory $\RBWtoc_{2\pi x,t}$ is given by
the distance to the origin, since the radius satisfies:
$$
| \RBWtoc_{2\pi x,t}(s) | = s ~.
$$
The families of paths $(\Wd_{x,f_\star(t)}, (x,f_\star(t)) \in \Cyl^+)$ that coalesce on the cylinder when $t$ evolves from $+\infty$ to 0, are then sent on radial paths  $(\RBWtoc_{x,t},(t \exp(i x) \in \mathbb{C}))$ that coalesce when they are approaching the origin 0.  This is the reason why $\RBWtoc$ is said in-radial, and the notation $\to\bullet$ evokes the direction of the paths, ``coalescing towards the origin''.

Moreover, for any $1<s\leq t$, $\varphi_\star$ sends the part of cylinder delimited by times $f_\star(s-1)=(s-1)/(4\pi^{2})$ and $f_\star(s)=s/(4\pi^{2})$ (i.e. with height $1/(4\pi^{2})$) to the ring centered at the origin and delimited by radii $s-1$ and $s$ (i.e. with width $1$). Then, on the unit time interval $[s-1;s]$, the increment of the argument of $\RBWtoc_{2\pi x,t}$, i.e.
$$
2\pi \Wu_{x,f_\star(t)}(f_\star(s-1)) - 2\pi \Wu_{x,f_\star(t)}(f_\star(s)) \modu 2\pi
$$
is distributed according to the standard BM at time 1 taken modulo $2\pi$. This is the reason why $\RBWtoc$ is said to be standard. As a consequence, the trajectory $\RBWtoc_{x,t}$ turns  a.s. a finite number of times around the origin.

{As the standard BW, the CBW and the in-radial Brownian web admit \textit{special points} from which may start more than one trajectory and whose set a.s. has zero Lebesgue measure. See Section 2.5 in \cite{SSS} for details. Except from these special points, the in-radial Brownian web $\RBWtoc$ can be seen as a tree made up of all the paths $\{\RBWtoc_{x,t}(s), 0\leq s \leq t\}$, $(x,t)\in \rpiz\times `R^{+}$, and rooted at the origin. Its vertex set is the whole plane. Th. \ref{theo:BiInfB} in the sequel also ensures that this tree contains only one semi-infinite branch with probability $1$.}

Let us denote by $\RBWfromc$ the image under  $\varphi_\star$ of the CBW $\Wu$ restricted to $\Cyl^{+}$. We call it the  \textit{standard out-radial Brownian web}. The map $\varphi_\star$ sends the trajectory $\{\Wu_{x,f_\star(t)}(s), s \geq f_\star(t)\}$ of the cylindric BM $\Wu_{x,f_\star(t)}$ starting at $(x,f_\star(t))\in\Cyl^{+}$ on the out-radial (continuous) path $\{\RBWfromc_{2\pi x,t}(s), t\geq s\}$ where
\[
\RBWfromc_{2\pi x,t}(s) := s \exp \l( 2i\pi \Wu_{x,f_\star(t)}(f_\star(s)) \r)~.
\]
Unlike the in-radial path $\RBWtoc_{x,t}$, $\RBWfromc_{x,t}$ is a semi-infinite path which moves away from the origin. Finally, the out-radial Brownian web $\RBWfromc$ appears as the dual of the in-radial Brownian web $\RBWtoc$. Indeed, the CBWs $\Wu$ and $\Wd$ are dual in the sense that no trajectory of $\Wu$ crosses a trajectory of $\Wd$ with probability $1$ (see the proof of Prop. \ref{pro:main-lattice}). Clearly, the map $\varphi_\star$ preserves this non-crossing property which then holds for $\RBWtoc$ and $\RBWfromc$.~\\~\\

\begin{figure}[!ht]
\begin{center}
\vspace{-1cm}
\includegraphics[height=7cm,width=6cm]{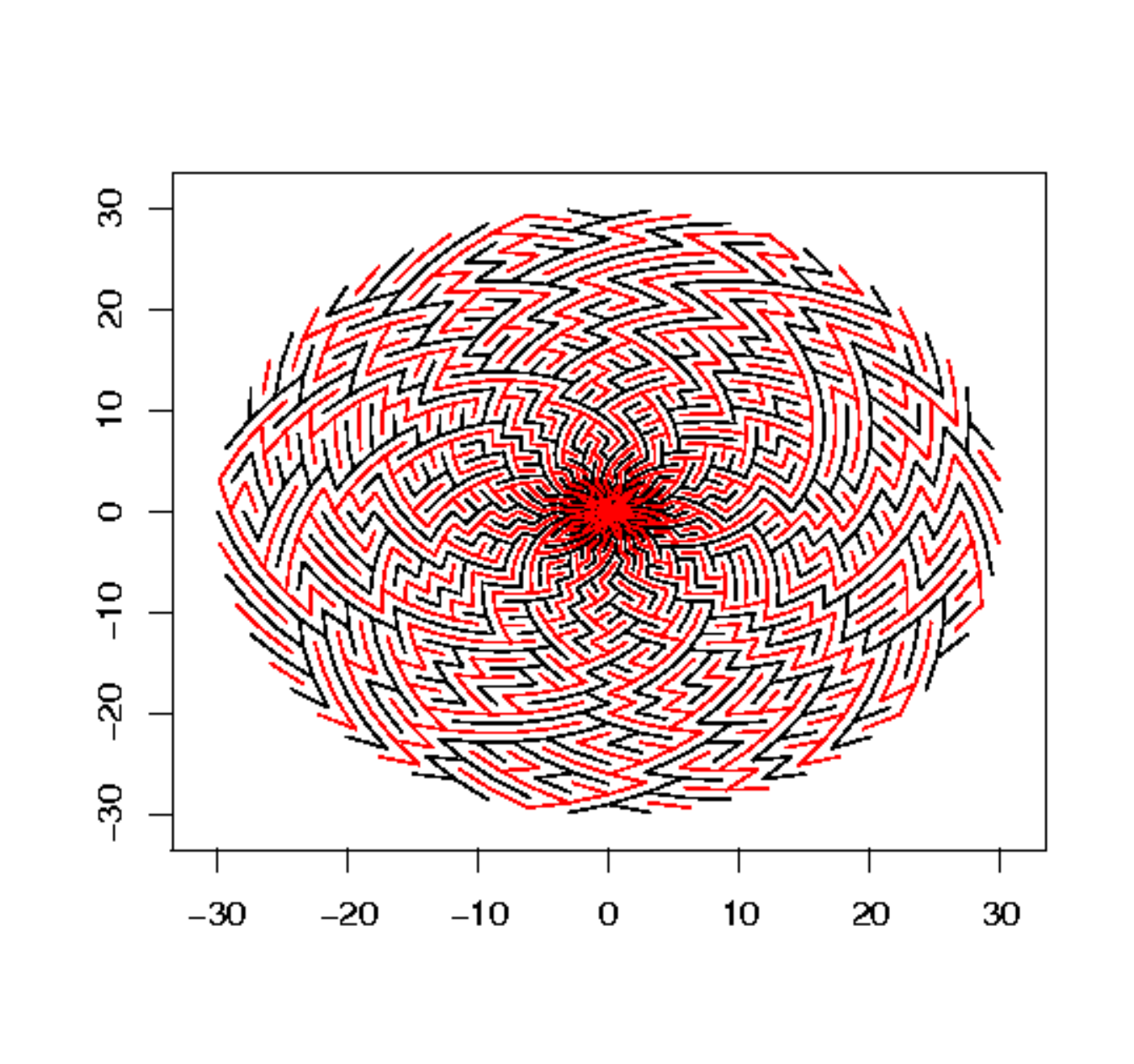}
\end{center}
\vspace{-1cm}
\caption{\label{fig:CylBWn} {\small \textit{Projections of the cylindric lattice webs $\W^{(2n),\ua}$ (black) and $\W^{(2n),\da}$ (red) on the plan by $\varphi_\star$. }}}
\end{figure}

Let us recall that $\W^{(2n),\ua}$ and $\W^{(2n),\da}$ denote the normalized cylindric lattice webs obtained from $\W^{2n,\ua}$ and $\W^{2n,\da}$: see \eqref{def:CLW}. Let us respectively denote by $\RDWtoc_{2\pi x,t}$ and $\RDWfromc_{2\pi x,t}$ the radial lattice webs obtained as
images under $\varphi_\star$ of $\W^{(2n),\ua}$ and $\W^{(2n),\da}$ restricted to $\Cyl^{+}$. Using the continuity of $\varphi_\star$, it is possible to transfer the convergence result of Prop. \ref{pro:main-lattice} from the cylinder to the plane. Then, the convergence result below is a direct consequence of Prop. \ref{pro:main-lattice}.

\begin{theo}
\label{theo:CV-radial}
The pair $(\RDWtoc,\RDWfromc)$ converges in distribution to the pair of standard radial Brownian webs $(\RBWtoc,\RBWfromc)$.
\end{theo}

\subsubsection{Other Radial Brownian Webs}
\label{sec:ORBW}

In this section we explore different radial projections of the cylindric Brownian web $(\Wu,\Wd)$ into the plane. Let us first describe the general setting. Let $f$ be an increasing continuous function, defining a one-to-one correspondence from an interval $I\subset `R^{+}$ onto an interval $J\subset `R$. Define the bijective map $\varphi_f$ by:

\begin{equation}
\label{def:varphi_f}
\app{\varphi_f}{\rpuz \times J}{\rpiz \times I}{(\theta, f(t))}{(2\pi\theta, t)} ~.
\end{equation}
As previously, $\rpiz\times I$ represents a subset of $`R^{2}$ (actually a ring) parametrized by polar coordinates. The map $\varphi_f$ sends the restriction of the CBW $\Wd$ to the part of cylinder $\rpuz\times J$ on a radial object defined on the ring $\rpiz\times I$, denoted by $f$-$\RBWtoc$ and also called \textit{radial Brownian web}. In this construction, the function $f$ is a winding parameter. For instance, if $1,2\in I$, the argument variation (in $\R$) around the origin of the $f-\RBWtoc_{x,2}$ between radii $1$ and $2$ (where $x\in[0,2\pi]$ is the initial argument) is a centered Gaussian r.v.  with variance ${4\pi^2}(f(2)-f(1))$. The standard radial Brownian web introduced in the previous section corresponds to the particular case $I=`R^{+}$, $J=`R^{+}$ and $f(t)=t/(4\pi^{2})$, for which the argument variation of a trajectory on a ring with width $c$ is simply a Gaussian $\mathcal{N}(0,c)$.\\

Our second example of maps $f$ allows to project the complete pair $(\Wu,\Wd)$ parametrized by $\Cyl$ to the plane. Let us consider the bijection from $I=(0,+\infty)$ onto $J=`R$ defined by $f(t):=\ln t$ (or any other map $f$ sending $(0,+\infty)$ onto $\mathbb{R}$). Then, the radial Brownian web $f$-$\RBWtoc$-- image of $\Wd$ by $\varphi_f$ --presents an accumulation phenomenon in the neighborhood of the origin. Indeed, the argument variation around the origin between radii $\varepsilon$ and $1$ has distribution $\mathcal{N}(0,{4\pi^2}|\ln(\varepsilon)|)$, and thus goes to $+\infty$ in the neighborhood of $0$ ($\varepsilon\rightarrow 0^+$) when it stays bounded in any other bounded ring far away from $0$.\\

Our third example of map $f$ provides a tree -- given by the trajectories of $f$-$\RBWtoc$ -- having many semi-infinite branches with asymptotic directions.
A semi-infinite branch $\zeta$ (if it exists) of the tree $f$-$\RBWtoc$ is said to admit an asymptotic direction $\theta\in\rpiz$ whenever $\arg(z_{k})\to\theta$, for any subsequence $(z_{k})\subset\zeta$ such that $|z_{k}|\to\infty$. To show that $f$-$\RBWtoc$ admits many semi-infinite branches, let us consider the bijection $f$ from $I=`R^{+}$ onto $J=[0;1)$ defined by $f(t):=\frac{2}{\pi} \arctan t$. For a small $\varepsilon\in(0,1)$, the map $\varphi_f$ projects the thin cylinder $\rpuz\times [1-\varepsilon;1)$ on the unbounded set $\rpiz\times [\tan(\pi(1-\varepsilon)/2);+\infty)$. On the (small) time interval $[1-\varepsilon;1)$, the CBMs have small fluctuations, and then the tree $f$-$\RBWtoc$ admits semi-infinite branches with asymptotic directions. The next result proposes a complete description of the semi-infinite branches of $f$-$\RBWtoc$.

\begin{rem}
The standard radial Brownian web could appear a bit impetuous to the reader: the fluctuation of the argument along a trajectory parametrized by the modulus, being a BM mod $2\pi$, the trajectories may have important fluctuations far from the origin. The choice $f(t)=\ln t$ of Example 2 provides a radial forest where the paths look like coalescing BMs locally and far from $O$: between radii $r$ and $r+1$, the fluctuations are of variance $1/r$. This model is invariant by inversion.
\end{rem}

\begin{pro}
\label{theo:AsymptDir}
Consider the $f$-RBW for a bijection $f$ from $I=\R_+$ into an interval $J$ with compact closure such that $f$ can be extended continuously to $\adh(J)$. With the above notations, the following statements hold.
\begin{enumerate}
\compact
\item A.s. any semi-infinite branch of $f$-$\RBWtoc$ admits an asymptotic direction.
\item A.s. for any $\theta\in\rpiz$, the tree $f$-$\RBWtoc$ contains (at least) one semi-infinite branch with asymptotic direction $\theta$.
\item For any (deterministic) $\theta\in\rpiz$, a.s. the tree $f$-$\RBWtoc$ contains only one semi-infinite branch with asymptotic direction $\theta$.
\item A.s. there exists a countable dense set $\mathcal{D}\subset\rpiz$ such that, for any $\theta\in\mathcal{D}$, the tree $f$-$\RBWtoc$ contains two semi-infinite branches with asymptotic direction $\theta$.
\item A.s. the tree $f$-$\RBWtoc$ does not contain three semi-infinite branches with the same asymptotic direction.
\end{enumerate}
\end{pro}

\begin{proof}
The first two items generally derive from the \textit{straight} property of the considered tree: see Howard \& Newman \cite{HowardNewman}. However, in the present context, it is not necessary to use such heavy method and we will prove them directly. For the sake of simplicity, we can assume that $J=[0,1)$. Let us first consider a semi-infinite branch $\zeta$ of $f$-$\RBWtoc$.

By construction of the $f$-RBW, there exists a path $\gamma$ of the CBW on $\rpuz \times J$ such that $\zeta=\varphi_f(\gamma)$.
The path $\gamma$ of the CBW on $\rpuz \times J$ is a Brownian motion that can be extended by continuity to $\rpuz \times [0,1]$ by $(\bar{\theta},1)$, say, implying that the first coordinate of $\zeta$ converges to $\bar{\theta}/(2\pi)$ when the radius tends to infinity. This means that the semi-infinite branch $\zeta$ admits $\bar{\theta}$ as asymptotic direction. The proof of the second item is in the same spirit.

The key argument for the three last statements of Th. \ref{theo:AsymptDir} is the following. With probability $1$, for any $\theta\in\rpiz$ and $x:=\theta/2\pi$, the number of CBMs of $\Wd$ starting at $(x,1)$ is equal to the number of semi-infinite branches of $f$-$\RBWtoc$ having $2\pi x$ as asymptotic direction. With Th. \ref{th:defCBW} $(o)$, it then follows that the number of semi-infinite branches of $f$-$\RBWtoc$ having the deterministic asymptotic direction $\theta\in\rpiz$ is a.s. equal to $1$. This key argument also makes a bridge between the (random) directions in which $f$-$\RBWtoc$ admits several semi-infinite branches and the \textit{special points} of $\Wd$. Given $t\in`R$, Th. 3.14 of \cite{fontes2006} describes the sets of points on the real line $`R\times\{t\}$ from which start respectively $2$ and $3$ BMs. The first one is dense and countable whereas the second one is empty, with probability $1$. These local results also hold for the CBW $\Wd$ (but we do not provide proofs).
\end{proof}

\begin{rem}
Cylinders may also be sent easily on spheres, by sending the horizontal slices $h\in (a,b)$ of the cylinder to the horizontal slice $g(h)$ of the sphere $\{(x,y,h) \in \mathbb{R}^3: x^2+y^2+g(h)^2=1\}$, where $-\infty \leq a < b \leq +\infty$, and $g$ is an increasing and bijective function from $(a,b)$ to $(-1,1)$. Somehow, sending cylinders onto the plane allows to contract one slice (or one end) of the cylinder, and sending it on the sphere amounts to contracting two slices (or the two ends) of the cylinder.
Again, this point of view will provide a suitable definition for the spherical Brownian web and its dual.
\end{rem}

\section{Elements on cylindric lattice and Brownian webs}
\label{sect:infinitebranches}

In this section, two differences between the CBW and its plane analogous are put forward. Firstly, each of CBW $\Wu$ and $\Wd$ contains a.s. exactly one bi-infinite branch; this is Th. \ref{theo:BiInfB}, the main result of this section. This property is an important difference with the planar BW which admits a.s. no bi-infinite path (see e.g. \cite{coupiertran} in the discrete case).
The distributions of these bi-infinite paths are identified by taking the limit of their discrete counterparts on the cylindric lattice web.

Secondly, the coalescence time of all the Brownian motions starting at a given slice admits exponential moments (Prop. \ref{pro:Coaltime}). This is also an important difference with the planar case, where the expectation of the coalescence time of two independent Brownian motions is infinite, which comes from the fact that the hitting time $\tau_1$ of $0$ by a Brownian motion starting at 1 is known to have distribution $\P(\tau_1\in dt)=e^{-1/(2t)}/\sqrt{2\pi t^3}$.

\subsection{The bi-infinite branch of the CBW}

For any $x,x'\in \Cyl^D$, $t\in\mathbb{R}$, denote by
\begin{align}
\label{CoalTime2}
T^\ua(x,x',t) =  & \inf \l\{ s>t~: \, \W^\ua_{(x,t)}(s) = \W^\ua_{(x',t)}(s) \r\}\\
T^\da(x,x',t) =  & \sup \l\{ s<t~: \, \Wd_{(x,t)}(s) = \Wd_{(x',t)}(s) \r\}
\end{align}
the coalescence times of the cylindric Brownian motions $\W^\ua_{(x,t)}$ and $\W^\ua_{(x',t)}$ one the one hand, and of $\W^\da_{(x,t)}$ and $\W^\da_{(x',t)}$ on the other hand. Set for $D\in \{\ua, \da\}$,
$$
T^D(t)= \max\l\{T^D(x,x',t), (x,t),(x',t)\in \Slice(t)\r\},
$$
the coalescence time of all the Brownian motions (going upward if $D=\ua$ and downward if $D=\da$) starting at $\Slice(t)$.\par
Consider a continuous function $\gamma\ : \R\mapsto \rpuz$.  We say that $\gamma$, or rather, its graph $\{(\gamma_t,t), t \in \R\}$ is a bi-infinite path of the CBW $\Wu$, if there exists an increasing sequence $(t_k, k \in \mathbb{Z})$ such that $\lim_{k\to -\infty} t_k=-\infty$,  $\lim_{k\to +\infty} t_k=+\infty$, and a sequence $(x_k, k \in \mathbb{Z})$ such that for any $k\in\mathbb{Z}$,  $\W^\ua_{(x_k,t_k)}(t_{k+1})=x_{k+1}$, and $\W^\ua_{(x_k,t_{k})}(s)=\gamma_s$ for $s\in[t_k,t_{k+1}]$.
Similarly, we say that  $\{(\gamma_t,t), t \in \R\}$ is a bi-infinite path of the CBW $\W^\da$, if
there exists an decreasing sequence $(t_k, k \in \mathbb{Z})$ such that $\lim_{k\to -\infty} t_k=+\infty$,  $\lim_{k\to +\infty} t_k=-\infty$ and a sequence $(x_k, k \in \mathbb{Z})$ such that for any $k\in\mathbb{Z}$,
$\Wd_{(x_k,t_k)}(t_{k+1})=x_{k+1}$, and $\Wd_{(x_k,t_k)}(s)=\gamma_s$ for $s\in[t_{k+1},t_k]$.

\begin{theo}
\label{theo:BiInfB}
With probability $1$, any two branches of the CBW $\Wu$ eventually coalesce. Furthermore, with probability $1$, the CBW $\Wu$ contains exactly one bi-infinite branch (denoted $\Cu$).
\end{theo}
A notion of semi-infinite branch is inherited from the cylinder via the map $\varphi_{\star}$:
\begin{cor}The standard out-radial Brownian web $\RBWfromc$ possesses a unique semi-infinite branch.
\end{cor}}
\begin{proof}[Proof of Theorem \ref{theo:BiInfB}] The first statement is a consequence of the recurrence of the linear BM.

Let us introduce some stopping times for the filtration ${\cal F}^\da$.
First let $\tau^{\da,1}=T^\da(0)$ (the coalescing time of the CBW $\W^\da$ coming from $\Slice(0)$ in the dual), and successively, going back in the past,  $\tau^{\da,k}=T^\da(\tau^{\da,k-1})$. Since the primal and dual paths do not cross a.s., it may be checked that all primal Brownian motion $\W^\ua_{(x,\tau^{\da,k})}$ for $x \in \Slice(\tau^{\da,k})$ have a common abscissa, say $x'_{k-1}$ at time $\tau^{\da,k-1}$, that is in $\Slice(\tau^{\da,k-1})$. In other words, they merge before time $\tau^{\da,k-1}$.
A simple picture shows that at $x'_{k-1}$, the dual $\W^\da$ has two outgoing paths, and thus the primal $\W^\ua_{(x'_{k-1},\tau^{\da,k-1})}$ is a.s. a single path (see e.g. \cite[Theorem 2.11]{SSS}, and use the fact that the special points of the CBW are clearly the same as those of the BW).

We have treated the negative part of the bi-infinite path. The positive path is easier, since a bi-infinite path must coincide with the trajectory $\W^\ua_{(x'_0,0)}$ for its part indexed by positive numbers.
As a consequence, the sequence defined by~:\\
-- for $k\geq 0$ by $t_{-k}=\tau^{\da,k}$, $x_{-k}=x'_k$,\\
-- for $k\geq 1$ by $t_k=k$, $x_{k}=W^\ua_{(x_{k-1},t_{k-1})}(t_k)$\\
 does the job if we prove that $\tau^{\da,k}$s are finite times that go to $-\infty$ a.s. But this is a consequence of the strong law of large numbers, since $\tau^{\da,k}$ is a sum of i.i.d. r.v. distributed as  $\tau^{\da,1}$ a.s. finite and positive (by continuity of the BM and comparison with the planar BW).
 \end{proof}

Similarly, it can be proved that any two branches of $\Wd$ eventually coalesce and that $\Wd$ contains a.s. a unique bi-infinite path that we denote $\mathbf{C}^\da$.

\subsection{The bi-infinite branch of the CLW}

As we saw in Prop. \ref{pro:main-lattice}, the CBW can be obtained as a limit of a CLW when the renormalization parameter $n$ in the CLW tends to $+\infty$. We first show that the CLW also has a bi-infinite path and use the explicit transition kernels for the trajectories of the CLW to obtain, by a limit theorem, the distribution of $(\mathbf{C}^\ua,\mathbf{C}^\da)$. The latter are two reflected Brownian motions, as described by \cite{STW}.\\

The coalescence times of the random walks starting at height $h\in \Z$ are respectively~:
\be
\bpar{ccl}
T_n^{\ua}(h)&=&\inf\l\{t\geq h~:~\Wnu_w(t)=\Wnu_{w'}(t), \forall w,w'\in \Slice^\ua_{2n}(h)\r\},\\
T_n^{\da}(h)&=&\sup\l\{t\leq h~:~\Wnd_w(t)=\Wnd_{w'}(t), \forall w,w'\in \Slice^\ua_{2n}(h)\r\}.
\epar\ee
Since for any two points $w,w'\in \Cylu$, $\Wnu_w$ and $\Wnu_{w'}$ eventually coalesce a.s., we have a.s., for any $h$,
\begin{equation}
T^{\ua}_n(h)<+\infty,\qquad  T^{\da}_n(h)>-\infty.\label{CLWcoalesPS}
\end{equation}

For $D\in \{\ua, \da\}$, a bi-infinite path of $\Cyl^D_{2n}$ is a sequence $(x_{i},i)_{i\in \Z}$, such that for all $i\in \Z$, $x_i-x_{i-1} \modu 2n \in\{1,2n-1\}$.  We say that $\W^{2n,D}$ contains a bi-infinite path $\Cn{}^{,D}$ if there is a bi-infinite path $\Cn{}^{,D}$ of $\Cyl^D_{2n}$ whose edges are included in the set of edges of $\W^{2n,D}$.

\begin{pro}\label{pro:CLWhasuniqBIP} A.s., $\Wnu$ and $\Wnd$ each contains a unique bi-infinite path.
\end{pro}
\begin{proof}
Take the slice $h$ and consider $T^{\ua}_n(h)$. Since the paths from $\Wnu$ do not cross those of $\Wnd$, the paths in $\Wnd$ started from $h+T^{\ua}_n(h)$ all meet before slice $h$. Let $\Cdn(h)$ be their common position at height $h$.
Let us consider the sequence $(\tau^{\ua,k}_{n}, k\in \N)$ defined similarly to the one introduced in the proof of Th. \ref{theo:BiInfB}: $\tau^{\ua,0}_n=0$ and for $k\geq 1$, $\tau^{\ua,k}_{n}=T_n^\ua(\tau^{\ua,k-1}_{n})$. This sequence converges to $+\infty$ a.s. since $\tau^{\ua,k}_n$ is the sum of $k$ independent r.v. distributed as $\tau^{\ua,1}_n=T_n^\ua(0)$.
The sequence of paths
$\gamma_k=\Wnd_{(C^{2n,\da}(\tau_n^{\ua,k}),\tau_n^{\ua,k}) }$ is increasing for inclusion and defines a bi-infinite path $\mathbf{C}^\da=\lim_{k\rightarrow +\infty}\ua \gamma_k$ that is unique by the property \eqref{CLWcoalesPS}. The construction of the bi-infinite path for $\Wnu$ follow the same lines.
\end{proof}

Let us describe more precisely the distribution of $(\Cun,\Cdn)$. Let $h_1\leq h_2$ be two heights. We show that $(\Cun,\Cdn)$ is distributed on a time interval $[h_1,h_2]$, as a Markov chain with explicit transitions.\par
For any process $X=(X_i,i \in \mathbb{Z})$ indexed by $\mathbb{Z}$, and $h_1 \leq h_2$, let us denote
\begin{equation*}
X[h_1,h_2]:= (X_{h_1}, X_{h_1+1}, \cdots, X_{h_2}),\mbox{ and } X[h_2,h_1]:= (X_{h_2}, X_{h_2-1},\cdots,X_{h_1}).
\end{equation*}

\begin{lem}\label{lem:qds} For $h_1\leq h_2$, we have\\
$(i)$ $\Cun(h_1)$ and $\Cdn(h_2)$ are independent r.v. respectively uniformly distributed in $\Sliceu_{2n}(h_1)$ and $\Sliced_{2n}(h_2)$,\\
$(ii)$ For any $(x_1,x_2)\in \Sliceu_{2n}(h_1) \times \Sliced_{2n}(h_2)$, conditionally on $(\Cun(h_1),\Cdn(h_2))=(x_1,x_2)$,
\ben\label{eq:qds}
\l(\Cun[h_1,h_2],\Cdn[h_2,h_1]\r)\eqd \l(\Wnu_{(x_1,h_1)}[h_1,h_2],\Wnd_{(x_2,h_2)}[h_2,h_1]\r)
\een
If ${\sf Pair}^{\ua,\da}(x_1,x_2,h_1,h_2)$ denotes the support of $(\Wnu_{(x_1,h_1)}[h_1,h_2], \Wnd_{(x_2,h_2)}[h_1,h_2])$,
then for any $(C_1,C_2)\in {\sf Pair}_n^{\ua,\da}(x_1,x_2,h_1,h_2)$
\ben\label{eq:qdk}
`P\l(( \Wnu_{(x_1,h_1)}[h_1,h_2],\Wnd_{(x_2,h_2)}[h_1,h_2])=(C_1,C_2)\r)=2^{-2(h_2-h_1)+\Nb(C_1,C_2)}
\een
where $\Nb(C_1,C_2)$ is the ``number of contacts" between $C_1$ and $C_2$~:
\ben
\Nb(C_1,C_2)=\#\{i \in [h_1,h_2-1]~: C_1(i)= C_2(i+1)\}.
\een
\end{lem}
\begin{proof} The family $(\Cun(h), h<h_1)$ (resp. $(\Cdn(h), h\geq h_2)$) is a function of the Rademacher r.v. placed on $\cup_{h<h_1} \Sliceu_{2n}(h)$ (resp. on $\cup_{h\geq h_2} \Sliceu_{2n}(h)$). Hence, $(\Cun(h), h<h_1)$ and $(\Cdn(h), h\geq h_2)$ are independent, and independent of $\Cun[h_1, h_2-1]$. Clearly, $\Cun(h_1)$ and $\Cdn(h_2)$ have invariant distributions by rotation, so they are uniform, and \eqref{eq:qds} holds.

{Using the Rademacher r.v. $\xi$'s defined at the beginning of Section \ref{section:CylLatticeWeb}, we have
\begin{multline}
\{(\Wnu_{(x_1,h_1)}[h_1,h_2],\Wnd_{(x_2,h_2)}[h_1,h_2])=(C_1,C_2)\}= \\
\{\forall h_1\leq i < h_2,\ \ \  \xi(C_1(i),i)=C_1(i+1)-C_1(i) \modu2n \\
 \mbox{ and }\quad \xi(C_2(i+1),i)=C_2(i+1)-C_2(i) \modu2n\},\end{multline}since the edges of the dual are determined by the edges of the primal. The number of Rademacher $(\xi(w),w \in \Cylun)$ contributing to the above event is $2(h_2-h_1)-\Nb(C_1,C_2)$, hence the result. $\Nb(C_1,C_2)=\#\{i~: (C_1(i),i)= (C_2(i+1),i)\}$ is the number of edges $(\xi(u),u \in \Cylun)$ contributing to the definition of both $(C_1,C_2)$. Apart these edges, each increment of $\Cd$ and of $C^\ua$ are determined by some different Rademacher r.v. Hence $2(h_2-h_1)-\Nb(C_1,C_2)$ edges determine the event
$\{ (\W_{(x_1,h_1)}^{\ua}[h_1,h_2],\W_{(x_2,h_2)}^{\da}[h_1,h_2])=(C_1,C_2)\}$.}
\end{proof}
From the above Lemma, it is possible to give a representation of the vectors $\Cun[h_1,h_2]$ and $\Cdn[h_2,h_1]$ with a Markov chain whose components both go in the same direction $\uparrow$.
\begin{lem} \label{lem:MC} We have
\[\l(\Cun[h_1,h_2],\Cdn[h_1,h_2]\r) \eqd (M_1[h_1,h_2],M_2[h_1,h_2])\]
where $M=(M_1,M_2)$ is a Markov chain whose initial distribution is uniform on $\Sliceu_{2n}(h_1)\times\Sliced_{2n}(h_1)$, and whose transition kernel $K$ is defined as follows:\\
if $d_{\mathbb{Z}/2n\mathbb{Z}}(a,a')>1$,
\ben\label{eq:ker1}
K((a,a'),(a + \eps \modu 2n, a' + \eps' \modu 2n)=1/4,~~~\textrm{ for any }(\eps,\eps')\in \{-1,1\}^2\een
if $d_{\mathbb{Z}/2n\mathbb{Z}}(a,a')=1$,
\ben\label{eq:ker2}\left\{
\begin{array}{lcl}
K((a,a+1 ),(a+1  ,a+2 ))&=&1/2,\\
K((a,a+1 ),(a-1  ,a+2 ))&=&1/4,\\
K((a,a+1 ),(a-1  ,a ))  &=&1/4,\\
K((a+1,a ),(a  ,a-1 ))  &=&1/2,\\
K((a+1,a ),(a+2, a-1 )) &=&1/4,\\
K((a+1,a ),(a+2  ,a+1 )) &=&1/4,
\end{array}
\right.\een
where $a,a-1,a+1,a+2$ are considered modulo $2n$.
\end{lem}
Notice that the starting points of $M$ is a pair of uniform points at time $h_1$, while for $\Cun[h_1,h_2]$ and $\Cdn[h_1,h_2]$ the starting points were on two different slices (see Lemma \ref{lem:qds}).

\begin{proof} First, both distributions have same support, which is
$$\bigcup_{(x_1,x_2)\in \Slice^\ua_{2n}(h_1)\times \Slice^\da_{2n}(h_2)} {\sf Pair}^{\ua,\da}(x_1,x_2,h_1,h_2),$$
the set of pairs of non-crossing paths living on $\Cylun\times \Cyldn$. By Lemma \ref{lem:qds}, we see that for any pair $(C_1,C_2)$ in this support we have $`P( (\Cun_{h_1}[h_1,h_2],\Cdn_{h_2}([h_1,h_2])= (C_1,C_2))= 2^{-2(h_2-h_1)+\Nb(C_1,C_2)}$. The Markov kernel has been designed to satisfy the same formula.
\end{proof}

\subsection{Distribution of $(\C^\ua, \C^\da)$}

In the sequel, we consider the sequence $(\Cun,\Cdn)_{n\in \N}$ correctly renormalized and interpolated as a sequence of continuous functions. We will prove its convergence in distribution on every compact set $[h_1,h_2]$ (with $h_1<h_2$) to $(\C^\ua,\C^\da)$, a pair of reflected Brownian motions modulo 1 (see Figure \ref{fig:chemininf400-4}). This result is similar to that of  Soucaliuc et al. \cite{STW} introduced in the next paragraph.\par

Let $F:\mathbb{R}\to[0,1]$ be the even, 2-periodic function defined over $[0,1]$ by $F(x)=x$.

Let us consider $U_1$ and $U_2$ two i.i.d uniform r.v. on $[0,1]$, and $B$ and $B'$ two i.i.d. BM starting at 0 at time $h_1$ and independent of $U_1$ and $U_2$.
Let $(Y^\ua,Y^\da)$ be the following continuous process defined for $t\in [h_1,h_2]$ and taking its values in $\Cyl^2$
\begin{equation}\label{def:Y}
{(Y^\ua, Y^\da)(t)=\l( U_1+\frac{B'_t}{\sqrt{2}}-H(t)\mod 1, U_1+\frac{B'_t}{\sqrt{2}}+H(t)\mod 1 \r),}
\end{equation}
where $H(t)$  represents half the ``distance'' $|Y^\ua(t)\to Y^\da(t)|$~:
\begin{equation}\label{def:H}
H(t)=\frac{F\l(U_2+\sqrt{2}B_t\r)}{2}.\end{equation}
Since $F$ is bounded by $1$, $Y^\ua$ and $Y^\da$ never cross.

\begin{figure}[!ht]
\begin{center}
\includegraphics[height=5cm, width=9cm]{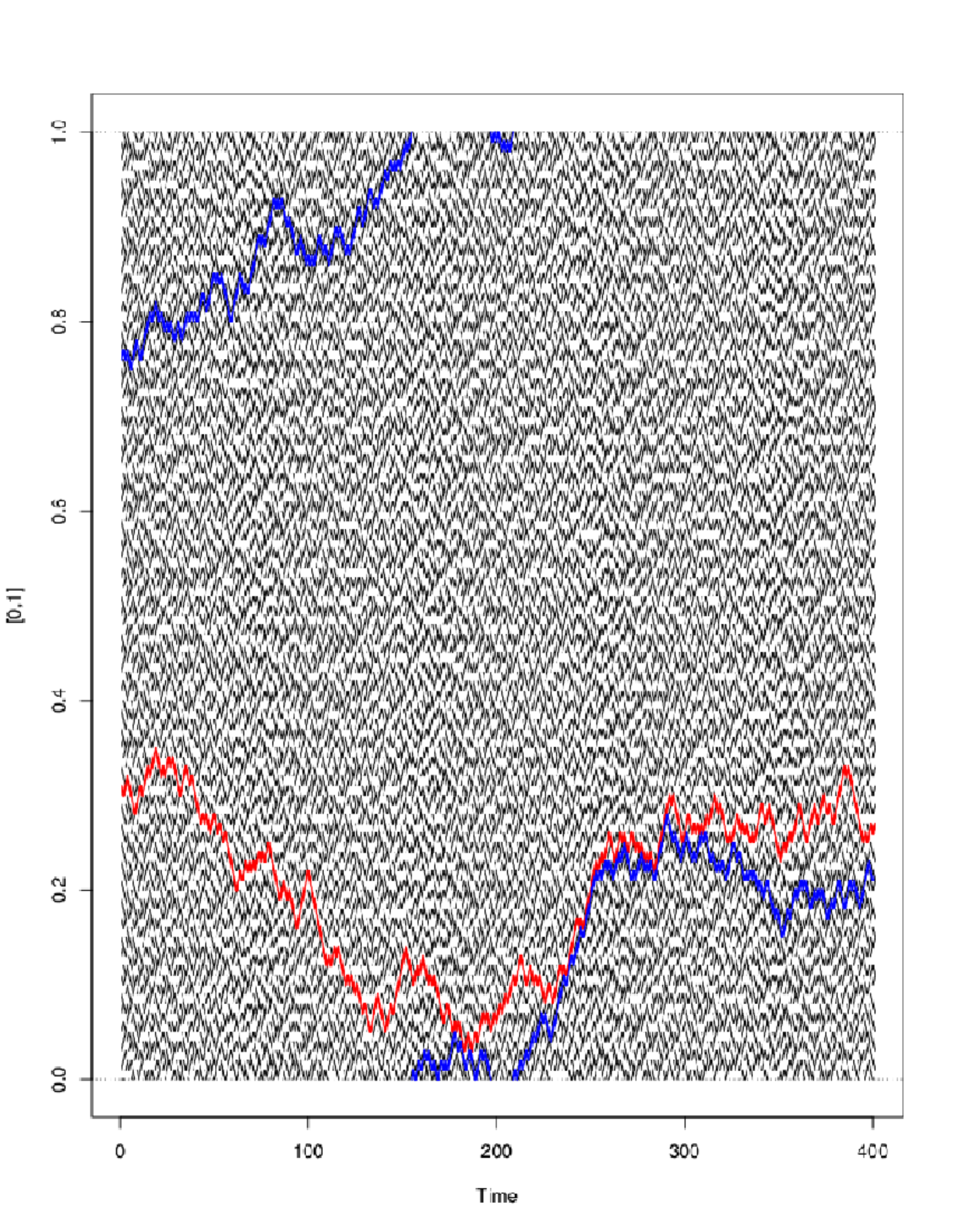}
\end{center}
\caption{\label{fig:chemininf400-4} {\small \textit{The infinite paths of the CBW (in red) and of its dual (in blue) are distributed as reflected Brownian motions modulo 1.}}}
\end{figure}

\begin{theo}\label{theo:convSTW}We have the following convergences in distribution:\\
(i) Let $h_1 <h_2$. Let $U^\ua_n$ and $U^\da_n$ be two independent uniform r.v. on $\Slice_{2n}(h_1)$ and $ \Slice_{2n}(h_2)$ respectively. Then in $\Co([h_1,h_2],(\rpuz)^2)$:
\ben\label{eq:rtp} \l(\frac{\Wnu_{U^\ua_n} (4n^2 \, .) }{2n}, \frac{\Wnd_{U^\da_n}(4n^2 \, .)}{2n}\r)
\dd (Y^\ua,Y^\da).
\een
(ii) In $\Co(\R,(\rpuz)^2)$:
\[\l(\frac{\Cun(4n^2\,.)}{2n}, \frac{\Cdn(4n^2\,.)}{2n}\r)\dd (\Cu,\Cd)\]
and { $(\Cu,\Cd)\eqd(Y^\ua,Y^\da)$.}
\end{theo}

Notice that for $t=h_1$, {$$(Y^\ua(h_1),Y^\da(h_1))=(U_1-F(U_2)/2\mod 1,U_1+F(U_2)/2\mod 1)$$ }
which is indeed a pair of i.i.d. uniform r.v. on $[0,1]$ as expected in view of Lemma \ref{lem:qds} $(i)$.\\

The remaining of this section is devoted to the proof of Theorem \ref{theo:convSTW}, which is separated into several steps. Let us start with point (i). \\

\noindent \textbf{Step 1: Tightness of $\l(\Wnu_{U^\ua_n}(4n^2 \, .)/(2n), \Wnd_{U^\da_n}(4n^2\, .)/(2n)\r)$}\\

By translation invariance, we may suppose that $h_1=0$ and set $h_2=T$.
The tightness of the family of distributions of $\l(\Wnu_{U^\ua_n}(4n^2 \, .)/(2n), \Wnd_{U^\da_n}(4n^2\, .)/(2n)\r)$ in $C([0,T],(\rpuz)^2)$ follows from the tightness of its marginals that are simple well rescaled random walks on the circle. Now, our aim is to identify the limiting distribution. For that purpose, and in view of Lemmas \ref{lem:qds} (ii) and \ref{lem:MC}, we study more carefully the Markov chain $(M_1,M_2)$.\\

\noindent \textbf{Step 2: Angle process between $M_1$ and $M_2$}\\

Let us extend the notation $[a\to b]$ and $|a\to b|$ for $a$ and $b$ in $\R/2n\Z$. For the Markov chain $M$ defined in Lemma \ref{lem:MC}, the angle process between the two components is
\[A(i)= |M_1(i) \to M_2(i)|,\qquad i\geq 0.\]
Of course, for any $i$, $(M_1(i),M_2(i))= (M_1(i),M_1(i)+A(i) \modu 2n)$. We will focus on the asymptotics of $((M_1(i), A(i)), i\geq 0)$.\\

Recall that $M_1$ and $M_2$ are simple non-independent random walks with Rademacher increments. Let us write:
\[\bpar{ccl}
M_1(i) & = & M_1(0)+\dis\sum_{j=1}^i R_{2j-1},\\
M_2(i) & = & M_2(0)+\dis\sum_{j=1}^i R_{2j}
\epar\]
where $(R_{2i},i\geq 1)$ and $(R_{2i-1},i\geq 1)$ are two families of i.i.d. Rademacher r.v., the two families being possibly dependent from each other.
The process $A$ takes its values in the set of odd integers in $[0,2n]$, and its are sums of 2 Rademacher r.v.\\

Now, let us consider the simple random walk
\[\overline{Z}(i) = A(0) + \sum_{j=1}^i (-1)^{j}R_j=M_2(i)-M_1(i)\]
starting from $A(0)$. If $M_1$ and $M_2$ were allowed to cross, then $A(i)$ would be equal to $\overline{Z}(2i)$. We have to account for the non-crossing property of the paths of $\Wnu$.\\

A random walk $(Z_i,i\geq 0)$ is said to be the simple random walk reflected at $0$ and $2n$, and starting at some  $b\in\cro{0,2n}$ if $(Z_i,i\geq 0)$ is a Markov chain such that
\[
\bpar{rcl}
\P(Z_{i+1}=1 \,|\, Z_{i}=0) &=& \P(Z_{i+1}=2n-1\, |\, Z_{i}=2n)=1\\
\P(Z_{i+1}=a\pm 1\,|\, Z_i=a)&=& 1/2,\textrm{ for any } a \in \cro{1,2n-1}.
\epar\]
For any discrete time process $X$, denote by
\[
\Delta X_i:=X_i-X_{i-1},
\] the $i$th increment of $X$.
We have
\begin{lem}\label{lem:rep}
The distribution of the process $((A(i),M_1(i)),i\geq 0)$ starting at $(A(0),M_1(0))$ where $A(0)\in\cro{1,2n-1}$ is odd, and $M_1(0) \in \cro{0,2n-1}$ is even is characterized as follows: \\

For $(Z_i,i\geq 0)$ a simple random walk reflected at $0$ and $2n$, and starting from $A(0)$, we have:
\ben\label{eq:tehez}
(A(i),i\geq 0)\eqd (Z_{2i},i\geq 0) \eqd F_{2n}(\overline{Z}_{2i},i\geq 0),
\een
where $F_{2n}:\mathbb{Z}\to \cro{0,2n}$, the even $4n$-periodic function, defined on $[0,2n]$ by $F_{2n}(x):= x $.\\

The random walk $M_1$ starting at $M_1(0)$ admits as increments the sequence
\ben\label{eq:fdsf}
\l(\Delta M_1(i),i\geq 1\r)=\l(-\Delta Z_{2i-1},i \geq 0\r),
\een that is the opposite of the increments with odd indices of $Z$.
\end{lem}

Notice that the second identity in \eref{eq:tehez} holds in distribution only: as defined, the reflection only modifies the increments that follow the hitting times of $0$ and $2n$, whereas the map $F_{2n}$ turns over large part of the trajectory $(Z_{i},i\geq 0)$. Denoting by $ w_{2n}(\ell)=\lfloor \ell/(2n)\rfloor$ the discrete ``winding number'' of $\ell$, according to Lemma \ref{lem:rep}, the increments of the process $M_1$ under this representations are
\[\Delta M_1(\ell) = (-1)^{w(\overline{Z }_{2\ell-1})}\Delta \overline{Z }_{2\ell-1}.\]

\begin{proof}[Proof of Lemma \ref{lem:rep}]
The distance $|M_1(i)\to M_2(i)|$ decreases when $\Delta M_1(i)=R_{2i-1}=1$ increases and increases when $\Delta M_2(i)=R_{2i}=1$, so that $\overline{Z}(2i)$ would be equal to $A(i)$ if the two walks were not constrained to not cross. We would also have
\begin{equation}\label{eq:M-et-barZ}
\Delta M_1(i) =  -(\overline{Z}_{2i-1}-\overline{Z}_{2i-2})=-\Delta \overline{Z}_{2i-1}.
\end{equation}
Since $0$ and $2n$ are even, and since $\overline{Z}(0)=A(0)$ is odd, the random walk can hit $0$ and $2n$ only after an odd number of steps. In other words, the reflection will concern only the steps with even indices. Therefore, let $(Z_i,i\geq 0)$ be the random walk $\overline{Z}$ reflected at 0 and $2n$: the odd increments of $Z$ and of $\overline{Z}$ are the same, and the even increments correspond except when $Z_{2i-1}\in \{0,2n\}$, in which case the reflection implies that $Z_{2i}=\ind_{Z_{2i-1}=0}+(2n-1)\ind_{Z_{2i-1}=2n}$. It is easy to check from \eqref{eq:ker1}-\eqref{eq:ker2} that $(Z_{2i}, \ i\geq 0)$ has the same distribution as the angle process $(A(i), i\geq 0)$ started from $A(0)$ one the one hand, and as $(F_{2n}(\overline{Z}_{2i}), i\geq 0)$ on the other hand.\\
Finally, notice that because the odd increments are the same, \eqref{eq:M-et-barZ} also holds for $Z$.
\end{proof}

\noindent \textbf{Step 3: Identification of the limit} \\

\begin{lem}\label{lem:qfs}
Let $U_1$, $U_2$ are two uniform r.v. on $[0,1]$, let $B$ and $B'$ be two BMs, all being independent.
We have in $\Co(\R_+,\rpuz)$ that
\ben\label{eq:qfegfsd}
\l(\frac{M_1(4n^2.)}{2n},\frac{A(4n^2.)}{2n}\r)\dd \l( U_1+B'_.-H(.)\mod 1, 2H(.)\r)
\een
where $H(t):={F(U_2 + \sqrt{2} B_t)}/{2},$ has been defined in \eqref{def:H}.\end{lem}

\begin{proof}Let us first consider the angle component. Since the discrete process $A(4n^2t)/(2n)$ is the difference between two (dependent) suitably rescaled random walks which are both tight under this rescaling, the process $A(4n^2t)/2n$ is tight in $\Co(\R_+,\rpuz)$. To characterize the limiting process, write
\[\frac{A(4n^2t)}{2n}= \frac{F_{2n}(\overline{Z}_{8n^2t})}{2n}
                   = F \l( \frac{\overline{Z}_{8n^2t}}{2n}\r)\]
since for every $x$ and every $n$, $F_{2n}( 2n x ) = 2n F( x )$. The central limit theorem implies the convergence  $\frac{\overline{Z}_{8n^2t}}{2n}\dd U_2+{\cal N}(0, 2t)$ for a fixed $t\geq 0$. Since, the mapping $g\mapsto (t\mapsto F(g(t))$ is continuous on $\Co(\R_+,\rpuz)$s, the independence and stationarity of the increments of $Z$ provide the finite dimensional convergence of the angle process in \eqref{eq:qfegfsd}.\\

For the first component, we know that $\frac{M_1(4n^2.)}{2n}$ converges in distribution to a BM modulo 1, but that is not independent from the limit $H$ of $\frac{A(4n^2.)}{2n}$. The result is a consequence of the following lemma, proved in the sequel.\end{proof}

\begin{lem}\label{lem:rgfz} Let $B$ and $B'$ be two independent BM, and let $X=B+B'$ be the sum process. For any $(b_0,x_0)\in\R^2$, conditionally on $\{(X_t=x_t,t\in[0,T]) ,B_0=b_0,B'_0=x_0-b_0\}$, we have
\ben\label{lem:dec}
(B.,B'.)_{[0,T]}&\eqd & \l(b_0-\frac{x_0}2+\frac{x_.}{2} +\frac{B''.}{\sqrt{2}}, -b_0+\frac{x_.}{2}-\frac{B''.}{\sqrt{2}}\r)_{[0,T]}
\een
for an independent BM $B''$.
\end{lem}

\noindent {\bf Step 4: Proof of Theorem \ref{theo:convSTW} $(ii)$.} \\

Consider two levels $h_1\leq h_2$.
First, remark that the restriction of $(\C^\ua, \C^\da)$ to the compact interval $[h_1,h_2]$ has same distribution as $(\Wu_{U^\ua},\Wd_{U^\da})$ on $[h_1,h_2]$, where $U^\ua$ and $U^\da$ are independent and uniformly distributed on $\Slice(h_1)$ and $\Slice(h_2)$ (indeed, $\Wu(h_1)$ depends only on what happens below the level $h_1$ and $\Wd(h_2)$ depends only on what happens above the level $h_2$.
\\
From \eqref{eq:rtp}, it  remains to prove that $(\Wu_{U^\ua},\Wd_{U^\da})$ on $[h_1,h_2]$ is distributed as $(Y^\ua,Y^\da)$.\\
For any $(x,h)\in \Cyl$, the map $\Pi_{(x,h)}\ : \ F\in (\mathcal{H}_O,d_{\mathcal{H}_O}) \mapsto W_{(x,h)} \in \Co([h,+\infty),\R\setminus \Z)$ that associates to a forest the path started at $(x,h)$ is continuous. From Prop. \ref{pro:main-lattice}, we thus deduce that the Markov chain $M$ of Lemma \ref{lem:MC} correctly renormalized converges on $[h_1,h_2]$, when $n\rightarrow +\infty$, to $(\Wu_{u_1-F(u_2)/2}(t),\Wd_{u_1+F(u_2)/2}(t))_{t\in [h_1,h_2]}$
 the paths of $\CBW$ and its dual. We deduce that $(Y^\ua,Y^\da)$ has the same distribution.
This concludes the proof of Theorem \ref{theo:convSTW}.

\begin{proof}[Proof of Lemma \ref{lem:rgfz}]
Since we are dealing with Markov processes with stationary increments and simple scaling properties, it suffices to show that for $X_1,X_2,N$ 3 i.d.d. ${\cal N}(0,1)$ r.v. , we have that conditionally on $S=X_1+X_2$,
$$(X_1,X_2) \stackrel{(d)}{=} \big(\frac{S}{2}+\frac{N}{\sqrt{2}},\frac{S}{2}-\frac{N}{\sqrt{2}}\big).$$
This is a consequence of Cochran theorem, which gives that $(X_1-S/2, X_2-S/2)$ is a Gaussian vector independent from $S$. Since $X_1-S/2=(X_1-X_2)/2=-(X_2-S/2)$, introducing $N/\sqrt{2}=X_1-S/2$ finishes the proof.
\end{proof}

\subsection{The coalescence times have exponential moments}

Th. \ref{theo:BiInfB} states that the coalescence times $T^\ua(x,x',t)$ and $T^\ua(t)$ are finite a.s. Due to the compactness of the space $\rpuz$, we can prove in fact that they admit exponential moments.

\begin{pro}
\label{pro:Coaltime}
$(i)$ There exist $b>0$, {$M<+\infty$ } such that for any $x,x'\in\rpuz$ and any $t\in`R$,
$$
`E \left[ e^{b (T^\ua(x,x',t)-t)}\right] < M.
$$

$(ii)$ For any $t\in `R$, the coalescence time $T^\ua(t)$ admits exponential moments~:
$$
\exists a>0 , \, `E \left[ e^{a (T^\ua(t)-t)}\right] < \infty ~.
$$
\end{pro}

\begin{proof}For both assertions, by the time translation invariance of the CBW, it suffices to consider only the case $t=0$ . \\
 $(i)$ We can assume that $0\leq x\leq x'< 1$.
We have before crossing time
$$
(\W^\ua_{(x',0)}(t)-\W^\ua_{(x,0)}(t), 0\leq t\leq T(x,x',0)) \eqd (x'-x+ \sqrt{2}B(t) \modu1, 0\leq t\leq T(x,x',0)),
$$
where $B$ is a standard usual Brownian motion.
Hence  $T(x,x',0)$ has same distribution as the exit time of a linear BM $B$ from the segment $[-(x'-x),1-(x'-x)]$.
This exit time is known to admit exponential moments (see e.g. Revuz \& Yor \cite[Exo. (3.10) Chap. 3]{revuz-yor}).\\
$(ii)$ We will use a very rough estimate to prove this fact.

Let for $k\geq 0$, $A_k$ be the following independent events~:
\[A_k=\{T^{\ua}(2k)\leq 2k+2\}\]
meaning that all trajectories born at height $2k$ have coalesce before time $2k+2$. If we show that
$p:=P(A_k)> 0$, then $T(0) \leq \min\{ k, A_k \textrm{ holds}\}$ is bounded by twice  a geometric r.v. $p$, and then has some exponential moments. So let us establish this fact.

For this, we use a single argument twice. Consider $Z$ the hitting time
of two BM starting at distance $1/2$ on $\rpuz$. Clearly
$q:=\mathbb{P}(Z\leq 1)>0$.
Let us now bound $P(A_0)$. For this consider ``half of the dual CBW'' $(\Wd_{(x,1)}, 0\leq x\leq 1/2)$ starting at $\Slice(1)$. With probability $q$ these trajectories merge before $\Slice(0)$. Conditionally to this event, all primal trajectories $(\W^\ua_{(x,0)}, x \in \Slice(0))$ starting at time 0 a.s. avoid the dual trajectories, and satisfy
$\W^\ua_{(x,0)}(1) \in (1/2,1),$
meaning that, with probability $q$ at least, they will be in the half interval $(1/2,1)$. But now, the two trajectories $\W^\ua_{(1/2,1)}$ and $\W^\ua_{(1,1)}$, will merge before time 2 with probability $q$. Conditionally to this second event, with probability $\geq 1/2$, the merging time of $(\W^\ua_{(x,1)}, x\in (1/2,1))$ is smaller than 1. Indeed on $\Cyl$, by symmetry, when $\W^\ua_{(1/2,1)}$ and $\W^\ua_{(1,1)}$ merge, they ``capture'' all the trajectories starting in $[1/2,1]$ (which will merge with them) or they capture all the trajectories starting in  $[0,1/2]$. Since both may happen, the probability of each of this event are larger than $1/2$. Hence $p \geq q^2/2$ and the proof is complete.
\end{proof}

\subsection{Toward explicit computations for the coalescence time distribution}
Notice that an approach with Karlin-McGregor type formulas can lead to explicit (but not very tractable) formulas for the distribution of the coalescing time of several Brownian motions.
Let us consider $0<x_1<x_2<\dots <x_k<1$, and denote by $T_k$ the time of global coalescence of the $k$ Brownian motions $W_{(x_1,0)}^\ua,\dots W_{(x_k,0)}^\ua$.

Taking to the limit formulas obtained by Fulmek \cite{fulmek}, we can describe the distribution of the first coalescence time $T^{k\to k-1}$ between two of these paths~:
\begin{equation}\label{def:Tktok-1}
T^{k\to k-1}(x_i,1\leq i \leq j)=\min\{T^\ua(x_i,x_{i+1},0),\ 1\leq i\leq k\}
\end{equation}
with the convention that $x_{k+1}=x_1$, and where $T^\ua(x_i,x_{i+1},0)$ is the time of coalescence of $\W^\ua_{(x_i,0)}$ and $\W^\ua_{(x_{i+1},0)}$ as defined in \eqref{CoalTime2}. We will omit the arguments $(x_i,1\leq i \leq j)$ in the notation $T^{k\to k-1}_{(x_i,1\leq i \leq j)}$ unless necessary. For $t>0$,
\begin{equation}
\mathbb{P}(T^{j\to j-1}>t)=\int dy_1\dots \int dy_j \ind_{0<y_1<y_2<\dots <y_j<1} \Big[\sum_{i=0}^{j-1} \sgn(\sigma^i) \prod_{\ell=1}^j \Phi_t\big(y_\ell-x_{\sigma^i(\ell)}\big)\Big]\label{eq:fulmek}
\end{equation}
where $\sigma^i$ denotes the rotation $\sigma^{i}(\ell)= \ell+i \mod j$  and where
$$\Phi_t(x)=\frac{1}{\sqrt{2\pi t}}\sum_{m \in \Z}   \exp\big(-\frac{(x-m)^2}{2t}\big).$$
Explicit formulas for the Laplace transform of $T^{j\to j-1}$ are not established in general cases to our knowledge, except for the following special case when $k=2$ and $\theta<0$ (see e.g. Revuz \& Yor \cite[Exo. (3.10) Chap 3]{revuz-yor}):
\begin{equation}\label{laplace:T2to1}
\E\Big(e^{\theta T^{2\to 1}}\Big)=\frac{\cosh\big(\sqrt{|\theta|} \frac{1+2x_1-2 x_2}{2}\big)}{\cosh\big(\frac{\sqrt{|\theta|}}{2} \big)}.
\end{equation}

\par Using that $\E(e^{\theta T^{j+1\to j}})=1+\int_0^{+\infty} \theta e^{\theta t} \P(T^{j+1\to j}>t)dt$ and the Markov property, we can finally link \eqref{def:Tktok-1} and $T_k$:
\begin{equation}
\E\Big(e^{\theta T_k}\Big)=\prod_{j=1}^{k-1}\E\left(\E\Big(e^{\theta T^{j+1\to j}(\W_1(T_j),\dots \W_{k}(T_j))} \ |\ \W_1(T_j),\dots \W_{k}(T_j)\Big)\right), \label{laplace:Tk}
\end{equation}
where $T_j$ is the time of the $k-j$th coalescence (at which there are $j$ Brownian motions left) and $(\W_1(T_j),\dots \W_{k}(T_j))$ are the values of the $k$ coalescing Brownian motions at that time (and hence only $j$ of these values are different).\par
It is however difficult to work out explicit expressions from these formula.

\section{Directed and Cylindric Poisson trees}
\label{sect:ConvCBW}

Apart from the (planar) lattice web $W^{2n}$, defined as the collection of random walks on the grid $\Gr=\{(x,t)\in \Z^2,\, x-t \mod 2 =0 \}$ (see \cite[Section 6]{fontes2004} or Figure \ref{fig:Cyln}), several discrete forests are known to converge to the planar BW; in particular the \textit{two-dimensional Poisson Tree} studied by Ferrari \& al. in \cite{ferrarilandimthorisson,ferrarifonteswu}. In Section \ref{sec:Poisson-tree}, a cylindric version of this forest is introduced and we state the convergence of this (continuous space) discrete forest to the CBW. See Th. \ref{theo:ffw-cyl} below.
Our proof consists in taking advantage of the local character of the assumptions $(B2O)$ and $(B2)$. Indeed, the cylinder locally looks like the plane and we can couple (on a small window) the directed and cylindrical Poisson trees in order to deduce $(B2O)$ from $(B2)$.

Finally, in Section \ref{sec:FCP}, we discuss under which assumptions, conditions $(B2)$ and $(B2O)$ can be deduced from each other.

\subsection{Convergence to the CBW}
\label{sec:Poisson-tree}

Let $n\geq 1$ be an integer, and  $r>0$ be a real-valued parameter. Consider a homogeneous Poisson point process (PPP in the sequel) $\mathcal{N}_\lambda$ with intensity $\lambda >0$ on the cylinder $\Cyl$ defined in \eqref{def:Cyl}.

Let us define a directed graph with out-degree $1$ having $\mathcal{N}_\lambda$ as vertex set as follows: from each vertex $X=(x,t)\in\mathcal{N}_\lambda$ add an edge towards the vertex $Y=(x',t')\in \mathcal{N}_\lambda$ which has the smallest time coordinate $t'>t$ among the points of $\mathcal{N}_\lambda$ in the strip $\{(x'',t'') \in \Cyl ~|~ d_{`R/\mathbb{Z}}(t,t'') \leq r\}$ where  $d_{`R/\mathbb{Z}}(x,x''):=\min\{|x-x''|, |1+x-x''|\}$. Let us set $\alpha(X):=Y$ the out-neighbor of $X$. {Notice that even if $X$ does not belong to $\mathcal{N}$ the ancestor $\alpha(X)\in \mathcal{N}$ of this point can be defined in the same way.}
For any element $X\in \Cyl$, define $\alpha^{0}(X):=X$ and, by induction, $\alpha^{m+1}(X):=\alpha(\alpha^{m}(X))$, for any $m\geq 0$. Hence, $(\alpha^{m}(X))_{m\geq 0}$ represents the semi-infinite path starting at $X$. We define by $\W_{X}^{\lambda,{r,}\ua}$ the continuous function from $[t;+\infty)$ to $`R/\mathbb{Z}$ which linearly interpolates the semi-infinite path $(\alpha^{m}(X))_{m\geq 0}$.

The collection $\W^{\lambda,{r,}\ua}:=\{\W_{X}^{\lambda,{r,}\ua}, X\in\mathcal{N}_\lambda\}$ is called the \textit{Cylindric Poisson Tree} (CPT). This is the analogue on $\Cyl$ of the two-dimensional Poisson Tree introduced by Ferrari et al. in \cite{ferrarilandimthorisson}. Also, $\W^{\lambda,{r,}\ua}$ can be understood as a directed graph with edge set $\{(X,\alpha(X)) : X \in \mathcal{N}_\lambda\}$. Its topological structure is the same as the CBW (see Th. \ref{theo:BiInfB}) or as the CLW (see Prop. \ref{pro:CLWhasuniqBIP}). The CPT a.s. contains only one connected component, which justifies its name: it is a tree and admits only one bi-infinite path (with probability $1$).

Let us choose $\lambda=n$ and rescale $\W^{\lambda,{r,}\ua}$ into $\W^{(n),r,\ua}$ defined as
$$
\W^{(n),r,\ua} := \left\{ \W^{n,\frac{r}{n},\ua}_{(x,t)}(n^2 s) ; \, (x,t)\in\mathcal{N}_n,\ s\geq \frac{t}{n^2} \right\} ~.
$$
\begin{theo}
\label{theo:ffw-cyl}
For $r=1/2$, the normalized CPT $\W^{(n),r,\ua}$ converges in distribution to the CBW as $n\to+\infty$.
\end{theo}

\begin{proof}
As noticed in Section \ref{section:CylBrownWeb}, only criteria $(IO)$ and $(B2O)$ of Th. \ref{theo:conv_cyl} have to be checked. The proof of $(IO)$ is very similar to the one of $(I)$ for the two-dimensional Poisson Tree (see Section 2.1 of \cite{ferrarifonteswu}) and is omitted. The suitable value $r=1/2$ ensures that the limiting trajectories are coalescing standard Brownian motions.

Let us now prove $(B2O)$. By stationarity of the CPT, it suffices to prove that for all $t>0$,
\begin{equation}
\label{B20-CPT}
\lim_{`e \to 0^+}\frac{1}{`e} \limsup_{n\to+\infty} \P \left(\eta^{O}_{\W^{(n),\frac{1}{2},\ua}}(0,t;[0\to \varepsilon]) \geq 3 \right) = 0.
\end{equation}
Recall that among all the trajectories in $\W^{(n),\frac{1}{2},\ua}$ that intersect the arc $[0\to `e]$ at time $0$, $\eta^{O}_{\W^{(n),\frac{1}{2},\ua}}(0,t;[0\to `e])$ counts the number of distinct positions these paths occupy at time $t$.

A first way to obtain (\ref{B20-CPT}) consists in comparing $\eta^{O}_{\W^{(n),\frac{1}{2},\ua}}(0,t;[0\to `e])$ and $\eta_{W^{(n)}}(0,t;0,`e)$, where $W^{(n)}$ denotes the normalized two-dimensional Poisson tree-- whose distribution converges to the usual BW, see \cite{ferrarifonteswu} --by using stochastic dominations similar to (\ref{eta^O<eta}) traducing that it is easier to coalesce on the cylinder than in the plane. Since $W^{(n)}$ satisfies $(B2)$ (see Section 2.2 of \cite{ferrarifonteswu}), $\eta^{O}_{\W^{(n),\ua}}(0,t;[0\to `e]) \leq_S \eta_{W^{(n)}}(0,t;0,`e)$ implies that $\W^{(n),\frac{1}{2},\ua}$ satisfies (\ref{B20-CPT}) which achieves the proof of Th. \ref{theo:ffw-cyl}.

A second strategy is to investigate the local character of the assumptions $(B2)$ and $(B2O)$. Indeed, the map $t\mapsto \eta^{O}_{\W^{(n),\frac{1}{2},\ua}}(0,t;[0\to `e])$ is a.s. non-increasing. It is then enough to prove (\ref{B20-CPT}) for (small) $0<t\ll 1$ in order to get it for any $t>0$. The same holds when replacing $\W^{(n),\frac{1}{2},\ua}$ with $W^{(n)}$. Now, when $t$ and $`e$ are both small, the (normalized) CPT $\W^{(n),\frac{1}{2},\ua}$ restricted to a small window containing $[0\to `e]\times[0;t]$ behaves like the (normalized) two-dimensional Poisson tree $W^{(n)}$ restricted to a window containing $[0;`e]\times[0;t]$ with high probability. As a consequence, $\W^{(n),\frac{1}{2},\ua}$ and $W^{(n)}$ should simultaneously satisfy $(B2O)$ and $(B2)$.

Let us write this in details. We use a coupling of the environment (the PPP) on some larger window since the trajectories of the discrete trees on a window are also determined by the environment around. Using some control of the deviations of the paths issued respectively from the intervals $I^{\Cyl}_\varepsilon=[0 \to \varepsilon ]\times \{0\}$ and $I_\varepsilon=[0;\varepsilon]\times \{0\}$, we determine larger windows $\mbox{Win}_\varepsilon^{\Cyl}$ and  $\mbox{Win}_\varepsilon$ which will determine the trajectories started from this sets to a certain time $t_\varepsilon$ up to a negligible probability $p_\varepsilon$. Using the constants that emerge from this study, we thereafter design a coupling between the PPP on the cylinder and on the plane that coincides on $\mbox{Win}_\varepsilon^{\Cyl}$ and  $\mbox{Win}_\varepsilon$ (up to a canonical identification). This will allow us to deduce $(B2)$ or $(B2O)$ from the other.

To design the windows that contains all paths crossing $I^{\Cyl}_\varepsilon$ (or $I_\varepsilon$) up to time $t$,  it suffices to follow the trajectories starting at $(0,0)$ and $(\varepsilon,0)$. Consider the path $(X_k=(x_k,y_k), k\geq 0)$ started from $(0,0)$ and consider the successive i.i.d. increments of this path
denoted by $(\xi^x_k,\xi^y_k)=\Delta X_{k}$. Before normalisation, $(\xi^x_1, \xi^y_1)$ consists of two independent r.v., where $\xi^x_1$ is uniform on $[-r,+r]$ with $r=1/2$, and $\xi^y_1$ has exponential distribution with parameter $\lambda=1$, since
$$
P(\xi^y_1 \geq y)=`P\Big(\mathcal{N} \cap \big(\big[-\frac{1}{2},\frac{1}{2}\big]\times [0,y]\big)  = \emptyset \Big) = e^{- y} ~.
$$

Now, starting at 0, the renormalized trajectory on $\W^{(n),\frac{1}{2},\ua}$ is a random walk whose increments $(\xi^{(n),x}_k, \xi^{(n),y}_k, k\geq 0)$ are i.i.d. such that $n\xi^{(n),x}_k \eqd \xi^x_1$, and $n^2\xi^{(n),y}_k\eqd \xi^y_1$. Let us define the number of steps for the rescaled path to hit ordinate $t$ by
$$
\tau^n_{t} := \inf \l\{ j\geq 1 ~|~ \sum_{k=1}^j \xi^{(n),y}_k \geq t \r\} \eqd \inf\l\{j\geq 1 ~|~ \sum_{k=1}^j \xi^{y}_k \geq n^2t\r\} ~.
$$
The points $\{\sum_{k=1}^j \xi^{y}_k,\ j\geq 1\}$ form a PPP $\Theta$ on the line with intensity $1$, so that $\tau^n_{t}=1+\#(\Theta \cap[0,n^2t])$ a.s. Therefore
$$
p_{c,t,n} := `P(\tau^n_t \geq cn^2) = `P(1+P(n^2t)\geq n^2c)
$$
where $P(x)$ is a Poisson r.v. with parameter $x$.
For $c=2t$  this probability $p_{2t,t,n}$ is exponentially small in $n$ and the event $A_{t,n}:=\{\tau_{n}^t\leq 2n^2 t\}$ has probability exponentially close to $1$. Now, on the event $A_{t,n}$, we can control the angular fluctuations of $\W^{(n),\frac{1}{2},\ua}$:
\ben
q_{t,n} & := & `P\l(\sup_{j\leq \tau^n_t} \l|\sum_{k=1}^j \xi^{(n),x}_k\r|\geq c \sqrt{t}\r)\\
& \leq & `P(A^c_{2t,n})+`P\l(\sup_{j\leq 2n^2 t} \l|\sum_{k=1}^{j} \xi^{(n),x}_k\r|\geq c \sqrt{t} \r) ~.
\een
Thus, consider the process defined by
$$
s_n(j/n^2) := \sum_{k=1}^{j} \xi^{(n),x}_k \, \eqd \, \frac{1}{n} \sum_{k=1}^{j } \xi^{x}_k ~, \quad \mbox{ for }j\geq 0,
$$
and interpolated in between. A simple use of Donsker theorem shows that
$$
(s_n(a))_{a\geq 0} \dd \l(\frac{1}{\sqrt{12}}B(a)\r)_{a\geq 0}
$$
in $\Co(\R_+,\R)$ where $B$ is a Brownian motion. Since for every $t$, on $\Co([0,2t],\R)$, the functional $g\mapsto \max|g|$ is continuous, one sees that
\ben
q_{t,n}&=& `P(A^c_{2t,t,n})+`P\l(\sup_{a\leq 2 t} \l|s_n(a)\r|\geq c \sqrt{t} \r)\\
&\xrightarrow[n\to +\infty]{}& `P\l(\sup_{a\leq 2 t} |B(a)| \geq c\sqrt{12 t}\r)= `P\l(\sup_{a\leq 1} |B(a)| \geq c\sqrt{6 }\r).
\een
Take $\varepsilon>0$. Choose $c$ large enough such that $`P\l(\sup_{a\leq 1} |B(a)| \geq c\sqrt{6}\r) \leq \varepsilon^2/2$, and $n$ large enough so that $q_{t,n}\leq \varepsilon^2$, and $t$ small enough so that $c\sqrt{t}<1/4$. We have proved that with probability larger than $1-O(\varepsilon^2)$, the walk hits  ordinate $t$ before its abscissa exits the window $[-c\sqrt{t},c\sqrt{t}]$. Since the decision sector for each step of the walker has width $2r/n$, with probability more than $1-O(\varepsilon^2)$, the union of the decision sectors of the walk before time $t$ are included in
\ben[-c\sqrt{t}-2r/n,c\sqrt{t}+2r/n]\subset[-1/3,1/3]\een for $n$ large enough.
It is now possible to produce a coupling between the PPP on the cylinder and the plane that coincides on a strip with width $2/3$~: take the same PPP on the two strips (up to a canonical identification of these domains), and take an independent PPP with intensity 1 on the remaining of the cylinder or of the plane. Henceforth, any computation that depends only of such a strip in the cylinder and in the plane will give the same result. Here, we then have here, for any event $\mbox{Ev}$ that depends on the trajectories passing through $I_\varepsilon^{\Cyl}$ or $I_\varepsilon$ up to time $t$ (for the constant satisfying what is said just above)
\ben
`P_{n}^{\Cyl}(\mbox{Ev}) = `P_{n}(\mbox{Ev})+ O(\varepsilon^2),
\een
so that the inheritance of $(B2)$ from the plane to the cylinder is guaranteed, as well as the converse.
\end{proof}

\subsection{From the plane to the cylinder, and vice-versa: principles}
\label{sec:FCP}

When a convergence result of some sequence of coalescing processes defined on the plane to the BW has been shown, it is quite natural to think that the similar convergence holds on the cylinder too, and that the limit should be the CBW. The converse, also, should hold intuitively.

The main problem one encounters when one wants to turn this intuition into a theorem, is that, in most cases the constructions we are thinking of are trees that are defined on random environments (RE) as a PPP or as lattices equipped with Rademacher r.v.. Both these models exist on the cylinder and on the plane, leading to clear local couplings of these models. But, more general RE and more general random processes exist, and it is not possible to define a ``natural'' model on the cylinder inherited from that of the plane. We need to concentrate on the cases where such a natural correspondence exists.

A similar restriction should be done for the algorithms that build the trajectories using the RE. In the cases studied in the paper,
the trajectories are made by edges, constructed by using a navigation algorithm, which decides which points to go to depending on a ``decision domain'' which may depend on the RE. For example, in the cylindric lattice web, the walker at position $(x,t)$ just needs to know the Rademacher variable attached to this point, so that its decision domain is the point $(x,t)$ itself. In the generalization of  Ferrari \& al. \cite{ferrarilandimthorisson,ferrarifonteswu} treated at the beginning of Section \ref{sect:ConvCBW}, the decision domain is a rectangle  $[x-r,x+r]\times(t,t+h]$ where $h$ is smallest positive real number for which this rectangle contains a point of the point process (many examples of such navigation processes have been defined in the literature, see \cite{baccellibordenave, baccellicoupiertran,B-M, colettidiasfontes,colettivalencia,colettivalle,fontes2004,FVV}).
We may call such model of coalescing trajectories as coming from ``navigation algorithms, with local decision domains''.

There exist models of coalescing random processes of different forms, or that are not local (such as minimal spanning trees). Again, it is not likely that one may design a general theorem aiming at comparing the convergence on the cylinder with that on the plane.

``For a model defined on the cylinder and on the plane on a RE'' as explained in  the proof of Theorem \ref{theo:ffw-cyl}, when a local coupling between windows (or strip) of the cylinders and of the plane exists, $(B2)$ and $(B2O)$ {``are morally equivalent''}.
Informally, the 4 conditions are:  \\
1) the models are invariant by translations on respectively, the cylinder and the plane;\\
\noindent 2) there exists a coupling between both probabilistic models which allows to compare  $\W^{(n),\ua}$ and $W^{(n)}$ at the macroscopic level: on a window $\mbox{Win}:=[0,A]\times [0,B]$ for some (small) $A,B>0$, the environments on which are defined  $\W^{(n),\ua}$ and $W^{(n)}$ can be coupled, and, under these coupling, these RE  coincide a.s.;\\
\noindent 3) the restriction of the trajectories from $\W^{(n),\ua}$ and $W^{(n)}$ on $[0,\varepsilon]\times[0,t_\varepsilon]$ are measurable with respect to the environment in $\mbox{Win}$ with probability $1-O(\varepsilon^{1+a})$ for some $a>0$;\\
\noindent 4) the largest decision domain before hitting ordinate $n^2t$ is included in a rectangle $[a_n,b_n]$ with probability $1-O(\varepsilon^{1+a})$ where $a_n=o(1)$ and $b_n=o(1)$ (for the rescaled version).

\section{Discrete Cylindric and Radial Poisson Tree}
\label{sec:RBW}

Coletti and Valencia introduce in \cite{colettivalencia} a family of coalescing random paths with radial behavior called the \textit{Discrete Radial Poisson Web}. Precisely, a Poisson point process $\Theta$ with rate 1 on the union of circles of radius $k\in \N\setminus \{0\}$, centered at the origin, is considered. Each point of $\Theta$ in the circle of radius $k$ is linked to the closest point in $\Theta$ in the circle of radius $k-1$, if any (if not, to the closest point of $\Theta$ in the first circle of radius smaller than $k-1$ which contains a point of $\Theta$). They show in \cite[Th.2.5]{colettivalencia} that under a diffusive scaling and restricting to a very thin cone (so that the radial nature of paths disappears), this web converges to some mapping of the (standard) BW. A similar result is established in Fontes et al. \cite{FVV} for another radial web.

Our goal in this section is to establish a convergence result for an analogous of the Discrete Radial Poisson Web of Coletti and Valencia \cite{colettivalencia} but which holds in the whole plane. Our strategy consists in considering a cylindrical counterpart to the Discrete Radial Poisson Web and to prove its convergence to the CBW (Theorem \ref{theo:DCPFtoCBW}). Thenceforth, it suffices to map the cylinder on the (radial) plane with $\varphi_\star$ defined in \eqref{def:phi} to obtain a global convergence result for the corresponding planar radial forest.\\

We modify a bit the model of \cite{colettivalencia} to make the involved normalizations more transparent and to reduce as much as possible the technical issues, while keeping at the same time the main complexity features. Consider an increasing sequence of non-negative numbers $(h_k, k \in \N)$, with $h_0=0$, and the associated slices of the cylinder:
\begin{equation}
\label{def:CSk}
\Cyl' = \rpuz \times \{h_k, \ k\in \N\}=\bigcup_{k\in \N} \Slice(h_k) ~.
\end{equation}
Consider the following Poisson point process on $\Cyl'$,
\begin{equation}
\label{def:Xi}
\Xi=\bigcup_{k\geq 0} \Xi_k ~,
\end{equation}
where  $\Xi_k$ is a PPP on $\Slice(h_k)$ with intensity $n_k>0$. The sequences $(h_k)_{k\geq 1}$ and $(n_k)_{k \geq 1}$ are the parameters of the model. Remark that the choice of $n_k=n$ (a constant) is treated in previous sections. Here we are interested in the case where $n_k, h_k\rightarrow +\infty$.\\

Given $\Xi$, let us define the ancestor $\alpha(Z)$ of a point $Z=(x,h_k)\in \Slice(h_k)$ as the closest point of $\Xi_{k+1}$ if the latter is not empty and the point $(x,h_{k+1})$ otherwise. This second alternative means that instead of moving to the closest point of the first non-empty slice with rank $k'>k$ (as in \cite{colettivalencia}), one just moves vertically to the next slice.

The ancestor line ${\sf AL}_Z$ of $Z=(x,h_k)$ is the sequence $(Z_j=(x_j,h_j),j\geq k)$ such that $Z_k=Z$ and for $j> k$, $Z_{j+1}=\alpha(Z_j)$. Upon $\Xi$ we define the \textit{Discrete Cylindric Poisson Tree} $\mathcal{T}$ as the union of the ancestor lines of the elements of $\Xi$:
$$
\mathcal{T} := \bigcup_{(x,h) \in \Xi} {\sf AL}_{(x,h)} ~.
$$Notice that when $(x,h)\in \Xi$, ${\sf AL}_{(x,h)}=\mathcal{T}_{(x,h)}$ is the path of $\mathcal{T}$ started at $(x,h)$. The notation ${\sf AL}_{(x,h)}$ allows to consider ancestor lines started from any points $Z\in \Cyl'$.

Contrary to Section \ref{sect:ConvCBW}, we do not consider a sequence of point processes parametrized by $n$ which goes to infinity, but rather we shift the cylinder which also implies that we see more and more points. Precisely, for any $k\geq j\geq 1$ and any $(x,h_k)\in\Xi$, let ${\sf AL}^{(j)}_{(x,h_k)}$ be the ancestor line ${\sf AL}_{(x,h_k)}$ translated by the vector $-(0,h_j)$. We can then associate to $\mathcal{T}$, the sequence of shifted forests $(\mathcal{T}^{(j)})_{j\geq 1}$ by
$$
\mathcal{T}^{(j)} := \bigcup_{(x,h) \in \dis\cup_{k\geq j} \,\Xi_k} {\sf AL}^{(j)}_{(x,h)} ~.
$$
\begin{figure}[!ht]
\begin{center}
\begin{tabular}{ccc}
\includegraphics[width=5cm,height=5cm]{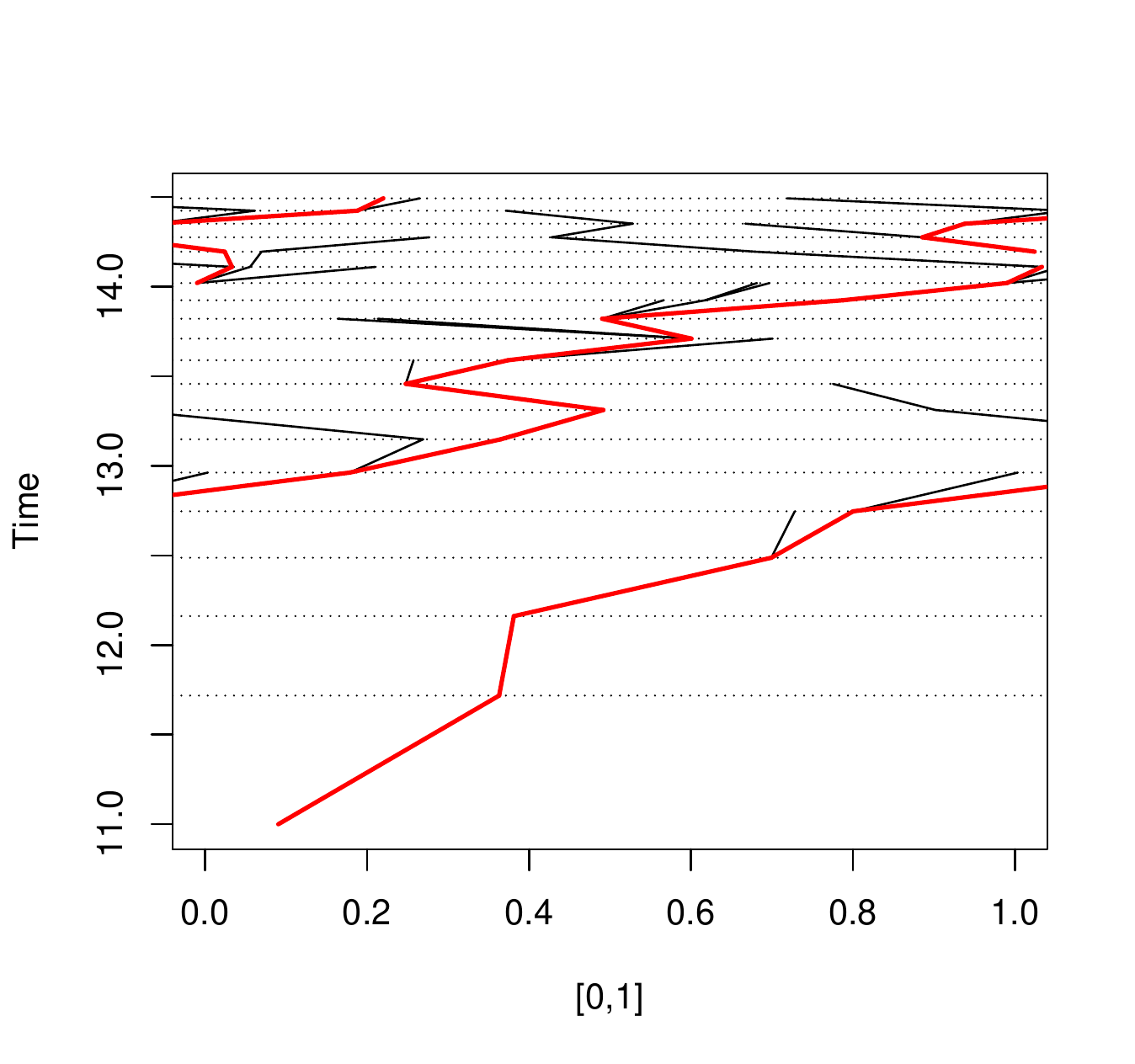} & \includegraphics[width=5cm,height=5cm]{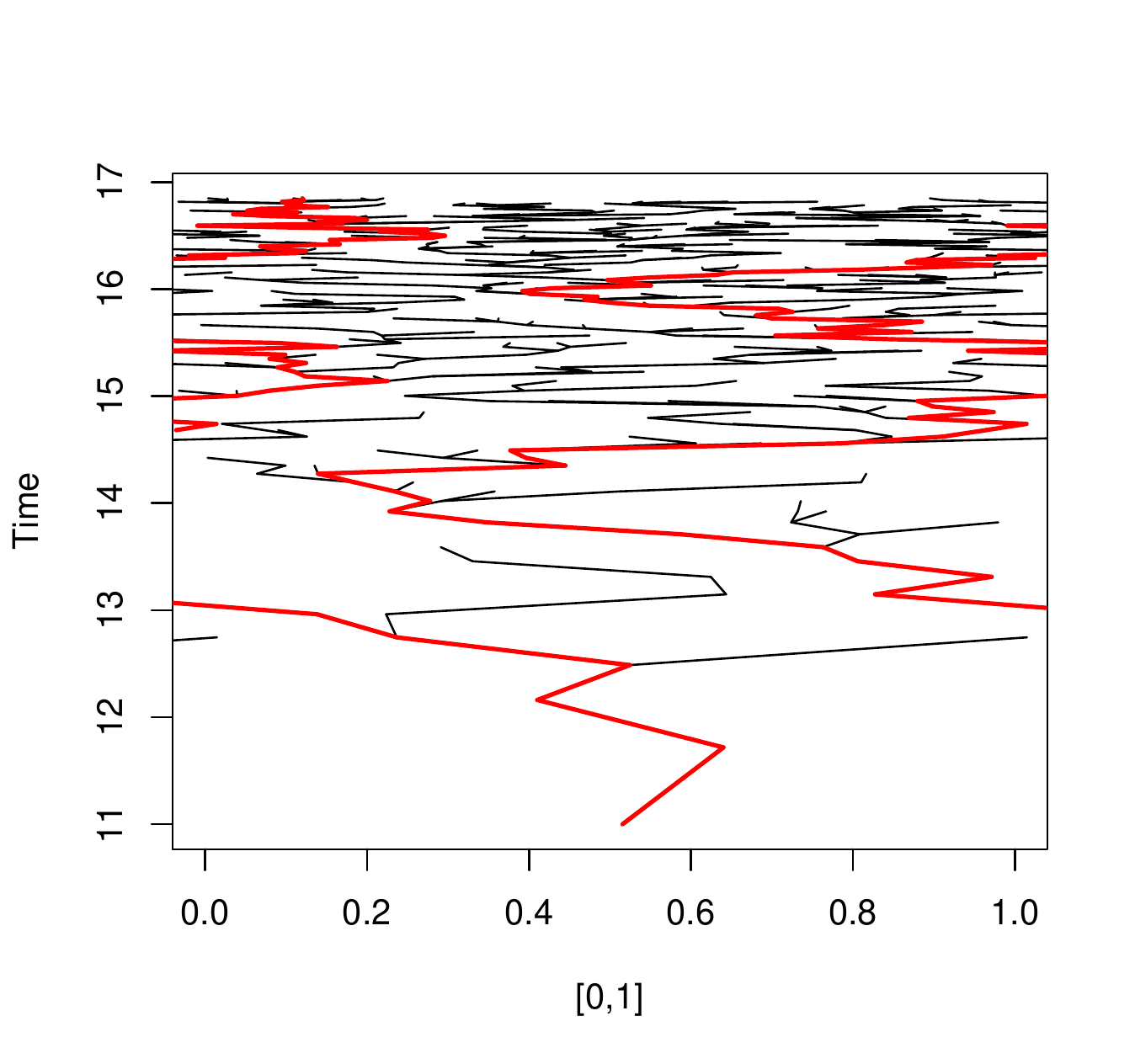} & \includegraphics[width=5cm,height=5cm]{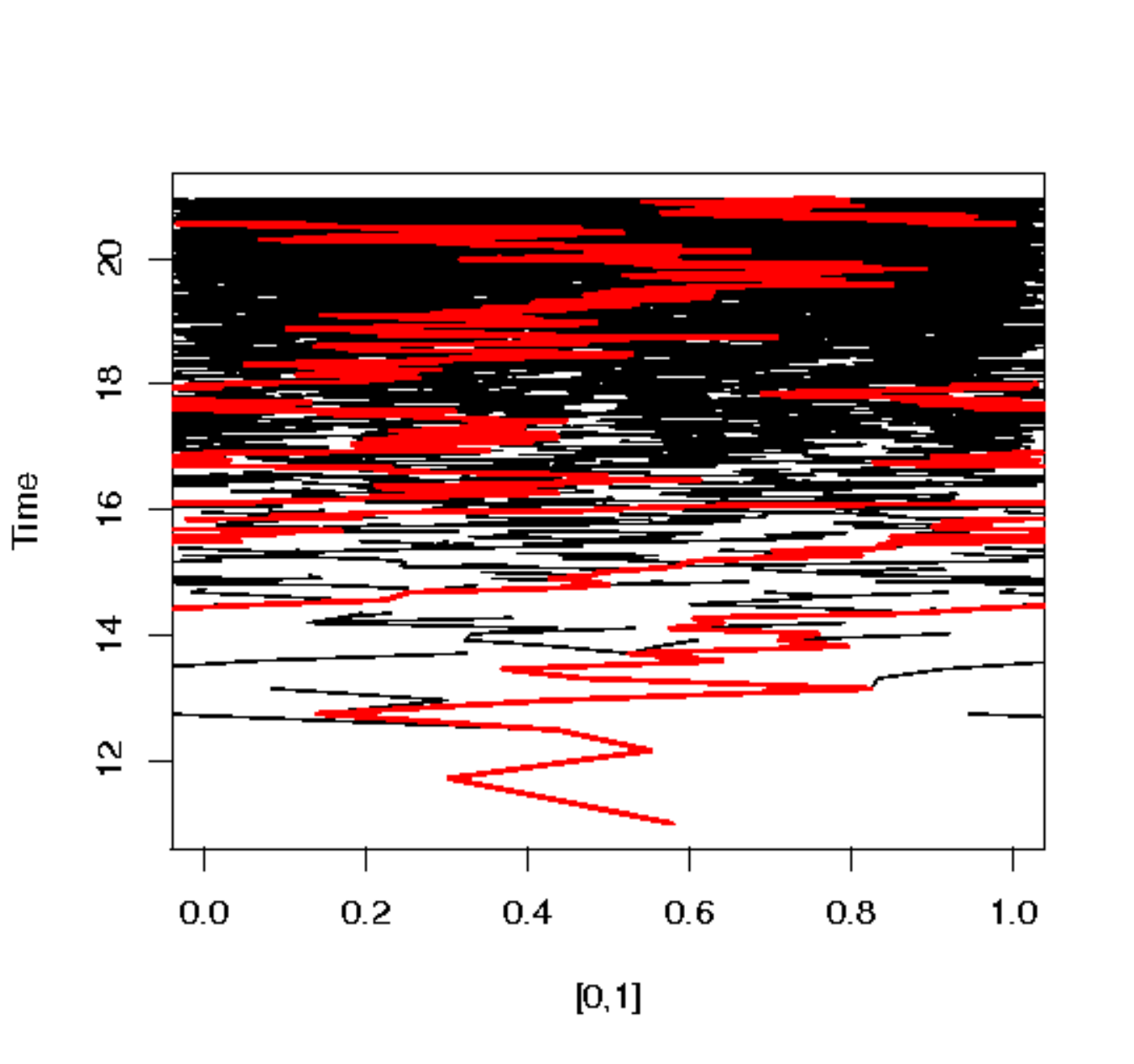}\\
\includegraphics[width=5cm,height=5cm]{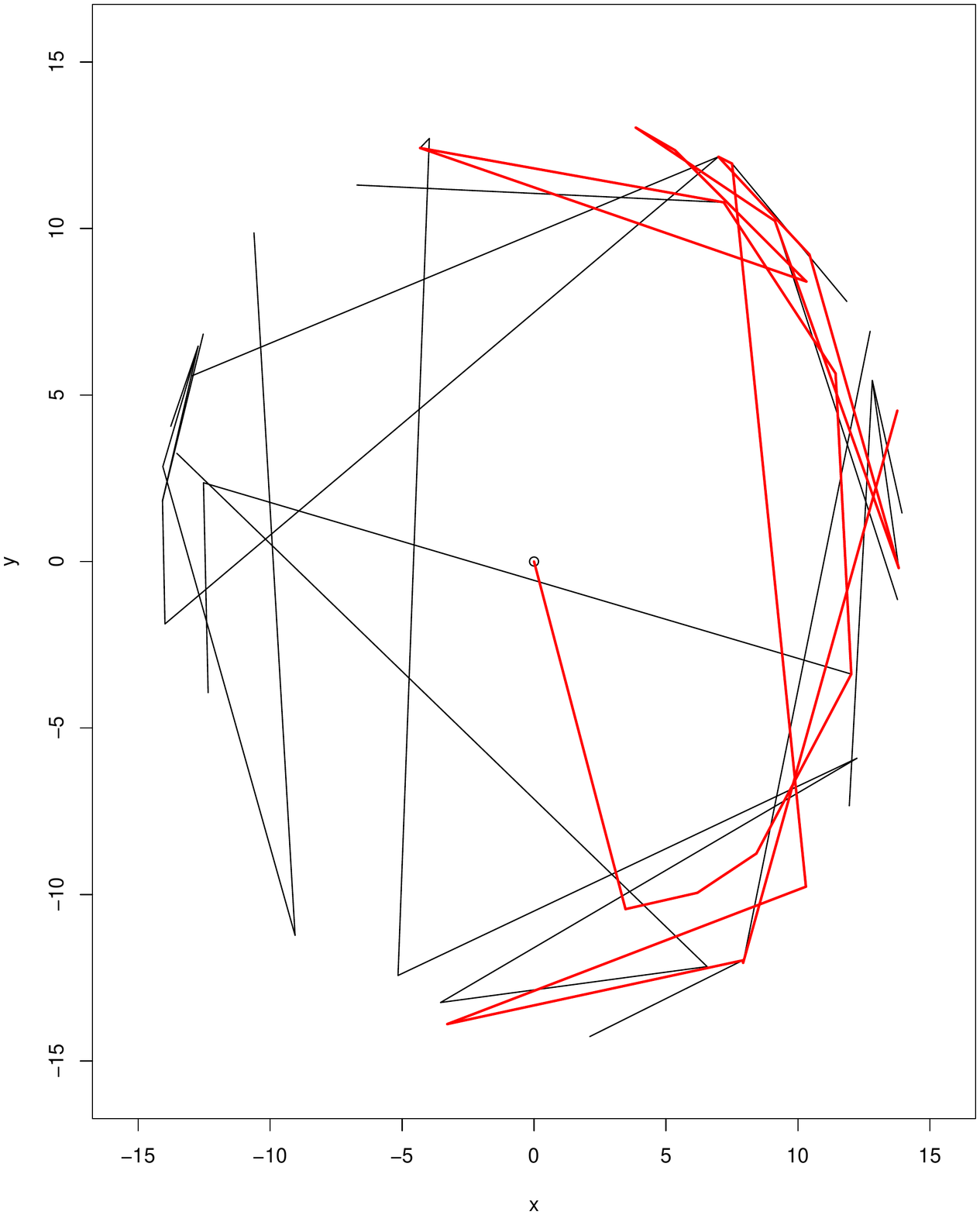} & \includegraphics[width=5cm,height=5cm]{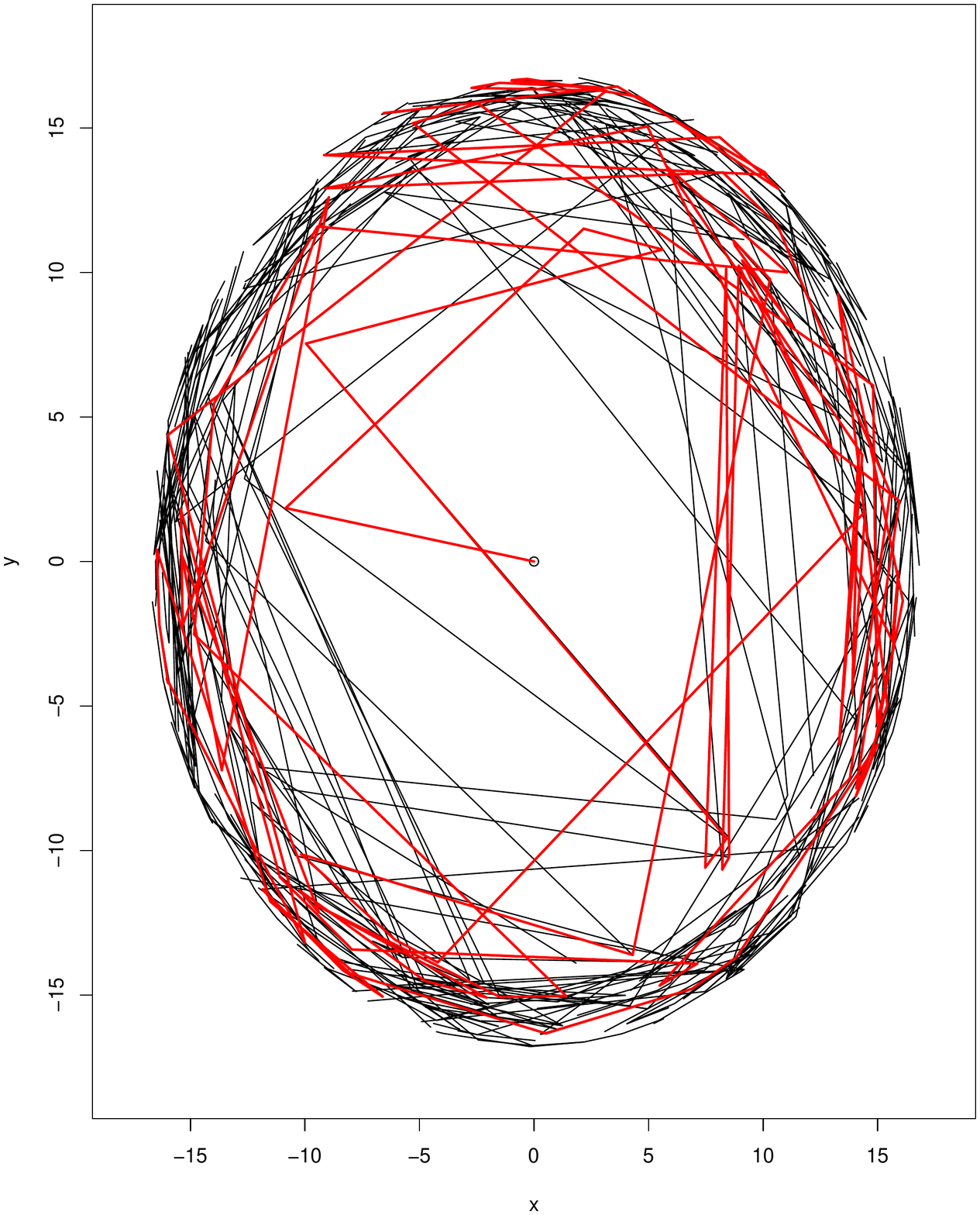} & \includegraphics[width=5cm,height=5cm]{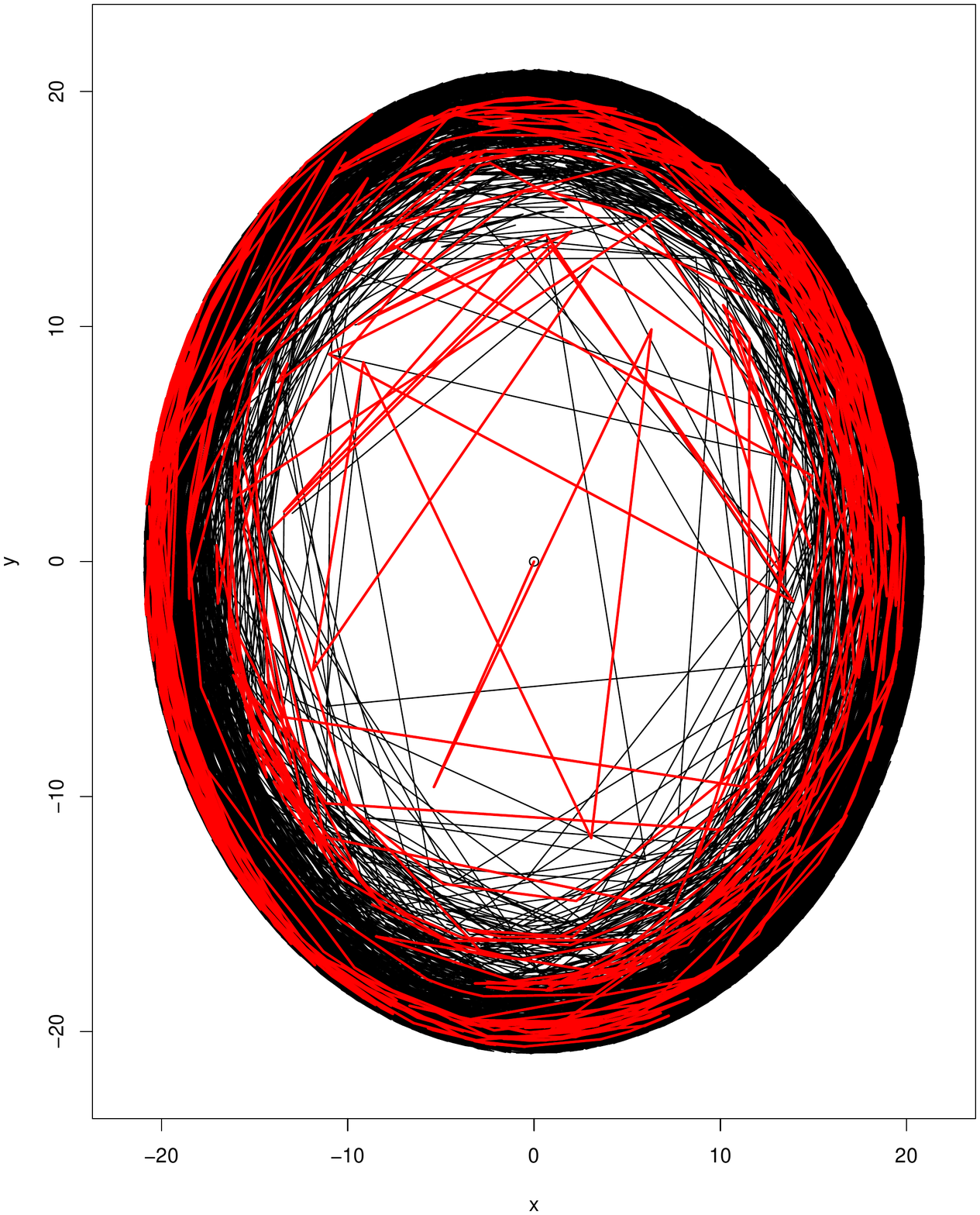}\\
(a) & (b) & (c)
\end{tabular}
\caption{\label{fig:CBW} {\small \textit{First line: Cylindrical forests $\mathcal{T}$ represented for $(h_k)_{k\in \{1\dots K\}}$ with (a) $K=20$, (b) $K=100$, (c) $K=1000$. The red path is obtained by picking a vertex at height $h_K$ at random and by finding its ancestors in the upper slices. Second line: Projected radial forests obtained by $\varphi_\star$ defined in \eqref{def:phi} and with the same values of $K$ are represented.}}}
\end{center}
\end{figure}

Our purpose is to prove that:
\begin{theo}
\label{theo:DCPFtoCBW}
Let us consider two sequences $(n_k)_{k \geq 1}$ and $(f_k)_{k \geq 1}$ of positive real numbers such that
\ben
\label{eq:newcond}
\lim_{k\to\infty} f_k = 0 \, , \; \sum_k e^{-n_kf_k} < +\infty \; \mbox{ and } \; \sum_{k} \frac{1}{n_k^2} = +\infty ~.
\een
Then there exists a sequence $(h_k)_{k \geq 1}$ tending to infinity such that the sequence of shifted forests $(\mathcal{T}^{(j)})_{j\geq 1}$ converges in distribution to the CBW restricted to the half cylinder $\Cyl^+=\rpuz \times \R_+$.
\end{theo}
The map $\varphi_\star$, as defined in \eref{def:phi}, sends the  half cylinder $\Cyl^{+}:=(\rpuz)\times `R^{+}$ onto the radial plane $(\rpiz)\times `R^+$. The image of the PPP $\Xi$ is a PPP $\Xi'$ on the plane, which is the superposition of the Poisson point processes $\Xi'_k$ of intensities $n_k/(8\pi^3h_k)$ on the circles with radii $4\pi^2h_k$. The image of the tree we built on the cylinder is a tree on ``the radial plane'', which can in fact be directly built by adapting the navigation used in the cylinder in the plane (go to the closest point in the next circle if any, and otherwise to the point with same argument). See Figure \ref{fig:CBW}. To get a convergence result on the radial plane, with the same flavour as that obtained in the plane, we need to discard the neighborhood of zero by a shift. The most economic way to state our results is as an immediate corollary of the previous result:
\begin{cor} \label{cor:rtheer}
Under the hypotheses of Theorem \ref{theo:DCPFtoCBW}, when $j\to+\infty$, the sequence $(\varphi_\star(\mathcal{T}^{(j)}))_{j\geq 1}$ converges in distribution to the $\RBWfromc$.
\end{cor}

Let us comment on the hypothesis \eref{eq:newcond}. First, it implies that $n_{k}\to\infty$. A consequence of Borel-Cantelli's lemma is that whenever $\sum_{k\geq 0}e^{-n_k}$ is finite, there exists a.s. a random rank from which the $\Xi_k$'s are non-empty. Hypothesis \eref{eq:newcond} is actually slightly more demanding: the condition $\sum_{k} 1/n_k^{2}=+\infty$ and the link between sequences $(n_k)_{k \geq 1}$ and $(f_k)_{k \geq 1}$ will appear in Sections \ref{sect:Cv1path} and \ref{sect:ProofIO}. \\

To prove Theorem \ref{theo:DCPFtoCBW}, we check the criteria of Theorem \ref{theo:conv_cyl}, namely $(IO)$ and $(EO)$. The convergence of an ancestor line of $\mathcal{T}^{(j)}$ to a BM modulo 1, when $j\rightarrow +\infty$, is first stated in Section \ref{sect:Cv1path}. In the proof, we see that the condition $\sum_{k} 1/n_k^{2}=+\infty$ of \eqref{eq:newcond} is necessary. We then deduce $(IO)$ in Section \ref{sect:ProofIO}, and use at some point the second item of \eqref{eq:newcond}. The proof of $(EO)$ is devoted in Section \ref{sect:ProofEO} and is based on a coalescence time estimate (Proposition \ref{pro:estimtpscoal} in Section \ref{sec:B10poisson}) whose proof uses the links between the cylindric and planar forests highlighted in Section \ref{sec:FCP}.

\subsection{Convergence of a path to a Brownian motion}
\label{sect:Cv1path}

Let us consider the ancestor line $\mathcal{T}_{(X_j,h_j)}=((X_k,h_k), k\geq j)$ started at a point $(X_{j},h_{j})\in\Xi_j$ for $j\geq 1$. For $k>j$, this path goes to infinity by jumping from $\Slice(h_{k-1})$ to $\Slice(h_{k})$.
The random increments $(\Delta X_k:=X_k-X_{k-1},k >j)$ are independent.
The distribution of $\Delta X_k$ is characterized for any measurable bounded map $f:[-1/2,1/2]\to\R$ by,
\begin{equation}
\E(f(\Delta X_k))=e^{-n_k}f(0)+\int_{-1/2}^{1/2} f(x) n_k e^{-2n_k |x|}\ dx.\label{def:xik}
\end{equation}
In other words, conditionally on the $\Slice(h_k)$ being not empty, $\Delta X_k$ is a Laplace r.v. conditioned on having absolute value smaller than $1/2$.
Hence,
\ben
\label{eq:displ}
\left\{\begin{array}{rrl}
\E(\Delta X_k)&=&0\\
\Var(\Delta X_k)&=&\frac{1}{2n_k^2}-\frac{e^{-n_k}(n_k^2+2n_k+2)}{4n_k^2} =: \sigma_k^2\\
\E((\Delta X_k)^4)&=& \frac{3}{2n_k^4}-\frac{e^{-n_k}(n_k^4+4n_k^3+12n_k^2+24n_k+24)}{16n_k^4}.
\end{array}\right.
\een
As $n_k\to\infty$, the variance $\sigma_k^2$ is equivalent to $1/(2n_k^{2})$. For the sequel, let us denote the variance of $X_k-X_0$ by:
\begin{equation}
V_k := \Var(X_k-X_0) = \sigma_1^2+\cdots+ \sigma_k^2.
\label{def:Vk}
\end{equation}
The variance $V_k$ is hence related to $n_k$ by \eqref{eq:displ}.\\

Let us now consider the time continuous interpolation of the shifted sequence $(X_k,h_k-h_j)_{k\geq j}$. For $\ell \in \N$, we set, if $h_j+t\in [h_{j+\ell}, h_{j+\ell+1})$,
\begin{equation}
\bar{X}_t^{(j)}= X_{j}+ \sum_{k=1}^{\ell} \Delta X_{j+k} + \frac{(h_j+t)-h_{j+\ell}}{h_{j+\ell+1}-h_{j+\ell}}
\ \Delta X_{j+\ell+1} ~.
\label{shifted_interpolation}
\end{equation}
In order to prove the convergence of $(\bar{X}^{(j)})_{j\geq 0}$ to a Brownian motion, it is natural to set
\ben
\label{eq:hV}
h_k = V_k , \textrm{ for any } k\geq 0 ~.
\een
Then, combining \eref{eq:displ}, \eref{def:Vk}, \eref{eq:hV} with \eref{eq:newcond}, it follows that $h_k\to\infty$ and $h_{k+1}-h_k=\sigma^2_{k+1}\rightarrow 0$, i.e. slices are getting closer and closer.
\begin{exm}
The hypothesis \eref{eq:newcond} is satisfied for example for $n_k=k^{\alpha}$ with $0<\alpha<1/2$ and $f_k=k^{\alpha'}$ with $0<\alpha'<\alpha$. This entails that $h_k=c_1+c_2 k^{1-2\alpha}$ where $c_1,c_2$ are positive constants.
\end{exm}

Let us introduce, for the sequel,
\begin{equation}
\label{def:Rt}
R(t) := \inf \{k\in \N,\ V_k\geq t\}
\end{equation}
the integer index such that $h_{R(t)-1}< t \leq h_{R(t)}.$ Note that $R(h_k)=k$.

\begin{lem}
\label{lem:convMB}
Under the previous notations and \eref{eq:newcond}, the following convergence holds in distribution in $\Co(\R_+,\R)$
\ben
(\overline{X}^{(j)}_t,t\geq 0) \xrightarrow[j\to +\infty]{(d)} (B_t,t\geq 0) ~,
\een
where $(B_t,t\geq 0)$ is a standard Brownian motion taken modulo $1$.
\end{lem}

\begin{proof}[Proof of Lemma \ref{lem:convMB}]
We are not under the classical assumptions of Donsker theorem, since the $\Delta X_k$'s are not identically distributed and since the convergence involves a triangular array because of the shift.
Because the $\Delta X_k$'s are independent, centered, with a variance in $1/n_k^2$ that tends to 0, we have for all $t\geq 0$,
$$ \lim_{j\rightarrow +\infty}\Var\big(\bar{X}^{(j)}_t-\bar{X}^{(j)}_0\big)=
\lim_{j\rightarrow +\infty}\sum_{\ell = j+1}^{R(h_j+t)}\Var(\Delta X_\ell)=
\lim_{j\rightarrow +\infty}\Big(V_{R(h_j+t)}-V_j\Big)=t,$$implying that $\bar{X}^{(j)}_t-\bar{X}^{(j)}_0$ converges in distribution to $\mathcal{N}(0,t)$ by Lindeberg theorem (e.g. \cite[Th. 7.2]{billingsley1968}). The convergence of the finite dimensional distributions is easily seen by using the independence of the $\Delta X_k$'s.
\par The tightness is proved if (see e.g. \cite[Th. 8.3]{billingsley1968}) for every positive $\varepsilon>0$ and $\eta>0$, there exists $\delta\in (0,1)$ and $j_0\in \N$ such that for every $j\geq j_0$ and every $t\in \R_+$,
\begin{equation}
\frac{1}{\delta} \P\Big(\sup_{t\leq s\leq t+\delta} \big| \bar{X}^{(j)}_s-\bar{X}^{(j)}_t\big|\geq \varepsilon\Big)\leq \eta.\label{goal:tightness}
\end{equation}
For $t\in \R_+$ and $j\in \N$,
\begin{align*}
\P\Big(\sup_{t\leq s\leq t+\delta} \big| \bar{X}^{(j)}_s-\bar{X}^{(j)}_t\big| \geq \varepsilon\Big) = &
\P\Big(\sup_{t\leq s\leq t+\delta} \big| \bar{X}^{(j)}_s-\bar{X}^{(j)}_t\big|^4\geq \varepsilon^4\Big)\\
\leq & \P\Big( \max_{R(h_j+t)\leq \ell\leq R(h_j+t+\delta) }  S_\ell^4   \geq \varepsilon^4 \Big),
\end{align*}where$$
S_\ell=\frac{h_{R(h_j+t)}-(h_j+t)}{h_{R(h_j+t)}-h_{R(h_j+t)-1}}\ \Delta X_{R(h_j+t)}+ \sum_{k=R(h_j+t)+1}^{\ell} \Delta X_k.$$
Since the $\Delta X_k$ are centered, $S_\ell$ defines a martingale (in $\ell\geq R(h_j+t)$), and $S_\ell^4$ is a submartingale. Using Doob's lemma for submartingales:
\begin{multline}
\varepsilon^4 \P\Big( \max_{R(h_j+t)\leq \ell\leq R(h_j+t+\delta)} S_\ell^4 \geq \varepsilon^4 \Big) \leq  \E\big(S^4_{R(h_j+t+\delta)}\big)\\
\leq  \E\Big(\sum_{k=R(h_j+t)}^{R(h_j+t+\delta)} (\Delta X_k)^4\Big)+\sum_{\stackrel{k\not= \ell}{R(h_j+t)\leq k,\ell\leq R(h_j+t+\delta)}} \E((\Delta X_k)^2)\E((\Delta X_\ell)^2),\label{eq:etape5}
\end{multline}using that the $\Delta X_k$'s are independent and centered. The last sum in the r.h.s. of \eqref{eq:etape5} is upper bounded by
$$\left(\sum_{k=R(h_j+t)}^{R(h_j+t+\delta)} \E\big((\Delta X_k)^2\big)\right)^2$$
that converges to $\delta^2$ when $j\rightarrow +\infty$. For the first term, there exists from \eqref{eq:displ} a constant $C$ such that for $k$ large enough, $\E((\Delta X_k)^4)\leq C \Var(\Delta X_k)^2 = C \sigma_k^4.$ Thus:
\begin{align*}
\E\l(\sum_{k=R(h_j+t)}^{R(h_j+t+\delta)} (\Delta X_k)^4\r) \leq & C M_{j,t,\delta}
  \sum_{k=R(h_j+t)}^{R(h_j+t+\delta)} \sigma_k^2 \sim_{j\rightarrow +\infty} C M_{j,t,\delta} \delta
\end{align*}
with $M_{j,t,\delta}=\sup\{\sigma_k^2, ~R(h_j+t)\leq k \leq R(h_j+t+\delta) \} \to 0 $ when $j\to+\infty$.
 Gathering these results, we see that up to a certain constant $C$,
$$\frac{1}{\delta} \P\Big(\sup_{t\leq s\leq t+\delta} \big| \bar{X}^{(j)}_s-\bar{X}^{(j)}_t\big|\geq \varepsilon\Big)\leq \frac{C}{\varepsilon^4 \delta}\big(\delta^2 + CM_{t_1,t,\delta} \delta \big),$$which converges to zero when $\delta\rightarrow 0$ and $j\rightarrow +\infty$.
\end{proof}

\subsection{Coalescence time estimate}
\label{sec:B10poisson}

In this section, we establish a coalescence time estimate that will be useful for proving (IO) and (EO). Following the lines of Section \ref{sec:FCP}, we can introduce a planar model corresponding to our cylindrical tree and ensuring the possibility of couplings between the cylinder and the plane. In the plane, the use of the Skorokhod embedding theorem and the results known for planar Brownian motions make it easier to obtain such estimates.
We thus first introduce in Section \ref{section:plan1} a planar model corresponding to the forest $\mathcal{T}$. We establish estimates for the coalescence time of two paths in this planar model. For this, we start with studying how the distance between the two paths evolves. The core of the proof relies on the Skorokhod embedding theorem (as in \cite{colettidiasfontes}), but with a clever preliminary stochastic domination of the distance variations. In Section \ref{section:plan-to-cylindre}, we return to the original model and deduce from the previous result estimates for the coalescence time of two paths of $\mathcal{T}^{(j)}$.

\subsubsection{Planar analogous}\label{section:plan1}
 We first define the planar model corresponding to our cylindrical problem. We consider the horizontal lines with ordinate $(h_k)_{k\in \N}$ in the upper half plane. For each $k\in \in \N$, we consider on the line $L_k:=\R\times \{h_k\}$ (or level $h_k$), an independent PPP $\Upsilon_k$ with intensity $n_k$. The Poisson point process on the union of the lines is denoted by $\Upsilon$, similarly to \eqref{def:Xi}. Each point of the level $h_k$ is linked with the closest point of the next level, namely level $h_{k+1}$.
This generates a forest that we denote $\mathcal{W}$, and which can be seen as the analogous of $\mathcal{T}$ in the plane.\\

For a given point $Z\in \cup_{k\in \N}L_k$, denote by $\alpha^P(Z)$ the ancestor of $Z$ for this navigation. This allows us to define, as for the cylinder, the ancestor line  ${\sf AL}^{P}_{Z}$ of any element $Z\in \R\times \R_+$.

The aim of this section is to provide an estimate on the tail distribution of the hitting time between the ancestor lines started from two points at a distance $d>0$ on the line $L_0$. Without restriction, we consider $Z=(0,0)$ and $Z'=(d,0)$, and denote their ancestor lines by $(Z_k, k\geq 0)$ and $(Z'_k, k\geq 0)$.
Let us denote by
\[D_k= Z_k'-Z_k,\]
 the distance between the two paths at level $h_k$. The result proved in this section is the following:
\begin{pro}
\label{pro:estimtpscoal}Let us define by $\tau=\inf\{k\in \N, D_k= 0\}$. There exists $C>0$ such that for $K\in  \N\setminus \{0\}$,
\begin{equation}
\P(\tau>K) \leq \frac{C d}{\sqrt{h_K}}.
\end{equation}
\end{pro}

The remaining of the section is devoted to the proof of Prop. \ref{pro:estimtpscoal}.
In the proof, we will also need the following quantity for $k\geq 1$:
$$\Delta_k= D_k-D_{k-1}.$$

The proof is divided into several steps. For the first step, we consider a PPP with an intensity constant and equal 1. In doing so, we introduce a sort of companion model that will help finding estimates for the planar model considered above.
 We will then proceed to the control of the hitting time of two ancestors lines, by using some rescaling properties.

\subsubsection*{Step 1: Evolution of the distance in one step, when the intensity is 1 }

Take two points $Z=(0,0)$ and $Z'=Z(d,0)$ at distance $d$ in $L_0$.
Assume that the PPP $\Upsilon_1$ on $L_1$ has intensity 1: let $(X_1,h_1)=\alpha^P(Z)$ be the closest point to $Z$ in $\Upsilon_1$ and by $(X_2,h_1)=\alpha^P(Z')$ the closest point to $Z'$ in $\Upsilon_1$. Let
\[D(d)=X_2-X_1=|\alpha^P(Z')-\alpha^P(Z)|\]
be the ``new distance'', and denote by
\[\Delta(d)=D(d)-d\] the variation of the distance between the levels $L_0$ and $L_1$.
\begin{pro}\label{pro:repre}
The distribution $\mu_d$ of $\Delta(d)$ is the following probability measure on $\R$
\ben
\mu_d(du) = (d+1)e^{-2d} \delta_{-d}(du)+ f_d(u) \ind_{[-d,d]}(u) \ du+de^{-u-d} \ind_{[d,+\infty)}(u) \ du ,\label{def:mud}
\een
where
\ben f_d(u)=-\frac{e^{-2d}}{2} +e^{-2|u|}\big(|u|+\frac{1}{2}\big).
\een
The atom   $m_d= (d+1)e^{-2d}$ of $\mu_d $ at $-d$ corresponds to case where coalescence occurs, that is $\alpha^P(Z)=\alpha^P(Z')$. Apart from this atom, $\mu_d$ is absolutely continuous with respect to the Lebesgue measure, and
\ben\label{eq:D}
\E(\Delta(d))&=&0 \quad \textrm { and }\quad \E(D(d))=d\\
\label{eq:VVVV}
\Var(\Delta(d))&=&V(d)=1-e^{-2\,d}+\frac{2}3\,e^{-2\,d}{d}^{3}+e^{-2\,d}{d}^{2}.
\een
\end{pro}
The proof is postponed to the end of the section.

Notice that the distribution of $\Delta(d)$ does not depend on the height $h_1$ of the level $L_1$, and that $f_d$ is a symmetric function.

\subsubsection*{Step 2: The sequence $(D_k,k\geq 0)$ is a martingale}

Denote by $D^{\lambda}(d)=X_2-X_1=\alpha^P(Z')-\alpha^P(Z)$ the new distance if the PPP $\Upsilon_1$ has intensity $\lambda>0$, and by $\Delta^\lambda(d)=D^\lambda(d)-d$. A simple scaling argument allows one to relate the distribution of $\Delta^\lambda(d)$ to that of $\Delta^1(d)\eqd\Delta(d)$~:
\ben\label{eq:efgz}
\Delta^{\lambda}(d)\eqd \frac{1}{\lambda} \Delta^1(\lambda \,  d),\quad \mbox{ for }d\geq 0,
\een
Now, since the intensity on $L_k$ is $n_k$, conditional on $D_{k-1}$,
\ben\label{eq:efgzk}
\bpar{ccl}
D_k&\eqd& D^{n_k}(D_{k-1})\\
\Delta_k &\eqd&\Delta^{n_k}(D_{k-1})\eqd \dis\frac{1}{n_k} \Delta^1( D_{k-1}\,n_k).
\epar
\een

Because of \eqref{eq:D}, $(D_k, k\geq 0)$ defines a martingale. The particular form of  $\mu_{d}$ makes it difficult to control the time at which it hits 0. We will dominate $D_k$ by another martingale that is easier to handle.

\subsubsection*{Step 3: Introduction of an auxiliary distribution $\overline{\mu}_d$}

We introduce  the following family of distributions indexed by $d>0$:
\begin{multline}\label{eq:fqfq2}
\overline{\mu}_d(du):=\alpha_d \delta_{-(c_d+d)}(du)+f_d(u)\ind_{[-d,d]}(u)\  du \\
+ de^{-u-d} \ind_{[d,+\infty)}(u) \ du+ \beta_d e^{-u-d} \1_{[d +\nu_d,+\infty)}(u)\, du.
\end{multline}
{Let $\overline{\Delta}(d)$ be a r.v. with distribution $\overline{\mu}_d$, and set $\overline{D}(d)$ to be the r.v. defined by
\[\overline{\Delta}(d)=\overline{D}(d)-d.\]
Our strategy is as follows: we will choose carefully the functions $\alpha_d,c_d,r_d,\beta_d$ satisfying for any $d>0$,
\ben\label{eq:conszef}
\alpha_d \leq (d+1)e^{-2d},\qquad
c_d \geq 0,\qquad
\nu_d  \geq 0,\qquad
\beta_d \geq 0.
\een
in such a way that for any $d>0$, $\overline{\mu}_d$ is a probability distribution with mean 0, and which will dominate $\mu_d$ in the sense of the forthcoming Lemma \ref{lem:comptaus}. The difference between $\overline{\mu}_d$ and $\mu_d$ is that the atom at $-d$ in $\mu_d$ is replaced by an atom at $-c_d-d<-d$ with, as a counterpart, a modification of the distribution at the right of $d$ which is replaced by a distribution larger for the stochastic order. \par
Proceeding like this, our idea is to bound stochastically the hitting of 0 by $(D_k,k\geq 0)$ (the coalescing time $\tau$) by the hitting time of $(-\infty,0)$ by an auxiliary Markov chain $(\overline{D}_k,k\geq 0)$.
}
\begin{pro} The measure $\overline{\mu}_d$ is a density probability with mean  $\E(\overline{\Delta}(d))=0$ iff
\begin{equation}
\alpha_d = e^{-2d}\frac{(1+d)(1+2d+\nu_d)}{1+2d+c_d+\nu_d}\qquad \mbox{ and }\qquad \beta_d = e^{\nu_d} \frac{(1+d)c_d}{1+2d+c_d+\nu_d}.\label{choix:alphabeta}
\end{equation}
\end{pro}
\begin{proof}
Compute the total mass of $\overline{\mu}_d$:
 \begin{align*}
 \langle \overline{\mu}_d,1\rangle= \beta_d e^{-2d-\nu_d}+\alpha_d+1-e^{-2d}d-e^{-2d}
 \end{align*}
and if the total mass is 1, the expectation of $\overline{\Delta}(d)$ is~:
 \begin{align*}
 \E\big(\overline{\Delta}(d)\big)
 = &  - \alpha_d \ (c_d +d)+ d(d+1)e^{-2d} + \beta_d  (d+\nu_d+1) e^{-2d-\nu_d}.
 \end{align*}Solving these equations in $\alpha_d$ and $\beta_d$ provides the announced result.
\end{proof}

We hence see that we have two degrees of freedom. In the sequel, we will choose:
\begin{equation}\label{choix:cd-nud}
c_d=1,\qquad \nu_d=2,
\end{equation}independent of $d$. This implies that:
\begin{equation}
\alpha_d=e^{-2d}\frac{2d^2+5d+3}{2(d+2)},\qquad \beta_d=e^2 \frac{d+1}{2(d+2)}.
\end{equation}From this, we can compute $\Var(\overline{\Delta}(d))$:
\begin{equation*}\Var(\overline{\Delta}(d))
=1+{\frac { \left( 4\,{d}^{4}+26\,{d}^{3}+66\,{d}^{2}+75\,d+27 \right)
{e^{-2\,d}}}{6\,d+12}}.
\end{equation*}

For the measure $\overline{\mu}_d$ which we have now completely constructed, we have:
\begin{lem}\label{lem:comptaus} Under \eqref{eq:conszef} and \eqref{choix:alphabeta}, we have for $d<d'$,
\ben\label{eq:gzr}
D(d)\,\1_{D(d)>0}\leq _S \overline{D}(d')\,\1_{\overline{D}(d')>0},
\een
in the sense that for all $t>0$, $\P(D(d)>t)\leq \P(\overline{D}(d')>t)$.
\end{lem}
\begin{proof}
First, for any $d'$, by construction of the measure $\overline{\mu}_{d'}$,
$$
D(d')\1_{D(d')>0}\leq _S \overline{D}(d')\1_{\overline{D}(d')>0}.
$$
Now, recall that $D(d)$ provides the new distance in the model of Step 1 when the intensity of the PPP on $L_1$ is 1 and when the starting points $Z=(0,0)$ and $Z'(d,0)$ are at distance $d$. One can follow a third point $Z''(d',0)$.
Since these paths do not cross, the distance $D_1\eqd D(d)$ between $Z_1$ and $Z'_1$ remains smaller than the distance $D'_1=Z''_1-Z_1\eqd D(d')$. This implies that $D(d)\1_{D(d)>0} \leq _SD(d')\1_{D(d')>0}$ holds. This concludes the proof.
\end{proof}

\subsubsection*{Step 4: Introduction of an auxiliary Markov chain}

To dominate $(D_k,k\geq 0)$ we introduce the Markov chain $(\overline{D}_k, k \geq 0)$ whose distribution respects the same scaling \eref{eq:efgz} as $(D_k,k\geq 0)$: conditionally on $\overline{D}_{k-1}$, we let $\overline{\Delta}_k= \overline{D}_k-\overline{D}_{k-1}$ have distribution
\[\overline{\Delta}_k\eqd \frac{1}{n_k} \overline{\Delta}( \overline{D}_{k-1}\,n_k).\]
\begin{pro}\label{prop:tauR-}
Let us define
\begin{equation}\label{def:tauR-}
\tau_{\R^-}=\inf\{k \geq 0 : D_k\leq 0\},\qquad \mbox{ and  }\qquad  \overline{\tau}_{\R^-}=\inf\{k \geq 0 : \overline{D}_k\leq 0\}~.
\end{equation}
For any $d>0$, if $D_0=\overline{D}_0=d$, we have
\[\tau_{\R^-}\leq_S \overline{\tau}_{\R^-}.\]
\end{pro}
\begin{proof}
Just observe that in \eqref{eq:gzr}, the r.v. in both sides have atoms at 0 that correspond to the entrance of $D_k$ and $\overline{D}_k$ in $(-\infty,0]$ (in fact the hitting time of $\{0\}$ for $D_k$ and of $(-\infty,0)$ for $\overline{D}_k$). The Markov property and \eqref{eq:gzr} allow to conclude.
\end{proof}

\subsubsection*{Step 5: Skorokhod embedding}
By the Skorokhod embedding theorem (see \cite[Th. 37.6, page 519]{billingsley1968-PM}), there exists a BM $B$ started at $0$ and a stopping time $T_1(d)$ such that $\overline{\Delta}(d)\stackrel{(d)}{=} B_{T_1(d)}$. Moreover, it is possible to construct two r.v. $U(d)\leq 0$ and $V(d)\geq 0$ such that $T_1(d)=\inf\{t\geq 0,\ B_t\notin [U(d),V(d)]\}$. $U(d)$ and $V(d)$ are independent from the BM $B$, but not independent (in general) one from the other. Since $B_{T_1(d)}=U(d)\leq 0$ or $B_{T_1(d)}=V(d)\geq 0$, $U(d)$ and $V(d)$ can be constructed from the distribution of $B_{T_1(d)}\stackrel{(d)}{=}\overline{\Delta}(d)$, i.e. $\overline{\mu}_d$, as follows (recall \eqref{choix:cd-nud}):
\begin{itemize}
\item With probability $2p_d$, $U(d)=-V(d)$ and $V(d)$ is a r.v. with density $\ind_{[0,d]}(v) (e^{-2v}(v+1/2)-e^{-2d}/2) /p_d$. We denote by $A_d$ this event.
\item With probability $1-2p_d=(1+2d)e^{-2d}$, on $A^c_d$, we set $U(d)=-d-1$. For $V(d)$, we have two cases since the right tail of $\overline{\mu}_d$ is the sum of two exponential tails, $de^{-d-u}\ind_{[d,+\infty)}(u)$ and $\beta_d e^{-d-u}\ind_{[d+\nu_d,+\infty)}(u)$. Conditionally on $A^c_d$:
\begin{itemize}
\item[$\bullet$] With probability $q_d=2d/(1+2d)$, $V(d)$ is a r.v. with density $ \frac{v+d+1}{2(1+d)}e^{d-v} \ind_{[d,+\infty)}(v)$ with respect to the Lebesgue measure. We call this event $E_d$.
\item[$\bullet$] With probability $1-q_d=1/(1+2d)$, $V(d)$ is a r.v. with density
$\frac{v+d+1}{3+2d} e^{d+2-v} \ind_{[d+2,+\infty)}(v)$. This event is $E^c_d\cap A^c_d$.
\end{itemize}
\end{itemize}

\begin{proof}[Justification of the construction of $U(d)$ and $V(d)$]
Recall from \eqref{def:mud} that $\overline{\mu}_d$ admits a symmetric density $f_d$ on $[-d,d]$. Thus, on the event $A_d=\{|\overline{\Delta}_d|<d\}$, which has probability
\begin{equation}
2p_d= \int_{-d}^d f_d(u)du=1-e^{-2d}-2d e^{-2d},
\end{equation}it is sufficient to define $U(d)=-V(d)$ with $V(d)$ a r.v. of density $\frac{f_d(v)}{p_d}\ind_{[0,d]}(v)$. Since the Brownian motion $B$ started at 0 exits the symmetric interval $[-V(d),V(d)]$ through the upper or lower bound with equal probabilities $1/2$, the likelihood of $B_{T_1(d)}$ for this part is as expected:
$$2p_d \Big(\frac{1}{2}\frac{f_d(v)}{p_d}\ind_{[0,d]}(v)+\frac{1}{2}\frac{f_d(-v)}{p_d}\ind_{[-d,0]}(v)\Big)=f_d(v)\ind_{[-d,d]}(v).$$
Let us now consider $A_d^c \cap E_d$. The lower bound is necessarily $U(d)=-d-1$, since it is the only possible value for $\overline{\Delta}_d$ below $-d$. As for the density of $V(d)$ conditionally to $A_d^c\cap E_d$, say $g(v)$, it has to be chosen such that we recover $d e^{-d-u}\ind_{[d,+\infty)}(u)$ once multiplied by $(1-2p_d)$, $q_d$ and by the probability that $B$ exits through the upper bound $V(d)$ rather than through the lower bound $U(d)=-d-1$:
\begin{align*}
& (1-2p_d) q_d \frac{d+1}{v+d+1} g(v)=d e^{-d-v}\ind_{[d,+\infty)}(v)\quad
\Rightarrow \quad q_d g(v)= \frac{e^{2d}}{1+2d} \frac{v+d+1}{d+1} d e^{-d-v}\ind_{[d,+\infty)}(v).
\end{align*}Since $g$ is a probability density, integrating over $v$ gives $q_d$: $q_d=2d/(1+2d)$. We then deduce the density of $V(d)$ conditionally to $A_d^c\cap E_d$. We proceed similarly for $A_d^c\cap E_d^c$.
\end{proof}

By recursion, we can define for $k\geq 1$ the time $T_k$ by
$$T_k=\inf\{t\geq T_{k-1},\ B_t-B_{T_{k-1}}\notin [U_k(\overline{D}_{k-1}),V_k(\overline{D}_{k-1})]\}$$where $U_k(\overline{D}_{k-1})$ and $V_k(\overline{D}_{k-1})$ are independent r.v. conditionally on $\overline{D}_{k-1}$, such that for any $D>0$,
\begin{equation}\label{def:Tk}
U_k(D)\stackrel{(d)}{=} \frac{1}{n_k} U(n_k D),\qquad V_k(D)\stackrel{(d)}{=} \frac{1}{n_k} V(n_k D),
\end{equation}where $U(d)$ and $V(d)$ have the law described above in the representation of $T_1(d)$ for $d>0$.
With this construction, we have that for $k\geq 1$, $B_{T_{k}}\stackrel{(d)}{=} \overline{D}_k-d$.

\subsubsection*{Step 6: Laplace transforms of $T_1(d)$ and $T_k(d)$:}

\begin{lem} For $\lambda>0$, there exists $c_0(\lambda)\in (0,1)$ independent of $d$ such that
\begin{equation}
0\leq \varphi_d(\lambda)=  \E\big(e^{-\lambda T_1(d)} \big)\leq c_0(\lambda)<1.\label{etape3}
\end{equation}Moreover, for $\lambda$ small, there exists a constant $C>0$ such that
$c_0(\lambda)\leq e^{-C\lambda}$.
\end{lem}

\begin{proof}Using the Skorokhod embedding described above,
\begin{align}
\varphi_d(\lambda) = & \E\big(e^{-\lambda T_1(d)} \ |\ A_d \big)\ 2p_d + \E\big(e^{-\lambda T_1(d)} \ |\ A^c_d ,\ E_d\big) \ (1-2p_d)q_d \nonumber\\
& \hspace{1cm}+ \E\big(e^{-\lambda T_1(d)} \ |\ A^c_d ,\ E_d^c\big) \ (1-2p_d)(1-q_d).\label{etape2}
\end{align} Our purpose is to bound $\varphi_d(\lambda)$ uniformly in $d$ by a constant strictly smaller than 1. On the events, $A_d$ and $A^c_d \cap E_d$, the interval $[U(d),V(d)]$ which defines $T_1(d)$ has at least one extremity that gets closer and closer to zero when $d$ tends to zero. So upperbounding the expectations in the first and second terms of the r.h.s. of \eqref{etape2} by a constant strictly less than 1 uniformly in $d$ is difficult. For the third term of \eqref{etape2} however, because $U(d)<-c_d=-1$ and $V(d)>\nu_d=2$, we have that
$$\E\big(e^{-\lambda T_1(d)} \ |\ A^c_d ,\ E_d^c\big)\leq \E(e^{-\lambda T'})<1$$
where $T'=\inf\{t\geq 0,\ B_t \notin [-1,2]\}$.
Additionally, since
$(1-2p_d)(1-q_d)=e^{-2d}\rightarrow_{d\rightarrow 0} 1$, this shows \eqref{etape3} with
\begin{align*}c_0(\lambda)= & \E_0\big(e^{-\lambda T'}\big)
=  \frac{\cosh\big(\sqrt{\frac{\lambda}{2}}\big)}{\cosh\big(3\sqrt{\frac{\lambda}{2}}\big)}<1.
\end{align*}When $\lambda\rightarrow 0$, $c_0(\lambda)=  1-2\lambda+o(\lambda) \leq e^{-2\lambda}$ which shows the second assertion with $C=2$.
\end{proof}
From this by using \eqref{def:Tk} and the self-similarity of the standard BM started at 0,
\begin{align}
T_k-T_{k-1} \stackrel{(d)}{=}  & \inf\Big\{t\geq 0,\ B_t \notin \big[\frac{1}{n_k}U(n_k \overline{D}_{k-1}),\frac{1}{n_k}V(n_k \overline{D}_{k-1})\big]\Big\}\nonumber\\
\stackrel{(d)}{=}  & \inf\Big\{t\geq 0,\ \frac{1}{n_k} B_{n_k^2 t} \notin \big[\frac{1}{n_k}U(n_k \overline{D}_{k-1}),\frac{1}{n_k}V(n_k \overline{D}_{k-1})\big]\Big\}\nonumber\\
\stackrel{(d)}{=} & \frac{1}{n_k^2} T_1(n_k \overline{D}_{k-1}).
\end{align}Hence it follows that
\begin{equation}
\E\big(e^{-\lambda (T_k-T_{k-1})}\ |\ \mathcal{F}_{T_{k-1}}\big)=\varphi_{n_k \overline{D}_{k-1}}\Big(\frac{\lambda}{n_k^2}\Big).\label{laplace:k}
\end{equation}

\subsubsection*{Step 7: Estimate  for the tail distribution of the coalescing time}
With the ingredients developed above, we can now follow ideas developed in \cite{coletti2009} for instance.
Recall that $\overline{\tau}_{\R_-}=\inf\{k\in \N, \overline{D}_k\leq 0\}$ and define $\theta=\inf\{t\geq 0,\ B_t=-d\}$. Let us consider $\zeta>0$. Then for $K\in \N\setminus\{0\}$:
\begin{align}
\P\big(\overline{\tau}_{\R_-}>K\big)= & \P\big(\theta > T_K\big)\nonumber\\
\leq &  \P\big(\theta > \zeta h_K\big)+\P\big(\theta>T_K,\ T_K < \zeta h_K\big)\nonumber\\
\leq & \frac{Cd}{\sqrt{\zeta h_K}}+ e^{\lambda \zeta h_K}\E\big(e^{-\lambda \sum_{k=1}^K (T_k-T_{k-1})}\big).\label{etape1}
\end{align}
For the Laplace transform in the last term, using \eqref{laplace:k}:
\begin{align}
\E\big(e^{-\lambda \sum_{k=1}^K (T_k-T_{k-1})}\big) = & \E\Big(\E_d\big(e^{-\lambda \sum_{k=1}^K (T_k-T_{k-1})}\ |\ \mathcal{F}_{T_{K-1}}\big)\Big)\nonumber\\
= & \E\Big(e^{-\lambda \sum_{k=1}^{K-1} (T_k-T_{k-1})} \E\big(e^{-\lambda (T_K-T_{K-1})}\ |\ \mathcal{F}_{T_{K-1}}\big)\Big)\nonumber\\
 =
& \E_d\Big(e^{-\lambda \sum_{k=1}^{K-1} (T_k-T_{k-1})} \varphi_{n_K \overline{D}_{K-1}}\big(\frac{\lambda}{n_K^2}\big)\Big)\nonumber\\
\leq  & c_0 \Big(\frac{\lambda}{n_K^2}\Big)\E_d\Big(e^{-\lambda \sum_{k=1}^{K-1} (T_k-T_{k-1})}\Big)\leq \prod_{k=0}^{K-1}c_0 \Big(\frac{\lambda}{n_{k+1}^2}\Big)\nonumber\\
\leq & \exp\big(-2 \lambda \sum_{k=1}^{K} \frac{1}{n_{k}^2}\big).\label{etape4}
\end{align}
Recall from \eqref{def:Vk} that $h_K=V_K\sim \sum_{k=1}^K \frac{1}{2n_k^2}$. Thus, from \eqref{etape1} and \eqref{etape4}:
\begin{align}
\P\big(\overline{\tau}_{\R_-}>K\big)\leq  \frac{Cd}{\sqrt{\zeta h_K}}+ C'\exp\big(\lambda h_K(\zeta-4)\big).\label{etape5}
 \end{align}Because $h_K\rightarrow +\infty$, and because the term in the exponential is negative for $\zeta$ sufficiently small, there exists $\lambda_0>0$ and $\zeta_0>0$ such that the r.h.s. of \eqref{etape5} is smaller than $\frac{C'' d}{\sqrt{\zeta_0 h_K}}$ for $K$ large enough.\par

This together with Proposition \ref{prop:tauR-} allow to conclude the proof of Proposition \ref{pro:estimtpscoal}. Starting from two points $Z$ and $Z'$ of $L_0$ at distance $d$ and denoting by $\tau$ the index of the level at which they coalesce, we have for any $K\in \N\setminus \{0\}$,
$$\P\big(\tau>K\big)\leq \P\big(\overline{\tau}_{\R_-}>K\big)\leq \frac{Cd}{\sqrt{h_K}}.$$
\hfill $\Box$

Let us finish this subsection with the proof of Proposition \ref{pro:repre} that had been postponed.

\begin{proof}[Proof of Proposition \ref{pro:repre}]
First, notice that $X_1$ has density $e^{-2|x|} \1_{x \in \R}$. Then, we can compute the distribution of $X_2$ conditionally on $X_1$. In what follows, all r.v. are independent, $R$ is a Rademacher r.v., $\Exp(k)$ denotes an exponential r.v. with expectation $1/k$.
\begin{itemize}
\item[--] Conditional on $X_1=x_1>0$, with $x_1<d$: \\
$\bullet$  $X_2=-(d-x_1)$ (merge) with probability $e^{-2(d-x_1)}$,\\
$\bullet$ with probability$1-e^{-2(d-x_1)}$, $X_2\sim{\cal L}(R  \Exp(2) | \Exp(2) <d-x_1)$.
\item[--] Conditional on $X_1=-x_1 <0$, with $x_1<d$\\
$\bullet$ $X_2=-(d+x_1)$ (merge) with probability $e^{-2(d-x_1)-2x_1}$\\
$\bullet$ $X_2\sim {\cal L}(R  \Exp(2) | \Exp(2) <d-x_1)$ with probability  $1-e^{-2(d-x_1)}$\\
$\bullet$ $X_2\sim {\cal L}(\Exp(1)+d-x_1 | \Exp(1) <2x_1)$ with probability  $e^{-2(d-x_1)}(1-e^{-2x_1})$
\item[--] Conditional on $X_1=x_1>0$, with $x_1>d$: \\
$\bullet$  merge with probability1
\item[--] Conditional on $X_1=-x_1 <0$, with $x_1>d$\\
$\bullet$ $X_2=-(d+x_1)$ (merge) with probability $e^{-2d}$\\
$\bullet$ $X_2\sim {\cal L}(x_1-d+ \Exp(1) | \Exp(1) <2d)$ with probability $1-e^{-2d}$
\end{itemize}
This yields the announced result.
In particular, the two trajectories started at $(0,0)$ and $(d,0)$ merge at ordinate 1 with probability:
\be
`P(D(d)=0)&=&\int_{x=0}^d e^{-2x} (e^{-2(d-x)}+e^{-2(d-x)-2x})dx+\int_{x=d}^{+\infty} {e^{-2x}} (1+e^{-2d})dx,
\ee which is $(d+1)e^{-2d}$, as announced.
\end{proof}

\subsubsection{Extension to the shifted cylinder}\label{section:plan-to-cylindre}

We now conclude the section with a corollary establishing an estimate for the coalescence time in $\mathcal{W}^{(j)}$, which is the forest $\mathcal{W}$ shifted by $(0,-h_j)$ similarly to $\mathcal{T}^{(j)}$. Then, we enounce an estimate for the shifted cylindrical forest $\mathcal{T}^{(j)}$.
\begin{cor}\label{cor:CoalTime}
Let $d>0$ and $j\in \N$. \\
(i) Let us consider the paths in $\mathcal{W}^{(j)}$ started at $(0,0)$ and $(d,0)$ (if $(0,h_j)$ and $(d,h_j)$, these points are connected at the level $j+1$ to the closest point of $\mathcal{W}$). Define their coalescing time as $\tau=\inf\{k\geq j,\ D_k=0\}$. There exists a constant $C>0$ such that for any $K> j$,
$$\P\big(\tau>K\big)\leq \frac{Cd}{\sqrt{h_K-h_j}}.$$
This can be translated, for any $t_0>0$ as:
\begin{equation}
\P\big(\mathcal{W}^{(j)}_{(0,0)}(t_0)\not= \mathcal{W}^{(j)}_{(d,0)}(t_0)\big)\leq \frac{C\ d}{\sqrt{h_{R(h_j+t_0)-1}-h_j}} \rightarrow_{j\rightarrow +\infty} \frac{C\ d}{\sqrt{t_0}}.\label{lim_cor}
\end{equation}
$(ii)$ Let us consider the paths of $\mathcal{T}^{(j)}$ started at $(0,0)$ and $(d,0)$ for $d\in(0,1/2]$ ($d=1/2$ is the maximal distance in the cylinder). Then, there exists $C>0$ such that for any $t_0>0$:
\begin{equation}
\P\big(\mathcal{T}^{(j)}_{(0,0)}(t_0)\not= \mathcal{T}^{(j)}_{(d,0)}(t_0)\big)\leq \frac{C\ d}{\sqrt{h_{R(h_j+t_0)-1}-h_j}} \rightarrow_{j\rightarrow +\infty} \frac{C\ d}{\sqrt{t_0}}.\label{lim_cor2}
\end{equation}

\end{cor}
\begin{proof}
The proof of (i) is an adaptation of the proof Step 7 of Prop. \ref{pro:estimtpscoal} by summing between levels $L_j$ and $L_K$. \\

Let us now consider $(ii)$. Intuitively, the coalescence time in the cylinder is stochastically dominated by the coalescence time in the plane. But since some slices in the cylinder may contain no points of the PPP (when no line $L_k$ in the plane is empty), and since the increments of the distance between the two paths are non standard when this distance is close to 0 and 1 (when only the case 0 matters in the plane), an additional argument is needed in the discrete case to establish the domination rigorously.

Recall the model introduced in Section \ref{section:plan1}. We consider the Markov chain $(D_k,k\geq 0)$ and denote by $\tau_d$ be stopping time at which the Markov chain started from $d$ hits 0. We also introduce similarly the distance process $(\bar{D}_k,k\geq 0)$ in the cylinder.\\
Now, let us define another Markov chain $(D'_k,k\geq 0)$ with the following transitions: %the kernel is the same as that of $D_k$ , except when $D_k$ goes over $1/2$ in which case, we add ``a negative push'':
\ben
{\cal L}(D'_{k+1} ~|~ D'_k=d)= {\cal L}\big( \min(D_{k+1},|1-D_{k+1}|) ~|~ D_k=d \big).\een
The distance $D'_{k}$ somehow mimics the distance on the cylinder by considering the minimum distance between two points of the same level in the clockwise and counter clockwise senses. Let us define by $\tau'_d$ the stopping time at which $(D'_k,k\geq 0)$ started from $d$ hits 0. Since $\min(D_{k+1},|1-D_{k+1}|)\leq D_{k+1}$, and since
\ben\label{eq:tjdktu}
\tau_d \leq_S \tau_{d'} \textrm{~when~} d<d',
\een
by using the same argument as in the proof of Lemma \ref{lem:comptaus}, we obtain by using iteratively \eref{eq:tjdktu} that
\[\tau'_d \leq_S  \tau_d.\]
To conclude, it remains to show that $(D'_k,k\geq 0)$ coincides with $(\bar{D}_k,k\geq 0)$, up to a probability going to 0 in $j$. Since we may produce a local coupling between $D'_k$ and $\bar{D}_k$ as long as $D_k$ possesses small fluctuations, it suffices to prove that all the increments of the paths $(\bar{Z}_k, k\geq 0)$ and $(\bar{Z}'_k,k\geq 0)$ that define $(D'_k,k\geq 0)$ in the cylinder are not 0 and smaller than $1/6$ after the slice $j$ with probability going to 1 when $j\to +\infty$. This indeed guarantees that the cylinder effects do not prevent the coupling: no jumps ``0'' occur and ``decision domains'' do not see that the environment is a cylinder.
The probability that there is no point within distance $\pm 1/6$ for a walk is $e^{-n_k/3}$, and by Borel-Cantelli's lemma, with probabilty 1 the two walks  $(\bar{Z}_k)$ and $(\bar{Z}'_k)$ will do a finite number of jumps larger than $1/6$. Hence, for any $`e>0$, for $j$ large enough, the distribution of $(D'_k,k\geq 0)$ and $(\bar{D}_k,k\geq 0)$ coincides with probability at least $1-`e$. Thus the coupling works, which allows to conclude.
\end{proof}

We have now the tools to prove the criteria of the convergence Theorem \ref{theo:conv_cyl}, (IO) and (EO). Both of these criteria make use of the estimates on coalescing time that we hve just established.

\subsection{Proof of $(IO)$}
\label{sect:ProofIO}

The purpose of this section is to prove the next Proposition which implies $(IO)$.
\begin{pro}
\label{prop:IO}
Assume \eref{eq:newcond}. Let $m\in \N\setminus\{0\}$ and $y_1=(x_1,t_1),\dots,y_m=(x_m,t_m)\in\Cyl^+$. For $j\geq 0$ and $1\leq \ell \leq m$, let us denote by $\gamma^{(j)}_{y_\ell}$ the path interpolating linearly the shifted ancestor line ${\sf AL}^{(j)}_{y_\ell+(0,h_j)}$.
Then, the sequence $(\gamma^{(j)}_{y_1},\dots \gamma^{(j)}_{y_m})$ converges in distribution, when $j\rightarrow +\infty$, to coalescing Brownian motions modulo 1 started at $y_1,\dots y_m$.
\end{pro}

Notice that the path $\gamma^{(j)}_{y_\ell}$ starts at $y_\ell$. We also recall that the ancestral line ${\sf AL}_{y_\ell+(0,h_j)}$ does not necessarily starts from a point of $\Xi$, but links the starting point $y_\ell+(0,h_j)$ to the closest point of $\Xi$ in the first non-empty slice of height greater or equal to $R(t_\ell+h_j)$. For the sequel, let us denote by $y_{j,\ell}$ this point.

\begin{proof}[Proof of Prop. \ref{prop:IO}]
The result for $m=1$ is due to Lemma \ref{lem:convMB} and the fact that $y_{j,1}-(0,h_j)$ converges a.s. to $y_1$.
The proof can be done by recursion, and we focus here on the case $m=2$ which can be generalized directly by following Arratia \cite{Arratia} and Ferrari, Fontes and Wu \cite[Lemmas 2.6 and 2.7]{ferrarifonteswu}.

Let us first recall a simple fact. Let $t\leq t'$ and $a,b \in \rpuz$. Two BM $(\W^\ua_{(a,t)}(s),s\geq t)$ and $(\W^\ua_{(b,t')}(s),s \geq t')$ on the cylinder are said to be coalescing BM if $\W^\ua_{(a,t)}(s)$ for $t\leq s \leq t'$ is a standard BM taken modulo 1, and if the two trajectories $(\W^\ua_{(a,t)}(s),s\geq t')$ and $(\W^\ua_{(b,t')}(s),s \geq t')$ are BM till their hitting time $\tau$. After this time, they coincide with $(\W^\ua_{(b,t')}(s), s\geq \tau)$.

To prove that $(\gamma^{(j)}_{y_1},\gamma^{(j)}_{y_2})$ converges to two coalescing BM, a strategy consists in decomposing the trajectories as follows, where we can assume to simplify that $y_1$ and $y_2$ are such that $t_1=t_2$:\\
$(a)$ as long as the two paths are far apart, say if
\ben\label{eq:far}d_{\R/\Z}\l(\gamma^{(j)}_{y_1}(s),\gamma^{(j)}_{y_2}(s)\r)>a^j(s)\een for a good sequence $a^j(s)\to 0$, then the next steps of these trajectories are likely to be characterized by $\Xi\cap I$ and $\Xi\cap I'$ for two random influence intervals $I$ and $I'$ that will not intersect. By the spatial properties of the PPP, it means that as long as $I\cap I'=\varnothing,$  the two trajectories behave as if they were constructed on different spaces, and then eventually behave as independent BM before their coalescing time (here, since the intensity is not constant, a dependence in $s$ is needed).\\
$(b)$  when \eref{eq:far} fails (the two paths are close) then another argument is developed to prove that the two paths will merge with a probability going to 1, within a $o(1)$ delay. This is given by the Corollary \ref{cor:CoalTime}.\\

\par
It remains to see in details how $(a)$ can be handled. Let us denote by $(\bar{X}^{(j)}_s)_{s\geq t}$ (as in \eqref{shifted_interpolation}) the path $\gamma^{(j)}_{(x,t)}$. The distribution of the increment $\Delta X_{j+k}$ (with the notation of section \ref{sect:Cv1path}) satisfies
\ben\label{eq:dsqd}
\P(|\Delta X_{j+k}|\geq r)=e^{-2n_{k+j} r}
\een
so that for $r=f_{j+k}/2$, with the $f_{j+k}$'s appearing in Assumption \eqref{eq:newcond}, the event $\{\forall k\geq 0,\ |\Delta X_{j+k}|\leq f_{j+k}/2\}$ will occur a.s. for $j$ large enough thanks to Borel-Cantelli's lemma. We then decree that two walks are close if when they get in the slice with intensity $n_{k+j}$, their distance is smaller than $f_{j+k}$, i.e. we choose $a^j(s)=f_{R(h_j+s)-1}$. This suffices to complete the proof.
\end{proof}

\subsection{Proof of $(EO)$}
\label{sect:ProofEO}

We follow the strategy developed in \cite{SSS}: we first show that the sequence $(\mathcal{T}^{(j)})_{j\geq 1}$ satisfies (\ref{but}), stated below, from which $(EO)$ follows.

\begin{pro}
\label{theo:conditionEO}
The sequence $(\mathcal{T}^{(j)})_{j\geq 1}$ satisfies for all $t_0,t>0$, $a>0$,
\begin{equation}
\label{but}
\limsup_{j\rightarrow +\infty} \E \big( \widehat{\eta}_{\mathcal{T}^{(j)}}^{O}(t_0,t ; [0 \to a]) \big) < +\infty ~.
\end{equation}
This implies that $(\mathcal{T}^{(j)})_{j\geq 1}$ satisfies (EO).
\end{pro}

The rest of this section is devoted to the proof of Proposition \ref{theo:conditionEO}.

\begin{proof}
Let us remark that, by translation invariance and linearity of the expectation, it is enough to prove \eqref{but} for small values of $a$.\\

Let us start with proving that \eqref{but} implies $(EO)$. We follow the corresponding proof in the planar context which is recalled at the end of Section \ref{section:BWTopo} (see also the end of Section 6.1 in \cite{SSS} or Section 6 of \cite{newmanravishankarsun}). Except that contrary to the planar context where explicit computation is possible, namely $\E(\widehat{\eta}_{W}(t_0,t ; a,b))=(b-a)/\sqrt{\pi t}$ where $W$ is the standard BW, another argument is needed to get the limit \eqref{limitEetaBW} in the cylinder context. We then have to show:
\begin{equation}
\label{LimitToEO}
\lim_{\varepsilon\to 0} \E \big( \widehat{\eta}^O_{\Wu}(t_0+\varepsilon,t-\varepsilon ; [a \to b]) \big) = \E \big( \widehat{\eta}^O_{\Wu}(t_0,t ; [a \to b]) \big) ~.
\end{equation}

Let us consider $t,t_0>0$ and $a,b\in\rpuz$. Let us first prove that there exists $\varepsilon_0\in(0,t)$ such that
\begin{equation}
\label{ExpWidehatEta0Finite}
\E \big( \widehat{\eta}^O_{\Wu}(t_0+\varepsilon_0,t-\varepsilon_0 ; [a \to b]) \big) < \infty ~.
\end{equation}
Without loss of generality, we still write $t_0$ and $t$ instead of $t_0+\varepsilon_0$ and $t-\varepsilon_0$. The inequality (be careful to the presence or not of the hat $\widehat{\phantom{\eta}}$ on $\eta$ ),
$$
\widehat{\eta}^O_{\Wu}(t_0,t ; [0 \to 1]) \, \leq \, \eta^O_{\Wu}(t_0,t ; [0 \to 1/2]) + \eta^O_{\Wu}(t_0,t ; [1/2 \to 1]) \;\; \mbox{a.s.}
$$
implies by rotational invariance
$$
\E \big( \widehat{\eta}^O_{\Wu}(t_0,t ; [a \to b]) \big) \leq 2 |a\to b|\,\E \big( \eta^O_{\Wu}(t_0,t ; [0 \to 1/2]) \big) ~.
$$
We finally get (\ref{ExpWidehatEta0Finite}) in combining the fact that $\eta^O_{\Wu}(t_0,t ; [0 \to 1/2])$ is stochastically dominated by $\eta_{W}(t_0,t ; 0,1/2))$ (already stated in (\ref{eta^O<eta})) and
$$
\E(\eta_{W}(t_0,t ; 0,1/2)) = 1+\E(\widehat{\eta}_{W}(t_0,t ; 0,1/2)) = 1+\frac{1}{2\sqrt{\pi t}} ~.
$$

{With $\widehat{\eta}^O_{\Wu}(t_0+\varepsilon,t-\varepsilon ; [a \to b]) \eqd \widehat{\eta}^O_{\Wu}(t_0,t-\varepsilon ; [a \to b])$, (\ref{ExpWidehatEta0Finite}) and
\begin{equation}
\label{CVpsEO}
\mbox{a.s. } \, \lim_{\varepsilon\to 0} \widehat{\eta}^O_{\Wu}(t_0,t-\varepsilon ; [a \to b]) = \widehat{\eta}^O_{\Wu}(t_0,t ; [a \to b]) ~,
\end{equation}
the Lebesgue's dominated convergence theorem applies and leads to the searched limit (\ref{LimitToEO}). Mainly because there is no coalescence on the arc $\{t_0+t\}\times [a \to b]$ for the trajectories starting before $t_0$ with probability $1$, there exists a random $\varepsilon>0$ such that for any $0\leq \varepsilon'\leq \varepsilon$, $\widehat{\eta}^O_{\Wu}(t_0,t-\varepsilon' ; [a \to b])$ is equal to the limit $\widehat{\eta}^O_{\Wu}(t_0,t ; [a \to b])$. This proves (\ref{CVpsEO}).}

\vskip 0.3cm

Now, let us show that $(\mathcal{T}^{(j)})_{j\geq 1}$ satisfies \eqref{but}. The strategy to get (\ref{but}) can be divided into two steps. First, we bound from above the expectation in \eqref{but} by twice the expected number of remaining paths $\gamma$ at time $t_{0}+t$ which are born before $t_{0}$ and such that $\gamma(t_{0})\in [0\to a]$, i.e. $2 \E(\eta^{O}_{\mathcal{T}^{(j)}}(t_0,t;[0 \to a]))$. See Lemma \ref{lem:EtaHatEta}. Thus, using the coalescence time estimate (Corollary \ref{cor:CoalTime}), we obtain an upper bound for this latter expectation when $j\rightarrow +\infty$. This is Lemma \ref{lem:GridB1O}.
The various lemma on which the proof of \eqref{but} is based are proved at the end of the present section.

\begin{lem}
\label{lem:EtaHatEta}
For all times $t_0,t>0$, for all $a>0$ and any $j\geq 1$, the following inequality holds:
\begin{equation}
\label{etape-but2}
\E\l(\widehat{\eta}^{O}_{\mathcal{T}^{(j)}}(t_0,t;[0 \to a])\r) \leq 2 \,\E\l(\eta^{O}_{\mathcal{T}^{(j)}}(t_0,t;[0 \to a])\r)~,
\end{equation}from which we deduce that:
\begin{equation}\label{etape-but2b}
\limsup_{j\rightarrow +\infty}\E\l(\widehat{\eta}^{O}_{\mathcal{T}^{(j)}}(t_0,t;[0 \to a])\r) \leq 2 \limsup_{j\rightarrow +\infty}\E\l(\eta_{\mathcal{W}^{(j)}}(t_0,t;0 , a)\r),
\end{equation}where $\mathcal{W}^{(j)}$ is the shifted forest introduced in Section \ref{section:plan-to-cylindre}.
\end{lem}

In view of \eqref{etape-but2b} of Lemma \ref{lem:EtaHatEta}, we now focus on showing that
\begin{equation}\label{but2}
\limsup_{j\rightarrow +\infty}\E\l(\eta_{\mathcal{W}^{(j)}}(t_0,t;0 , a)\r)<+\infty ~.
\end{equation}

Let us choose $r\in \N\setminus \{0\}$ (intended to be large) and $m(a,r):=\min\{m:\,m\geq ar\}$. We consider the grid
$$
\Gr(t_0,a,r) := \l\{\frac{k}{r}, k\in \{0,\dots r_a\} \r\} \times \{h_{R(h_j+t_0)-1}-h_j\} ~.
$$
The height $h_{R(h_j+t_0)-1}-h_j$ corresponds to the (shifted) largest slice just before height $h_j+t_0$ (possibly $h_j+t_0$ itself): see Figure \ref{fig:Aajm}. Since the sequence $h_k-h_{k-1}=\sigma^2_k$ tends to zero by \eqref{eq:displ}, there exists $j_0$ such that, for any $j\geq j_0$, there is a slice carrying points between $t_0$ and $t_0+t$ in $\mathcal{W}^{(j)}$, i.e. $t_0<h_{R(h_j+t_0)}-h_j<t_0+t$.

Let us focus on the paths starting from the points of $\Gr(t_0,a,r)$. For $0\leq k\leq m(a,r)$, denote by $\gamma_{k}(.)$ the ancestor line starting at $(k/r, h_{R(h_j+t_0)-1}-h_j)$. Even if the points $(k/r, h_{R(h_j+t_0)-1})$ do not belong to the point process $\Upsilon$ defined in Section \ref{section:plan1}, they connect to the nearest point of $\Upsilon\cap L_{R(h_j+t_0)}$. So each path $\gamma_k(.)$ a.s. coincides with a path of $\mathcal{W}^{(j)}$ after one step. Let us define the event
\begin{equation}
\label{def:Aajm}
A_{a,j,r} := \left\{\begin{array}{c}
\mbox{the ancestors of the points of }\mathcal{W}^{(j)} \cap \big([0, a]\times \{h_{R(h_j+t_0)-1}-h_j\}\big) \\
\mbox{are also ancestors of some points of the grid }\Gr(t_0,a,r)
\end{array}
\right\} ~.
\end{equation}
The event $A_{a,j,r}$ is described in Figure \ref{fig:Aajm}.
%Given $j$, the cardinal of $\Upsilon\cap\big([0\to a]\times\{h_{R(h_j+t_0)}\}\big)$ is ${\sf Poisson}(a n_{R(h_j+t_{0})})$ distributed.
We claim that when the mesh $1/r$ of the grid $\Gr(t_0,a,r)$ tends to $0$, the probability of $A_{a,j,r}$ tends to $1$:

\begin{lem}
\label{lem:Aajm}
For all times $t_0,t>0$, for all $a>0$ and any $j\geq j_0$,
\begin{equation}
\label{etape-but3}
\lim_{r\rightarrow +\infty} \P\big(A_{a,j,r}^c\big) = 0 ~.
\end{equation}
\end{lem}

\begin{figure}[!ht]
\centerline{\includegraphics{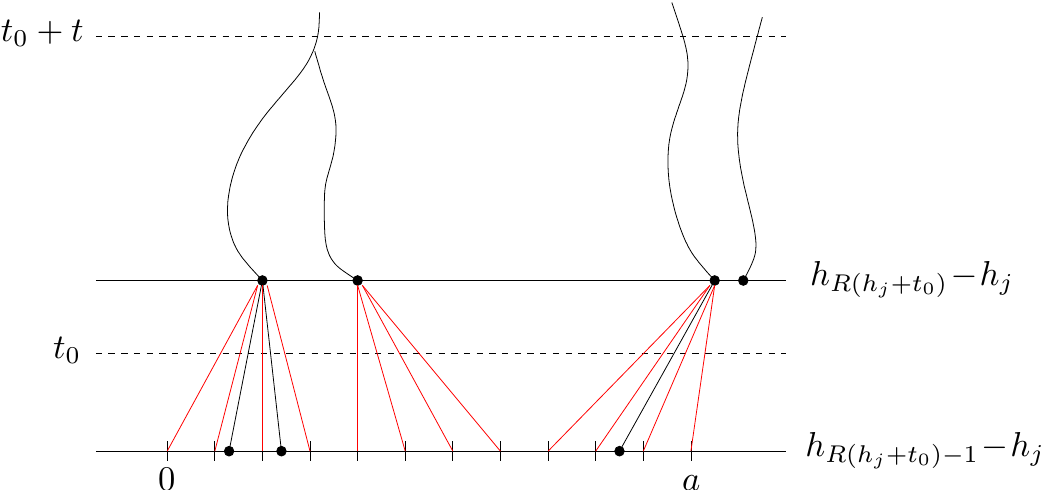}}
\caption{\label{fig:Aajm} {\small \textit{Edges and paths of the shifted forest $\mathcal{W}^{(j)}$ are in black whereas edges starting from points of the grid $\Gr(t_0,a,r)$ are in red. On this picture, the event $A_{a,j,r}$ holds: the vertices at the height $h_{R(h_j+t_0)}-h_j$ which are ancestors of points of $\mathcal{W}^{(j)}\cap\big([0, a]\times\{h_{R(h_j+t_0)-1}-h_j\})$-- i.e. the first and third ones, from left to right --are also ancestors of some points of the grid $\Gr(t_0,a,r)$.}}}
\end{figure}

The event $A_{a,j,r}$ has been introduced in order to compare $\eta_{\mathcal{W}^{(j)}}(t_0,t;0 , a)$ to the number of remaining paths at height $t_{0}+t$, starting from the deterministic points of $\Gr(t_0,a,r)$. Then, the coalescence time estimate (Corollary \ref{cor:CoalTime}) leads to the following bound:

\begin{lem}
\label{lem:GridB1O}
For all times $t_0,t>0$, for all $a>0$, there exists a constant $C>0$ and an integer $j_1=j_1(t_0,t)$ (which does not depend on $r$) such that for any $j\geq j_1$ and any $r\in \N\setminus\{0\}$,
$$
\E \big( \eta_{\mathcal{W}^{(j)}}(t_0,t ; 0,a) \ind_{A_{a,j,r}} \big) \leq 1 + \frac{2 m(a,r) C}{r \sqrt{t}} ~,
$$
where $C$ is the universal constant given by Corollary \ref{cor:CoalTime}.
\end{lem}

{We can now conclude. Let $t_0,t>0$ and $a>0$. First, we bound $ \E(\eta_{\mathcal{W}^{(j)}}(t_0,t;0, a))$ from above by
\begin{equation}
\label{ConclusionEO}
\E \big(\eta_{\mathcal{W}^{(j)}}(t_0,t;0,a) \ind_{A_{a,j,r}}\big) + M \P \big( A_{a,j,r}^c\big) + \E \big( |\Upsilon_{R(h_j+t_{0})-1}([0,a])| \ind_{\{|\Upsilon_{R(h_j+t_{0})-1}([0,a])|>M\}} \big) ~,
\end{equation}
for any $j,M,r$. Let $j\geq j_0\vee j_1$. Then Lemma \ref{lem:GridB1O} applies and provides a bound for the first term of (\ref{ConclusionEO}). Then, take $M=M_{j}:=2 n_{R(h_j+t_{0})-1}$ (twice the intensity of the PPP). For $j$ large enough and with this choice of $M$, the third term of (\ref{ConclusionEO}) is smaller than $1$. It then follows:
$$
\E \l(\eta_{\mathcal{W}^{(j)}}(t_0,t;0,a) \r) \leq 2 + \frac{2 m(a,r) C}{r \sqrt{t}} + M_j \P \big( A_{a,j,r}^c\big) ~.
$$
Let us point out that till now, the parameter $r$ is totally free. So we can choose it large enough so that $m(a,r)/r\leq a+1/r\leq 2a$ and $M_j \P(A_{a,j,r}^c) \leq 1$ (by Lemma \ref{lem:Aajm}). In conclusion, for any $j$ large enough, $\E(\eta_{\mathcal{W}^{(j)}}(t_0,t;0,a))$ is bounded by $3 + \frac{4 a C}{\sqrt{t}}$. This gives (\ref{but2}), ending the proof of Proposition \ref{theo:conditionEO}.}
\end{proof}

This section ends with the proofs of Lemmas \ref{lem:EtaHatEta}, \ref{lem:Aajm} and \ref{lem:GridB1O}.

\begin{proof}[Proof of Lemma \ref{lem:EtaHatEta}]Let us first prove \eqref{etape-but2}. We denote, for an interval $I$ of $\rpuz$, by $\widehat{\eta}^{O}_{\mathcal{T}^{(j)}}(t_{0},t ; I ; [0 \to a])$ the number of paths $\gamma\in \mathcal{T}^{(j)}$ born before $t_0$ and such that $\gamma(t_0)\in I$ and $\gamma(t_0+t)\in [0\to a]$:
$$
\widehat{\eta}^{O}_{\mathcal{T}^{(j)}} (t_0,t ;I ; [0 \to a]) := \Card\big\{\gamma(t_0+t)\in [0\to a],\ \gamma\in \mathcal{T}^{(j)}, \ \gamma \mbox{ born before }t_0,\ \gamma(t_0)\in I\big\} ~.
$$
Then, the following inequality holds almost surely:
\begin{equation}
\label{etape-but1}
\widehat{\eta}^{O}_{\mathcal{T}^{(j)}}(t_0,t ; [0\to a])\leq \sum_{k=0}^{[1/a]} \widehat{\eta}^{O}_{\mathcal{T}^{(j)}}(t_0,t ; [ka \to (k+1)a] ;[0\to a]),
\end{equation}
where $[1/a]$ is the integer part of $1/a$. The inequality \eqref{etape-but1} is due to the fact that two paths starting from different intervals $[ka \to (k+1)a]$ and $[\ell a \to (\ell+1)a]$ can coalesce and give a single point in the l.h.s. while they are counted twice in the r.h.s. Notice that when $a$ is not the inverse of an integer, the right hand side is a bit larger than what is needed, because the first and last intervals intersect. We conclude using the rotational invariance:
\begin{align*}
\E\big(\widehat{\eta}^{O}_{\mathcal{T}^{(j)}}(t_0,t ; [0\to a])\big)\leq  & \sum_{k=0}^{[1/a]} \E\big(\widehat{\eta}^{O}_{\mathcal{T}^{(j)}}(t_0,t ; [ka \to (k+1)a] ; [0\to a])\big)\\
= & \sum_{k=0}^{[1/a]} \E\big(\widehat{\eta}^{O}_{\mathcal{T}^{(j)}}(t_0,t ; [0\to a] ; [-ka \to -(k-1)a])\big)\\
= & \E\big(\eta^{O}_{\mathcal{T}^{(j)}}(t_0,t ; [0\to a])\big) +  \E \big( \widehat{\eta}^{O}_{\mathcal{T}^{(j)}}(t_0,t ; [0\to a] ; [1-a[1/a]\to a])\big) \\
\leq &  2\E\big(\eta^{O}_{\mathcal{T}^{(j)}}(t_0,t ; [0\to a])\big) ~.
\end{align*}
since the arc $[1-a[1/a]\to a]$ is included in $[0\to a]$. This proves \eqref{etape-but2}.\\

We now turn to the proof of \eqref{etape-but2b}. Having proved \eqref{etape-but2}, it is sufficient to show that
$$\limsup_{j\rightarrow +\infty}\E\l(\eta^{O}_{\mathcal{T}^{(j)}}(t_0,t;[0 \to a])\r) \leq  \limsup_{j\rightarrow +\infty}\E\l(\eta_{\mathcal{W}^{(j)}}(t_0,t;0 , a)\r)~.
$$
For this, we construct a coupling between $\mathcal{W}^{(j)}$ and $\mathcal{T}^{(j)}$. Let us introduce the following event:
$$E_{j} := \left\{\begin{array}{c}
\mbox{for }k\geq j, \mbox{ none of the }\Slice(h_k)\mbox{ are empty}
\end{array}
\right\} ~.
$$
Notice that $\P(E^c_j)\leq \sum_{k\geq j}e^{-n_k}$ which converges to zero when $j\rightarrow +\infty$, and by Borel Cantelli's lemma and \eqref{eq:newcond}, there exists a random level $J$, finite a.s., such that $E_J$ holds. Following the idea in \eqref{ConclusionEO}, we have:
\begin{multline}\label{etape6}
\E \big(\eta^{O}_{\mathcal{T}^{(j)}}(t_0,t;[0 \to a])\big)
\leq  \E \big(\eta^{O}_{\mathcal{T}^{(j)}}(t_0,t;[0 \to a])\ind_{E_{R(h_j+t_0)-1}}\big)\\ + M \P \big( E_{R(h_j+t_0)-1}^c\big) + \E \big( |\Xi_{R(h_j+t_{0})-1}([0,1])| \ind_{\{|\Xi_{R(h_j+t_{0})-1}([0,1])|>M\}} \big) ~.
\end{multline}Choosing $M=M_j=2n_{R(h_j+t_0)-1}$, we can control the third term as in \eqref{ConclusionEO}. For this choice of $M=M_j$, the second term is upper bounded by $2n_{R(h_j+t_0)-1} \sum_{k\geq R(h_j+t_0)-1}e^{-n_k}$ which converges to zero when $j\rightarrow +\infty$ by \eqref{eq:newcond}.
Now for the first term in the r.h.s. of \eqref{etape6}, let us prove that
\begin{equation}\label{etape7}
\E \big(\eta^{O}_{\mathcal{T}^{(j)}}(t_0,t;[0 \to a])\ind_{E_{R(h_j+t_0)-1}}\big)\leq \E \big(\eta_{\mathcal{W}^{(j)}}(t_0,t;0 , a)\big),
\end{equation}
in which case, taking the $\limsup$ in \eqref{etape6} when $j\rightarrow +\infty$ gives \eqref{etape-but2b}.\\

To show \eqref{etape7}, we produce a coupling between $\mathcal{W}^{(j)}$ and $\mathcal{T}^{(j)}$ ensuring that on $E_{R(h_j+t_0)-1}$,
\begin{equation}
\label{etape8}
\eta^{O}_{\mathcal{T}^{(j)}}(t_0,t;[0 \to a]) \leq \eta_{\mathcal{W}^{(j)}}(t_0,t;0 , a)\big).
\end{equation}
Consider the paths of $\mathcal{W}^{(j)}$ touching $[0,a]\times \{t_0\}$ and that survive until level $t_0+t$. Let us denote by $K$ the random number of points in $[0,a]\times \{t_0\}$ corresponding to these paths ($K\geq \eta_{\mathcal{W}^{(j)}}(t_0,t;[0 \to a])$) and let us call $0\leq a_1<\dots <a_K\leq a$ the abscissa of these points. Recall that if $Z'=\alpha(Z)\in L_{k+1}$ is the ancestor of $Z\in L_k$, then base the isosceles triangle with apex $Z$ and admitting $Z'$ as other vertex contains only one atom of $\Xi_{k+1}$: $Z'$. This triangle is called the influence triangle of $Z$. Let us denote by $\theta_L$ the left border of the influence region of $\mathcal{W}^{(j)}_{(a_1,t_0)}$ (i.e. the union of the influence triangles of the vertices constituting $\mathcal{W}^{(j)}_{(a_1,t_0)}$). We consider the point process on the cylinder consisting of the atoms of $\Xi$ in the region
$$
\mathcal{R}_j = \{(x,t)\in \R\times [h_j+t_0,+\infty),\ \theta_L(t)\leq x<\theta_L(t)+1 \} ~,
$$
which is a PPP conditioned on $E_{R(h_j+t_0)-1}$. On this PPP, let us construct the corresponding forest $\mathcal{T}^{(j)}$ and compute the r.v. $\eta^{O}_{\mathcal{T}^{(j)}}(t_0,t;[0 \to a])$.

Let $\kappa\in \{1,\dots K\}$ be the (random) index $k$ of the rightmost path $\mathcal{W}^{(j)}_{(a_k,t_0)}$ that does not intersect the border $\theta_L(.)+1$. By construction, all the paths $\mathcal{W}^{(j)}_{(a_k,t_0)}$ for $k\in \{1,\dots \kappa\}$ are unchanged on the cylinder (meaning that the successive ancestors of $(a_k,t_0)$ remain the same as in the plane). Hence, among the paths started at $(a_1,t_0),\ldots,(a_{\kappa},t_0)$, the remaining ones at level $t_0+t$ are the same in the plane and in the cylinder: $\eta^{O}_{\mathcal{T}^{(j)}}(t_0,t;[0 \to a_{\kappa}])$ and $\eta_{\mathcal{W}^{(j)}}(t_0,t ; 0,a_{\kappa})$ are equal. If $\kappa = K$ then \eqref{etape8} is proved. Else, let us consider $\mathcal{W}^{(j)}_{(a_{\kappa+1},t_0)}$ and denote by $Z$ the first vertex of this path (in the plane) such that $Z\in \mathcal{R}_j$ and $Z'=\alpha(Z)\notin \mathcal{R}_j$. In other words, $\mathcal{W}^{(j)}_{(a_{\kappa+1},t_0)}$ intersects $\theta_L+1$ for the first time when going from $Z$ to $Z'$. In the cylinder, the paths $\mathcal{T}^{(j)}_{(a_{\kappa+2},t_0)},\dots,\mathcal{T}^{(j)}_{(a_K,t_0)}$ (if they exist) are all trapped between $\mathcal{T}^{(j)}_{(a_{\kappa+1},t_0)}$ and $\mathcal{T}^{(j)}_{(a_{1},t_0)}$. Two cases may be distinguished:
\begin{itemize}
\item[$\bullet$] If $\mathcal{T}^{(j)}_{(a_{\kappa+1},t_0)}$ coalesces with $\mathcal{T}^{(j)}_{(a_{1},t_0)}$ before time $t_0+t$ then the same holds for $\mathcal{T}^{(j)}_{(a_{\kappa+2},t_0)},\ldots,\mathcal{T}^{(j)}_{(a_{K},t_0)}$. In this case, the contribution of $\mathcal{T}^{(j)}_{(a_{\kappa+1},t_0)},\ldots,\mathcal{T}^{(j)}_{(a_{K},t_0)}$ to $\eta^{O}_{\mathcal{T}^{(j)}}(t_0,t;[0 \to a])$ is null. So,
\begin{eqnarray*}
\eta^{O}_{\mathcal{T}^{(j)}}(t_0,t;[0 \to a]) = \eta^{O}_{\mathcal{T}^{(j)}}(t_0,t;[0 \to a_{\kappa}]) & = & \eta_{\mathcal{W}^{(j)}}(t_0,t ; 0,a_{\kappa}) \\
& \leq & \eta_{\mathcal{W}^{(j)}}(t_0,t ; 0,a) ~.
\end{eqnarray*}
\item[$\bullet$] Since the influence triangle of the vertex $Z$ overlaps the influence region of $\mathcal{T}^{(j)}_{(a_{1},t_0)}$ then all the paths $\mathcal{T}^{(j)}_{(a_{\kappa+2},t_0)},\ldots,\mathcal{T}^{(j)}_{(a_{K},t_0)}$ have to coalesce with $\mathcal{T}^{(j)}_{(a_{\kappa+1},t_0)}$ or $\mathcal{T}^{(j)}_{(a_{1},t_0)}$ before time $t_0+t$. Either $\mathcal{T}^{(j)}_{(a_{\kappa+1},t_0)}$ coalesces with $\mathcal{T}^{(j)}_{(a_{\kappa},t_0)}$ before time $t_0+t$ and then $\eta^{O}_{\mathcal{T}^{(j)}}(t_0,t;[0 \to a])$ is still smaller than $\eta_{\mathcal{W}^{(j)}}(t_0,t ; 0,a)$ (as in the first case). Or, $\mathcal{T}^{(j)}_{(a_{\kappa+1},t_0)}$ does not coalesce with $\mathcal{T}^{(j)}_{(a_{\kappa},t_0)}$ before time $t_0+t$ and $\eta^{O}_{\mathcal{T}^{(j)}}(t_0,t;[0 \to a]) = \eta^{O}_{\mathcal{T}^{(j)}}(t_0,t;[0 \to a_{\kappa}])+1$. This also prevents the planar paths $\mathcal{W}^{(j)}_{(a_{\kappa+2},t_0)},\ldots,\mathcal{W}^{(j)}_{(a_{K},t_0)}$ to coalesce with $\mathcal{W}^{(j)}_{(a_{\kappa+1},t_0)}$. Their contribution to $\eta_{\mathcal{W}^{(j)}}(t_0,t ; 0,a)$ is at least $1$:
$$
\eta^{O}_{\mathcal{T}^{(j)}}(t_0,t;[0 \to a]) = \eta^{O}_{\mathcal{T}^{(j)}}(t_0,t;[0 \to a_{\kappa}]) + 1 \leq \eta_{\mathcal{W}^{(j)}}(t_0,t ; 0,a) ~.
$$
\end{itemize}
This shows \eqref{etape8} and concludes the proof of the Lemma.
\end{proof}

\begin{proof}[Proof of Lemma \ref{lem:Aajm}] Let $j\geq j_0$. On the event $A_{a,j,r}^c$, there exists a point of $\mathcal{W}^{(j)}$ at height $h_{R(h_j+t_0)}-h_j$, say $Z$, which is the ancestor of an element of $\mathcal{W}^{(j)}\cap\big([0, a]\times\{h_{R(h_j+t_0)-1}-h_j\}\big)$ but of none of the points of the grid $Gr(t_0,a,r)$. This occurs only if $Z$ belongs to the segment $[0, a]\times\{h_{R(h_j+t_0)}-h_j\}$ and is surrounded by two other points of $\mathcal{W}^{(j)}$ on the same level which are very close to it. Precisely, $A_{a,j,r}^c$ implies the existence of a segment in $[-1/r, a+1/r]\times\{h_{R(h_j+t_0)}-h_j\}$ with length $2/r$ and containing at least $3$ points of the Poisson point process $\Upsilon_{R(h_j+t_0)}$ of intensity $n_{R(h_j+t_{0})}$. The number of points of $\Upsilon_{R(h_j+t_0)}$ in $[-1/r, a+1/r]\times\{h_{R(h_j+t_0)}-h_j\}$ being a Poisson r.v. with parameter $(a+2/r) n_{R(h_j+t_0)-1}$, we can deduce that the minimum distance between two consecutive points of this PPP possesses a density. Thus, the probability of $A_{a,j,r}^c$ tends to $0$ as $r\to+\infty$.
\end{proof}

\begin{proof}[Proof of Lemma \ref{lem:GridB1O}] This last proof is based on the proof of Lemma 2.7 in \cite{newmanravishankarsun}. Let us denote by $\eta_{r,j}(t_0,t;0,a)$ the number of remaining paths $\gamma_{k}(.)$, $k=0,\ldots,m(a,r)$, in $\mathcal{W}^{(j)}$ at time $t_0+t$, that started from the grid $Gr(t_0,a,r)$. The event $A_{a,j,r}$ has been introduced in order to write:
\begin{equation}
\label{E0etape3}
\eta_{\mathcal{W}^{(j)}}(t_0,t ; 0,a) \ind_{A_{a,j,r}} \leq \eta_{r,j}(t_0,t;0,a) ~.
\end{equation}
Besides, the number of paths counted by $\eta_{r,j}(t_0,t;0,a)$ is upper bounded by the number $m(a,r)+1$ of paths $\gamma_{k}(.)$'s starting from the grid $Gr(t_0,a,r)$ minus the number of pairs $(k,k+1)$ that have coalesced before height $t_{0}+t$, i.e.
\begin{equation}
\label{E0etape4}
\eta_{r,j}(t_0,t;0,a) \leq (m(a,r)+1) - \sum_{k=0}^{m(a,r)-1} \ind_{\gamma_{k}(t_{0}+t)=\gamma_{k+1}(t_{0}+t)} ~.
\end{equation}
Using \eqref{lim_cor}, we have for sufficiently large $j$:
\begin{eqnarray}
\label{E0etape5}
\E\Big( \sum_{k=0}^{m(a,r)-1} \ind_{\gamma_{k}(t)=\gamma_{k+1}(t)} \Big) & = & m(a,r) \big( 1 - \P(\gamma_{0}(t_{0}+t)\not=\gamma_{1}(t_{0}+t)) \big) \nonumber \\
& \geq & m(a,r) - \frac{m(a,r) C}{r \sqrt{h_{R(h_j+t_0+t)}-h_{R(h_j+t_0)-1}}} \nonumber \\
& \geq & m(a,r) - \frac{2 m(a,r) C}{r \sqrt{t}} ~,
\end{eqnarray}
since $h_{R(h_j+t_0+t)}-h_{R(h_j+t_0)-1}$ tends to $t$ as $j\to\infty$. Finally, (\ref{E0etape3}), (\ref{E0etape4}) and (\ref{E0etape5}) lead to the expected result.
\end{proof}

{\begin{rem}[Comparison between navigation in different spaces]
Just above Corollary \ref{cor:rtheer}, we observed that certain navigations on the cylinder can be sent onto navigations in the radial plane, and that both navigations are very similar in nature. In the whole paper, we often used that some structures can be transported from the plane onto the cylinder, and to the radial plane, provided the transport keeps the crucial features of the models considered. \\
With the example adapted from the work of Coletti and Valencia \cite{colettivalencia}, in this Section \ref{sec:RBW}, we illustrate how working on the cylinder allows us to state global convergence results for the radial tree correctly renormalized. However, in some cases such as the radial spanning tree (RST) of Baccelli and Bordenave \cite{baccellibordenave}, the cylindrical forest can appear very complicated so that is can be easier to stick to the original radial problem, showing the limitations of this method.\\
In the RST, a homogeneous PPP is given in the plane. A radial tree with vertex set the points of the PPP and rooted at the origin $O$ is constructed. In this tree, the ancestor of a vertex $x$ is the closest point of the PPP with smaller radius. When we send this tree and PPP in the cylinder, the circle of radius $\rho$ is sent on the slice of height $h(\rho)$. The resulting vertex set on the cylinder is not a homogeneous PPP. Additionally, the neighborhood of a given point becomes complicated in the cylinder, as shown in Fig. \ref{fig:decdo}.
\begin{figure}[htbp]
\centerline{\includegraphics[height=3cm,width=8cm]{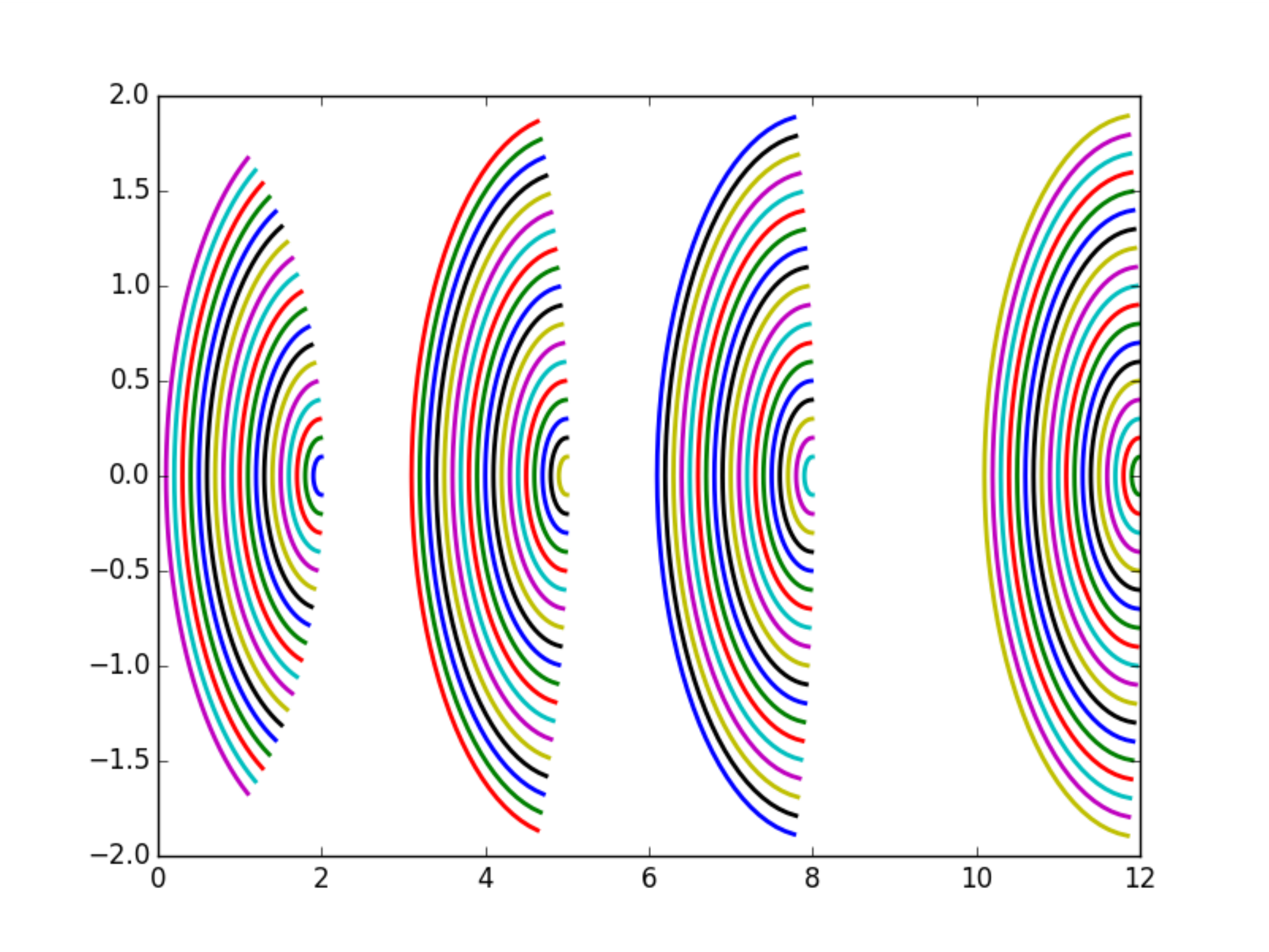}\includegraphics[height=6cm,width=4cm]{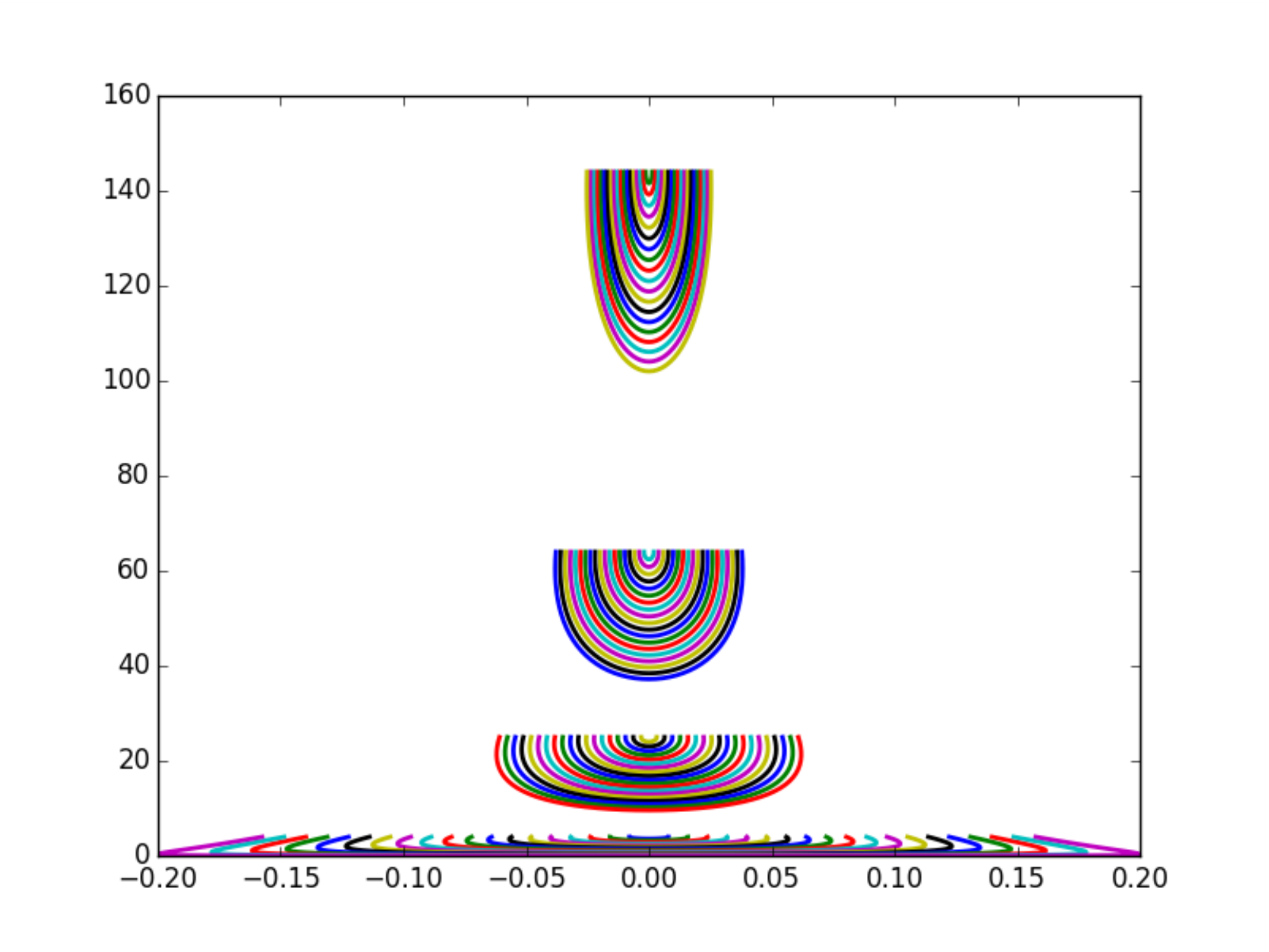}}
\caption{\label{fig:decdo} {\small \textit{Decision domains for Baccelli-Bordenave navigation on the ``radial plane'' according to the distance to the origin, and their images on the cylinder by $\varphi_{h}^{-1}$ where $h(x)=x^2$. One sees that the decision domains are much thinner far from the bottom of the cylinder. This explains somehow the non convergence to the BW, and the existence of asymptotic directions.}}}
\end{figure}\end{rem}
}

{\footnotesize
%\bibliographystyle{abbrv}
%\bibliography{biblio}

}

\end{document}